%% file: NeurIPS_26/0_article.tex
\documentclass{article}

 \usepackage[preprint]{neurips_2026}

\usepackage[utf8]{inputenc} % allow utf-8 input
\usepackage[T1]{fontenc}    % use 8-bit T1 fonts
\usepackage{hyperref}       % hyperlinks
\usepackage{url}            % simple URL typesetting
\usepackage{booktabs}       % professional-quality tables
\usepackage{amsfonts}       % blackboard math symbols
\usepackage{nicefrac}       % compact symbols for 1/2, etc.
\usepackage{microtype}      % microtypography
\usepackage{xcolor}         % colors

\usepackage{algorithm}
\usepackage{algorithmic}

\usepackage{amsmath}
\usepackage{amssymb}
\usepackage{mathtools}
\usepackage{amsthm}

\theoremstyle{plain}
\newtheorem{theorem}{Theorem}[section]
\newtheorem{proposition}[theorem]{Proposition}
\newtheorem{lemma}[theorem]{Lemma}
\newtheorem{remark}[theorem]{Remark}

\theoremstyle{definition}
\newtheorem{definition}[theorem]{Definition}
\newtheorem{assumption}[theorem]{Assumption}
\newtheorem{example}[theorem]{Example}
\theoremstyle{remark}

\input{NeurIPS_26/5_macros}

\input{NeurIPS_26/6_ColorBox}

\title{Avoiding Bias in Clipped SGD for Overparameterized Models under Generalized Smoothness}

\author{%
  Aleksandr Lobanov \\
  MSU AI Center \\
  \texttt{\normalsize lobanovav@my.msu.ru}\\
  \And 
  Anastasia Koloskova \\
  University of Zurich \\
  \texttt{\normalsize anastasiia.koloskova@uzh.ch}\\
}

\begin{document}

\maketitle

\input{NeurIPS_26/1_abstract}
\input{NeurIPS_26/2_main}

\bibliographystyle{plainnat-fixed}
{\small
\bibliography{3_references.bib}
}

\input{NeurIPS_26/4_appendix}

% \input{NeurIPS_26/For_NSGD}

%%%%%%%%%%%%%%%%%%%%%%%%%%%%%%%%%%%%%%%%%%%%%%%%%%%%%%%%%%%%

% \newpage
% \input{NeurIPS_26/7_Checklist}

\end{document}

%% file: NeurIPS_26/5_macros.tex
\usepackage{amsfonts}
\usepackage{dsfont}
\usepackage{wrapfig}
\usepackage[symbol]{footmisc}

\usepackage{tocloft}
\usepackage{bbm}

\allowdisplaybreaks

\usepackage{mdframed}

\usepackage{xcolor}
\usepackage{colortbl}
\usepackage{color}
\usepackage{graphicx}
\usepackage{multirow}
\usepackage{pifont}

\usepackage[font=small]{caption}
\usepackage[font=small]{subcaption}

\usepackage[algo2e]{algorithm2e} 
\usepackage{verbatim}
\usepackage{xspace} %

\usepackage{enumerate}
\usepackage{enumitem}

\newcommand{\dotprod}[2]{\left\langle #1,#2 \right\rangle}
\newcommand{\Obound}[1]{\mathcal{O}\left( #1 \right)}
\newcommand{\OboundTilde}[1]{\tilde{\mathcal{O}}\left( #1 \right)}

\newcommand{\norms}[1]{\left\| #1 \right\| }
\newcommand{\expect}[1]{\mathbb{E}\left[ #1 \right]}
\newcommand{\expectcond}[2]{\mathbb{E}\left[ #1 \mid #2 \right]}

\newcommand{\clip}[1]{ \text{clip}_c\left( #1 \right)}

\newcommand{\func}[1]{f\left( #1 \right)}
\newcommand{\Gf}[1]{G_\lambda\left( #1 \right)}

\definecolor{PineGreen}{HTML}{008B72}

  \DeclareMathOperator*{\argmin}{arg\,min}

  \renewcommand{\gg}{\mathbf{g}}

  \let\lll\ll
  \renewcommand{\ll}{\mathbf{l}}

  \providecommand{\cH}{\mathcal{H}}

  \providecommand{\cL}{\mathcal{L}}

  \providecommand{\cT}{\mathcal{T}}

  \usepackage{bm}
  \newcommand{\bxi}{\boldsymbol{\xi}}

\usepackage[textsize=tiny]{todonotes}

\providecommand{\mycomment}[3]{\todo[caption={},size=footnotesize,color=#1!20, inline]{\textbf{#2: }#3}}%
\providecommand{\inlinecomment}[3]{%
  {\color{#1}#2: #3}}%
\newcommand\commenter[2]%
{%
  \expandafter\newcommand\csname i#1\endcsname[1]{\inlinecomment{#2}{#1}{##1}}
  \expandafter\newcommand\csname #1\endcsname[1]{\mycomment{#2}{#1}{##1}}
}

\newcommand{\circledOne}{\text{\ding{172}}}
\newcommand{\circledTwo}{\text{\ding{173}}}
\newcommand{\circledThree}{\text{\ding{174}}}
\newcommand{\circledFour}{\text{\ding{175}}}
\newcommand{\circledFive}{\text{\ding{176}}}
\newcommand{\circledSix}{\text{\ding{177}}}
\newcommand{\circledSeven}{\text{\ding{178}}}

\commenter{AL}{blue}

\usepackage{url}

%% file: NeurIPS_26/6_ColorBox.tex
\usepackage[many]{tcolorbox}    	% for COLORED BOXES (tikz and xcolor included)

\definecolor{main}{HTML}{5989cf}    % setting main color to be used
\definecolor{sub}{HTML}{cde4ff}     % setting sub color to be used
\definecolor{thmback}{HTML}{F7F3FF} % very light lavender
\definecolor{thmframe}{HTML}{6E5A9A} % muted purple

\definecolor{corback}{HTML}{F3FAF6} % very light green
\definecolor{corframe}{HTML}{5A8F6B} % muted green
\definecolor{lemback}{HTML}{F7F7F7} % very light gray
\definecolor{lemframe}{HTML}{8A8A8A} % muted gray
\colorlet{sub-desaturated}{sub!40!white}
\colorlet{sub-red}{red!5!white}

\tcbset{
    colback = white,
    before skip = 0.2cm,    % add extra space before the box
    after skip = 0.2cm      % add extra space after the box
}                           % setting global options for tcolorbox

\newtcolorbox{boxA}{
    colback = thmback, 
    colframe = thmframe,
    fontupper = \bf,
    boxrule = 1.5pt,
    top = 0pt,           % no top padding
    bottom = 0pt,        % no bottom 
    left = 0pt,    % no left margin/indent
    right = 0pt
}

\newtcolorbox{boxB}{
    fontupper = \bf\color{black}, % font color
    boxrule = 0.5pt,
    colframe = black,
    top = 0pt,           % no top padding
    bottom = 0pt,
    left = 0pt,    % no left margin/indent
    right = 0pt,
    rounded corners,
    arc = 5pt   % corners roundness
}

\newtcolorbox{boxC}{
    colback = sub-desaturated, % background color
    boxrule = 0pt,  % no borders
    top = 0pt,           % no top padding
    bottom = 0pt,        % no bottom 
    left = 0pt,    % no left margin/indent
    right = 0pt    % no right margin/indent
}

\newtcolorbox{boxRED}{
    colback = sub-red, % background color
    boxrule = 0pt,  % no borders
    top = 0pt,           % no top padding
    bottom = 0pt,
    left = 0pt,    % no left margin/indent
    right = 0pt    % no right margin/indent
}

\newtcolorbox{boxD}{
    colback = corback, 
    colframe = corframe, 
    boxrule = 0pt, 
    toprule = 2pt, % top rule weight
    bottomrule = 2pt, % bottom rule weight
    top = 0pt,           % no top padding
    bottom = 0pt,
    left = 0pt,    % no left margin/indent
    right = 0pt
}

\newtcolorbox{boxE}{
    enhanced, % for a fancier setting,
    boxrule = 0pt, % clearing the default rule
    borderline = {0.75pt}{0pt}{main}, % outer line
    borderline = {0.75pt}{2pt}{sub} % inner line
}

\newtcolorbox{boxF}{
    colback = white,
    enhanced,
    boxrule = 1.5pt, 
    colframe = white, % making the base for dash line
    borderline = {1.5pt}{0pt}{black, dashed}, % add "dashed" for dashed line,  % no borders
    top = 0pt,           % no top padding
    bottom = 0pt,
    left = 0pt,    % no left margin/indent
    right = 0pt    % no right margin/indent
}

\newtcolorbox{boxG}{
    enhanced,
    boxrule = 0pt,
    colback = sub,
    borderline west = {1pt}{0pt}{main}, 
    borderline west = {0.75pt}{2pt}{main}, 
    borderline east = {1pt}{0pt}{main}, 
    borderline east = {0.75pt}{2pt}{main}
}

\newtcolorbox{boxH}{
    colback = white, 
    colframe = black, 
    boxrule = 0pt, 
    leftrule = 6pt % left rule weight
}

\newtcolorbox{boxI}{
    colback = sub, 
    colframe = main, 
    boxrule = 0pt, 
    toprule = 6pt % top rule weight
}

\newtcolorbox{boxJ}{
    sharpish corners, % better drop shadow
    colback = sub, 
    colframe = main, 
    boxrule = 0pt, 
    toprule = 4.5pt, % top rule weight
    enhanced,
    fuzzy shadow = {0pt}{-2pt}{-0.5pt}{0.5pt}{black!35} % {xshift}{yshift}{offset}{step}{options} 
}

\newtcolorbox{boxK}{
    sharpish corners, % better drop shadow
    boxrule = 0pt,
    toprule = 4.5pt, % top rule weight
    enhanced,
    fuzzy shadow = {0pt}{-2pt}{-0.5pt}{0.5pt}{black!35} % {xshift}{yshift}{offset}{step}{options} 
}

\newtcolorbox{boxL}{
    fontupper = \color{main},
    rounded corners,
    arc = 6pt,
    colback = sub, 
    colframe = main!50, 
    boxrule = 0pt, 
    bottomrule = 4.5pt 
}

\newtcolorbox{boxM}{
    rounded corners,
    arc = 6pt,
    colback = blue!5, 
    colframe = black, 
    boxrule = 0pt, 
    bottomrule = 1.5pt,
    toprule = 1.5pt,
    enhanced,
    fuzzy shadow = {0pt}{-3pt}{-0.5pt}{0.5pt}{black!35}
}

\newtcolorbox{boxN}{
    colback = lemback,
    colframe = lemframe,
    boxrule = 1pt,
    top = 0pt,
    bottom = 0pt,
    left = 0pt,
    right = 0pt
}

%% file: NeurIPS_26/1_abstract.tex
\begin{abstract}
    Modern machine learning is dominated by complex, overparameterized architectures capable of interpolating data and achieving zero training loss. For such models, we investigate the convergence properties of two popular modifications to standard SGD: clipped SGD and normalized SGD. We show that under overparameterization and a mild assumption on batch size, both clipped and normalized SGD do not suffer from the bias typically introduced by clipping, converging effectively at the same rate as their deterministic counterparts. This provides a rigorous theoretical justification for the empirical success of gradient clipping methods. In our analysis, we employ the $(L_0,L_1)$-smoothness condition, under which we obtain convergence rates that improve upon the best known results in prior work. Furthermore, we extend our analysis to specific challenging regimes, including  heavy-tailed noise,  $(H_0,H_1)$-smoothness (which is strictly weaker than standard assumptions in optimization literature) and the deterministic regime.

\end{abstract}

%% file: NeurIPS_26/2_main.tex
\section{Introduction}\label{sec:Introduction}
    Modern machine learning (ML) models, particularly deep learning (DL) architectures, are typically heavily overparameterized, often containing more parameters than available training samples. This characteristics allows models to interpolate the training data and achieve near-zero training loss~\citep{Ma_2018}. The training of these models is commonly performed using iterative first-order stochastic optimization methods, most notably variants of stochastic gradient descent (SGD) \citep{Robbins_1951,Nemirovski_1983} and adaptive counterparts, e.g., Adam \citep{Kingma_2014}.

    To ensure training stability, gradient clipping \citep{Mikolov_2012} has become a widely used algorithmic modification in practical optimization. Beyond stability, it serves as an essential component of privacy-preserving training algorithms, such as DP-SGD \citep{Abadi_2016}. However, despite its ubiquity, the theoretical properties of gradient clipping are surprisingly poorly understood. Standard theoretical analyses typically predict that clipped SGD (ClipSGD) \citep{Koloskova_2023} introduces a systematic bias, preventing convergence to the true optimum and instead stagnating at an error floor determined by the stochastic noise of the gradients and clipping threshold.

    This theoretical prediction contradicts the empirical reality where overparameterized models routinely converge to zero loss. Furthermore, classical analyses often rely on the simplifying assumption of global $L$-smoothness \citep{Lan_2020,Bach_2024}.
    In practice, modern optimization landscapes are often non-smooth or exhibit rapidly growing gradients that violate standard Lipschitz continuity conditions. Some methods designed under $L$-smoothness assumption and exhibiting strong theoretical guarantees \citep[see, for example, variance-reduced methods for finite-sum problems in][]{Johnson_2013,Defazio_2014,Reddi_2016,Schmidt_2017} perform significantly worse~in~practice~in DL \citep{Defazio_2019}. To bridge this gap between theory and practice,~recent~works proposed relaxed smoothness assumptions, including $(L_0, L_1)$-smoothness~\citep{Zhang_2019_Why_gradient,Zhang_2020_Improved} and $(H_0, H_1)$-smoothness \citep{Vaswani_2025,Liu_2025,Alimisis_2025}.
    
    In this work, we provide a rigorous theoretical justification for the success of gradient clipping in the interpolation regime. We analyze two clipping algorithms: ClipSGD \citep{Mikolov_2012} and Normalized SGD (NSGD) \citep{Nesterov_1984} under the Strong Growth Condition (SGC) \citep{Vaswani_2019_Fast_and_faster}, interpreting SGC as a natural property of the overparameterization regime where the stochastic gradient norm vanishes at the optimum. We demonstrate that under this condition and mild condition on the batch size, the bias typically associated with clipping disappears, allowing these algorithms to converge with speed comparable to their deterministic counterparts. 

    More specifically, our contributions are summarized as follows:
    \vspace{-0.5em}
    \begin{itemize}

        \item (\textbf{Section~\ref{sec:Stochastic Setup via Constant Step Size}}). We provide the first convergence results for ClipSGD and NSGD under the strong growth condition for both non-convex and convex, $(L_0,L_1)$-smooth objectives. Our analysis is the first to theoretically show that, under the mild condition on the batch size, these algorithms do not suffer from the bias typically introduced by clipping and converge with nearly the same efficiency as their deterministic counterparts, and~outperform~standard~SGD.

        \item (\textbf{Section~\ref{sec:Stochastic Setup via Constant Step Size}}). We extend our analysis of NSGD to settings where stochastic gradients follow heavy-tailed distributions. We provide the first result showing that, even under heavy-tailed noise (combined with the strong growth condition), NSGD maintains its efficiency, converging similarly to its deterministic counterpart.

        \item (\textbf{Section~\ref{sec:Numerical Experiments}}). We provide practical guidelines for batch size selection in overparameterized ResNet-18 models, validated on the CIFAR-10 dataset, thereby connecting our theoretical batch-size condition to practical training choices.

        \item (\textbf{Appendix~\ref{app:Extending the Analysis to H_0 H_1 Smoothness}}). In the stochastic setting, we extend our analysis of clipped SGD and normalized SGD to convex, $(H_0,H_1)$-smooth objectives \citep{Vaswani_2025,Liu_2025}. In the deterministic setup, we analyze Gradient Descent (GD) with a warm-up step size strategy (viewed as a clipping-based algorithm). Under the $(H_0,H_1)$-smoothness condition, our analysis improves upon the best-known convergence guarantees for non-convex, convex, and strongly convex objectives. We further improve~these~rates~under the Polyak-Łojasiewicz (PL) condition \citep{Polyak_1963,Karimi_2016,Yue_2023}.  
    \end{itemize}

\vspace{-0.5em}
\section{Problem Setup}\label{sec:Problem Setup and New Assumption}
\vspace{-0.5em}
In this paper we consider optimization problems of the form:
\begin{equation}
    \label{eq:init_problem}
    f^* := \min_{x \in \mathbb{R}^d} \left\{ f(x) := \mathbb{E}_{\xi \sim \mathcal{D}} \left[\func{x,\xi} \right] \right\},
    \vspace{-0.4em}
\end{equation}
where $f$ is a possibly non-convex, and possibly stochastic function. This problem formulation is general and encompasses a wide range of applications, including the optimization of deterministic functions when $\forall \xi:  \func{x,\xi} \equiv \func{x}$. Moreover, by taking $\mathcal{D}$ to be the discrete uniform distribution over the interval $[1,n]$, problem \eqref{eq:init_problem} can be reformulated as a finite-sum formulation: $f(x) := \frac{1}{n}\sum_{\xi=1}^{n} \left[f_{\xi}(x)\right]$, which also covers many practical scenarios. For example, in the context of supervised machine learning, $f_{\xi}$ represents the loss of the model $x$ on the data point $\xi$. Next, we will narrow the class of problems under consideration by introducing assumptions on the objective function and the gradient oracle. We start with assumptions on  function $f$.
\vspace{-0.8em}
\subsection{Assumptions on the objective function}
\vspace{-0.7em}

Throughout the paper, our analysis is conducted under the $(L_0,L_1)$-smoothness condition. This condition was introduced in \citep{Zhang_2019_Why_gradient,Zhang_2020_Improved} as a natural relaxation of the classical $L$-smoothness assumption \citep[see, e.g.,][]{Nesterov_2013}. In contrast to standard $L$-smoothness, which imposes a uniform global Lipschitz constant on the gradient, $(L_0,L_1)$-smoothness allows the local smoothness parameter to depend on the norm of the gradient $\norms{\nabla \func{x}}$. 
\begin{boxF}
\begin{assumption}\label{ass:L_0_L_1Smooth}
    Let $f$ be $(L_0,L_1)$-smooth, then $\forall x,y \in \mathbb{R}^d$, as long as $\norms{y - x} \leq \frac{1}{L_1}$, we have
    \begin{equation*}
        \norms{\nabla \func{y} - \nabla \func{x}} \leq \left( L_0 + L_1 \norms{\nabla \func{x}} \right) \norms{y - x}.
    \end{equation*}
\end{assumption}
\end{boxF}
We use this as the main assumption in our work. The motivation for Assumption~\ref{ass:L_0_L_1Smooth} comes from the empirical findings of \citep{Zhang_2019_Why_gradient}. The authors observed that, in a range of DL problems, the local smoothness constant decreases during training and is closely related to the gradient norm. Moreover, in their experiments, the curvature at the final stages of training could be dramatically smaller than the curvature at initialization; for instance, in LSTM training on the PTB dataset, it was reported to be up to $1000$ times smaller. This indicates that the $L$-smoothness assumption may be too pessimistic for such problems, whereas Assumption~\ref{ass:L_0_L_1Smooth} offers~a~more~adaptive~and~realistic~alternative. 

We note that Assumption~\ref{ass:L_0_L_1Smooth} is strictly weaker than classical $L$-smoothness. Indeed, $(L_0,L_1)$-smoothness reduces to $L$-smoothness when $L_1=0$. Moreover, the inclusion is strict: the exponential of an inner product is $(L_0,L_1)$-smooth but not $L$-smooth \citep[see Example~1.5,][]{Gorbunov_2024}.

% In addition, for several of our results, we provide extensions to the class of $(H_0,H_1)$-smooth functions \cite{Liu_2025,Alimisis_2025}. This condition relates the local smoothness of the objective to the function suboptimality gap $\func{x} - f^*$.

% \begin{boxF}
% \begin{assumption}[Optional]\label{ass:H_0H_1_Smooth}
%     Let $f$ be $(H_0,H_1)$-smooth, then $\forall x,y \in \mathbb{R}^d$ with ${\|y - x\| \leq \frac{1}{\sqrt{H_1}}}$:
%     \begin{equation*}
%         \norms{\nabla \func{y} - \nabla \func{x}} \leq \left( H_0 + H_1 \left[ \func{x} - f^* \right] \right) \norms{y - x}.
%     \end{equation*}
% \end{assumption}
% \end{boxF}
% The Assumption~\ref{ass:H_0H_1_Smooth} is strictly weaker than the Assumption~\ref{ass:L_0_L_1Smooth}, and consequently also strictly weaker than classical $L$-smoothness. To see this, Proposition~B.1 of \cite{Alimisis_2025} establishes that every $(L_0,L_1)$-smooth function is also $(H_0,H_1)$-smooth with $H_0 = L_0 + L_0 L_1$, $H_1 = \frac{4 L_1^2 + L_1}{2}$. Moreover, Example~1 of \cite{Liu_2025} provides a simple function, motivated by deep neural network training, that satisfies $(H_0,H_1)$-smoothness but fails to satisfy $(L_0,L_1)$-smoothness. Thus, $(H_0,H_1)$-smoothness captures a broader class of objectives than the $(L_0,L_1)$-smoothness condition.
\vspace{-0.8em}
\subsection{Assumptions on the gradient oracle}
\vspace{-0.7em}
Throughout this work, we adopt a standard assumption commonly found in the stochastic optimization literature \citep{Nemirovski_2009, Moulines_2011,Juditsky_2011,Roux_2012,Lan_2012,Ghadimi_2012,Ghadimi_2013,Stich_2018,Koloskova_2023,Lobanov_2024_JOTA}: \textit{the unbiasedness} of the gradient oracle. This oracle, when queried at any point $x \in \mathbb{R}^d$, provides access solely to a stochastic gradient~$\nabla \func{x, \xi}$,~satisfying:
\begin{boxF}
\begin{assumption}
    \label{ass:unbiased}
    The gradient oracle is unbiased, i.e.
    % \vspace{-0.5em}
    \begin{equation}
    \label{eq:Assumption_unbiased_gradient_oracle}
        \mathbb{E}_{\xi} \left[\nabla \func{x, \xi}\right] = \nabla \func{x}.
    \end{equation}
\end{assumption}
\end{boxF}
% \vspace{0.5em}

In this work we also focus on \textit{the interpolation regime}, which is typically observed in large ML models (particularly DL architectures), such as overparameterized neural networks, where the number of parameters exceeds the number of training examples, thereby enabling perfect data fitting \citep{Ma_2018, Allen_2019, Gower_2019,Gower_2021,Cooper_2021,Nakkiran_2021}.
\begin{boxC}
    
\begin{definition}\label{def:interpolation} The interpolation condition holds if
\begin{equation*}
    x^* \in \argmin f \Rightarrow  x^* \in \argmin f(\cdot, \xi) \text{~~~~ a.s.}
\end{equation*}
\end{definition}
\end{boxC}
% \vspace{0.5em}

As a restriction on the stochastic noise of gradient oracle, we employ the generalized strong growth condition \citep{Vaswani_2019_Fast_and_faster,Koloskova_2020,Lobanov_2023_PL}:
\begin{boxF}
\begin{assumption}\label{ass:strong growth condition}
    The generalized strong growth condition holds if $\forall x \in \mathbb{R}^d$, $\rho \geq 1$ and $\sigma \geq 0$
    \begin{equation}
        \label{eq:strong growth condition}
        \expect{\norms{\nabla f(x, \xi)}^2} \leq \rho \norms{\nabla f(x)}^2 + \sigma^2.
    \end{equation}
\end{assumption}
\end{boxF}
% \vspace{0.5em}

We introduced Assumption~\ref{ass:strong growth condition} for generality, as with $\rho = 1$ it reduces to the standard bounded variance assumption $\mathbb{E}[\norms{\nabla \func{x,\xi} - \nabla \func{x}}^2] \leq \sigma^2$, which is widely used in the literature \citep[see, e.g.,][]{Karimireddy_2020,Yang_2022}.  However, it can be observed that the general form \eqref{eq:strong growth condition} does not satisfy the interpolation condition. Consequently, the main results of this work focus on perhaps the most common assumption that satisfies Definition~\ref{def:interpolation}: Assumption~\ref{ass:strong growth condition} with $\sigma = 0$~—~known as \textit{the strong growth condition} (SGC) \citep{Cevher_2019,Vaswani_2019_Painless,Vaswani_2019_Fast_and_faster,Hermant_2025}. See Appendix~\ref{app:On the generalized strong growth condition} for a discussion of the generality of form~\eqref{eq:strong growth condition}.

It is worth noting that in some works \citep[e.g., see][]{Vaswani_2019_Painless,Vaswani_2019_Fast_and_faster}, SGC is classified as a \textit{stationary-point interpolation}, which implies that holds: ${\nabla f(x^*) = 0 \Rightarrow \nabla f(x^*,\xi) = 0}$.

In part of this work, we provide a generalization to the heavy-tailed (HT). We assume a generalized $p$-th central moment \citep{Liu_2025_Heavy_Tailed_Noises}, which can be reduced to Assumption~\ref{ass:strong growth condition}~for~HT~noises. 
\begin{boxF}
\begin{assumption}[Optional]
\label{prop:SGC_under_heavy tailed noise}
    We assume that gradient oracle has a generalized $p$-th central moment, i.e. there exists $\rho \geq 1$, $\sigma \geq 0$ such that, for all $x \in \mathbb{R}^d$ and $p \in (1,2)$:
    \begin{equation*}
        % \label{eq:Heavy-tailed SGC}
        \expect{\norms{\nabla f(x, \xi) - \nabla \func{x}}^p} \leq (\rho - 1) \norms{\nabla f(x)}^p + \sigma^p.
    \end{equation*}
\end{assumption}
\end{boxF}
% \vspace{0.5em}

When $\rho = 1$, Assumption~\ref{prop:SGC_under_heavy tailed noise} reduces to the bounded $p$-th central moment condition, which can be considered a standard in the heavy-tailed setting \citep{Zhang_2020_Why_are_adaptive,Liu_2023,Hubler_2025}. Moreover, Assumption~\ref{prop:SGC_under_heavy tailed noise} can be linked to a form of \eqref{eq:strong growth condition} (a $p$-th moment) by applying finite form of Jensen's inequality: $\expect{\norms{\nabla f(x, \xi)}^p} \leq 2^{p-1} \left( \rho \norms{\nabla \func{x}}^p + \sigma^p\right)$.

In this work, the case $\sigma = 0$ in Assumption~\ref{prop:SGC_under_heavy tailed noise}~is~of primary interest. It is straightforward to see that this case satisfies Definition~\ref{def:interpolation} and can also be classified as stationary-point interpolation. 
%We demonstrate in Example~\ref{example2} that it, too, permits an unbounded second moment/variance, meaning the noise may follow a heavy-tailed distribution. 
And this specific form ${\expect{\norms{\nabla \func{x,\xi}}^p} \leq 2^{p-1}  \rho \norms{\nabla \func{x}}^p}$ we term \textit{the SGC under heavy-tailed noise}.
% \vspace{-0.3em}
% \begin{boxC}
% \begin{example}
%     \label{example2}
%     Consider the function $f(x) = \frac{1}{2} \|x\|^2$ with gradient oracle $\nabla \func{x,\xi} = \left(Z + 1\right) x$ (where $Z$ is a random variable with a symmetric Pareto distribution). Then, for $\alpha \in (p,2]$ we obtain the strong growth condition under heavy-tailed noise $\expect{\norms{\nabla f(x, \xi)}^p} \leq 2^{p-1} \rho \norms{\nabla f(x)}^p$ with $\rho = \frac{\alpha}{\alpha - p} + 1 < \infty$,~and~${\expect{\norms{\nabla f(x, \xi)}^2} = \infty}$.
% \end{example}
% \end{boxC}
% % \vspace{0.5em}

% % \vspace{-0.4em}
% The proofs of examples are deferred to Appendix~\ref{app:Motivation_heavy}.
\vspace{-0.8em}
\section{Related works}
\vspace{-0.7em}
    
    This section discusses the most relevant related work.
    % \vspace{-1em}
    \paragraph{Clipped SGD} has been studied extensively across settings. In non-convex deterministic $(L_0,L_1)$-smooth optimization, \cite{Zhang_2019_Why_gradient} first theoretically demonstrated the benefits of clipping. This was later refined by \cite{Koloskova_2023}, who improved bounds for Clipped GD and analyzed Clipped SGD under stochastic $(L_0,L_1)$-smoothness with bounded variance; however, their guarantees converge only to an error floor $\min\{\sigma,\sigma^2/c\}$, even with mini-batching. Another line of work mitigates clipping bias using error feedback and momentum \citep{Islamov_2025}, but for modified clipping schemes rather than vanilla ClipSGD, which is the focus of our analysis. For convex objectives, deterministic analyses identified linear and sublinear regimes \citep{Lobanov_2024_Linear_Convergence_Rate,Vankov_2025}, while stochastic works established high-probability guarantees under light- and heavy-tailed noise by choosing sufficiently large clipping radius \citep{Gaash_2025,Chezhegov_2025}. Other recent works study convex interpolation under $(L_0,L_1)$- and $(H_0,H_1)$-smoothness, but rely on non-standard metrics \citep{Gorbunov_2024,Alimisis_2025}. \textit{In contrast, we analyze ClipSGD with an arbitrary clipping radius in standard metrics (the expected minimum gradient norm and the expected optimality gap over $N$ iterations), proving convergence to arbitrary accuracy under SGC and a mild batch-size $\Obound{\rho}$, with two regimes for both non-convex and convex objectives.}
    \vspace{-1em}
    \paragraph{Normalized SGD} dates back to \cite{Nesterov_1984} and has since been studied under bounded variance (BV) and heavy-tailed noise. In convex $(L_0,L_1)$-smooth BV settings, \cite{Lobanov_2025} showed that NSGD improves the iteration complexity of Clipped SGD, but at the cost of a much worse oracle complexity, $\mathcal{O}(\varepsilon^{-5})$. For non-convex $L$-smooth BV problems, momentum NSGD achieves $\mathcal{O}(\varepsilon^{-4})$ sample complexity \citep{Cutkosky_2020,Zhao_2021}; the same rate was obtained under non-convex $(L_0,L_1)$-smoothness in \citep{Hubler_2024}, although their analysis does not separate $L_0$ and $L_1$. In the heavy-tailed regime, \citep{Hubler_2025,Liu_2025_Heavy_Tailed_Noises} analyzed momentum NSGD for non-convex $L$- and $(L_0,L_1)$-smooth objectives, respectively, obtaining $\mathcal{O}(\varepsilon^{-(3p-2)/(p-1)})$ sample complexity. \textit{In this work, we identify a regime where NSGD improves the iteration complexity of Clipped SGD without degrading oracle complexity, by exploiting the strong growth condition under heavy-tailed noise. We show that with batch size $\mathcal{O}(\rho^{1/(p-1)})$, NSGD matches its deterministic counterpart. While \citep{Liu_2025_Heavy_Tailed_Noises} uses a related Assumption~\ref{prop:SGC_under_heavy tailed noise}, their analysis is restricted to $\rho=1$ and thus does not capture our results for $\rho>1$. Moreover, their method requires access to the full gradient and deterministic function value to set $\beta$ and $\eta$, which is typically impractical for large-scale models. Our analysis~removes~these~limitations~for~both~$\rho>1$~and~$\sigma>0$.}
\vspace{-2em}
\section{Stochastic Setup via Constant Step Size}\label{sec:Stochastic Setup via Constant Step Size}
\vspace{-0.5em}
    %Algorithms
    In this section, we analyze the convergence of two gradient clipping algorithms in a stochastic
    \begin{minipage}{0.4\textwidth}
     setup: Clipped Stochastic Gradient Descent (ClipSGD, which clips the gradient when its norm exceeds a given radius) and Normalized Stochastic Gradient Descent (NSGD, which performs normalization/clipping at every iteration). For brevity, both algorithms are presented in Algorithm~\ref{algo:clip_algo}. Note that in Step 2 of Algorithm~\ref{algo:clip_algo}, $\nabla \func{x^k,\bxi^k}$ denotes the stochastic gradient computed on a mini-batch of size $B \geq 1$. In Step 3, depending on the chosen method, we define the update direction as follows: for ClipSGD, we use the clipping operator ${\gg(x^k, \bxi^k) = \alpha_{\bxi^k} \nabla f(x^k, \bxi^k)}$, where $\alpha_{\bxi^k} = \min \left\{ 1, c/\norms{\nabla \func{x^k, \bxi^k}} \right\}$; for NSGD, we use the normalized gradient 
    \end{minipage}
    \begin{minipage}{0.2\textwidth}
    \end{minipage}
    \begin{minipage}{0.58\textwidth}
    \vspace{-1em}
    \begin{algorithm}[H]
        \caption{Clipping Algorithms in a Stochastic Setup}
        \label{algo:clip_algo}
        \begin{algorithmic}
            \STATE {\bfseries Method:} \textcolor{blue}{ClipSGD} or \textcolor{red}{NSGD}
            \STATE {\bfseries Input:} initial point $x_0 \in \mathbb{R}^d$, iterations number $N$, batch size $B$, step size $\eta > 0$ and parameters $c>0$, $\lambda > 0$   
            \FOR{$k=0$ {\bfseries to} $N-1$}
                \STATE \hspace{-1.4em} {1.} Draw fresh i.i.d. samples $\xi^k_1,...,\xi^k_B$
                \vspace{0.2em}
                \STATE \hspace{-1.4em} {2.} $\nabla f(x^k, \bxi^k) = \frac{1}{B} \sum_{i=1}^B \nabla f(x^k, \xi^k_i)$ 
                \vspace{0.4em}
                \STATE \hspace{-1.31em} 
                {3.} Compute $\gg(x^k, \bxi^k)$ for the selected algorithm.
                \vspace{0.2em}
                \STATE 
                \hspace{-1.31em}
                {3.1.} For \textcolor{blue}{ClipSGD}:\quad \,${\gg(x^k, \bxi^k) = \alpha_{\bxi^k} \nabla f(x^k, \bxi^k)}$ 
                % \vspace{0.1em}
                \STATE \hspace{-1.31em} {3.2.} For \textcolor{red}{NSGD}: \quad\quad $\gg(x^k, \bxi^k) =  \frac{\nabla f(x^k, \bxi^k)}{\|\nabla f(x^k, \bxi^k)\| + \lambda}$
                \vspace{0.2em}
                \STATE \hspace{-1.4em}
                {4.} $x^{k+1} \gets x^{k} - \eta \cdot \gg(x^k, \bxi^k)$
            \ENDFOR
            \STATE {\bfseries Return:} $x^{N}$
        \end{algorithmic}
    \end{algorithm}
    \end{minipage}
    % \vspace{-1em}
    $\gg(x^k, \bxi^k) = \frac{\nabla f(x^k, \bxi^k)}{\norms{\nabla f(x^k, \bxi^k)} + \lambda}$ with parameter $\lambda >0$. For each methods in Algorithm~\ref{algo:clip_algo}, we consider a constant step size $\eta>0$.
    % \vspace{-0.1em}
    In our theoretical analysis, we show that under the SGC (Assumption~\ref{ass:strong growth condition} with $\sigma = 0$), ClipSGD and NSGD converge to a stationary point for non-convex functions (or to a minimizer for convex functions) with nearly the same efficiency as their deterministic counterparts. 
    
    %We emphasize “nearly” because, in stochastic setup, there is no guarantee of monotonic decrease in~gradient~norm.
    
    %Moreover, we observe that compared to ClipSGD, NSGD can improve the iteration complexity specifically in terms of the $L_1$ components without increasing the batch size. Thus, we identify a regime where NSGD outperforms ClipSGD in both iteration and oracle~complexity~criteria.
\vspace{-1em}
\subsection{Non-convex functions}\label{subsec:Non convex functions}
\vspace{-1em}
        In this part, we focus on non-convex functions that satisfy $(L_0,L_1)$-smoothness (see Assumption~\ref{ass:L_0_L_1Smooth}). We start by presenting our novel convergence result for Clipped SGD.
        \begin{boxA}
        \begin{theorem}[Non-convex, ClipSGD]\label{th:NC_ClipSGD}
            Let $f$ is $(L_0,L_1)$-smooth (Assumption~\ref{ass:L_0_L_1Smooth}) and gradient oracle is unbiased (Assumption~\ref{ass:unbiased}) and satisfies SGC (Assumption~\ref{ass:strong growth condition} with $\sigma = 0$), then ClipSGD (Algorithm~\ref{algo:clip_algo}) with step size $\eta \leq \left[9 \left( L_0 + L_1 c  \right)  \right]^{-1}$, and batch size $B \geq 72 (\rho - 1)$~guarantees:
            \begin{equation*}
                \min_{k \in [0,N-1]} \expect{ \norms{\nabla \func{x^k}}} \leq \sqrt{\frac{18 F_0}{\eta N}}  +\frac{18 F_0}{\eta c N},
            \end{equation*}
            where $N$ is the number of iterations, $F_0 = \func{x^0} - f^*$.
        \end{theorem}       
        \end{boxA}
        The proof of Theorem~\ref{th:NC_ClipSGD} is provided in Appendix~\ref{app:NC_ClipSGD} and corresponds to the case $\sigma = 0$. Note that by selecting the maximum step size $\eta = \left[9 \left( L_0 + L_1 c  \right)  \right]^{-1}$, we obtain the following convergence result: $\min_{k \in [0,N-1]} \expect{ \norms{\nabla \func{x^k}}}$ is upper bounded by $\Obound{\sqrt{\frac{L_0 F_0}{ N}} + \sqrt{\frac{L_1 c F_0}{ N}} + \frac{L_0 F_0}{c N} + \frac{L_1 F_0}{N}}$.

        \vspace{-1.5em}
        \paragraph{Comparison to the prior works.} First, the convergence result presented in Theorem~\ref{th:NC_ClipSGD} matches (when $B = \Obound{\rho}$) the convergence of ClipGD in the non-convex \textit{deterministic setting} under $(L_0, L_1)$-smoothness \citep[see Theorem 2.1,][]{Koloskova_2023}. Second, the first term in Theorem~\ref{th:NC_ClipSGD} coincides with the convergence rate of SGD  (\textit{without clipping}) under the $L$-smoothness, SGC assumptions with step size $\eta \leq (\rho L)^{-1}$ \citep[see Theorem 3,][]{Vaswani_2019_Fast_and_faster}: $\Obound{(\eta N)^{-1/2}}$. Therefore, it is straightforward to observe that Theorem~\ref{th:NC_ClipSGD} improves over SGD in both: Iteration complexity (due to a larger step size---an advantage from using mini-batches, and through the refined assumption of $(L_0,L_1)$-smoothness); Oracle complexity (in scenarios analogous to those in \cite{Zhang_2019_Why_gradient} where $L_0 \lll L$ and $L_1 \lll L$, and potentially due to the addition of the term $\frac{F_0}{\eta c N}$ which depends on the clipping radius $c$; note, however, that the term $\Obound{(\eta N)^{-1/2}}$ remains dominant).
        \begin{boxD}
        \begin{proposition}[Generalized SGC]\label{cor:NC_ClipSGD}
            If Assumption~\ref{ass:strong growth condition} holds for $\sigma \geq 0$, then Theorem~\ref{th:NC_ClipSGD} takes the following form with $\eta \leq \left[9 \left( L_0 + L_1 c  \right)  \right]^{-1}$ and $B \geq 72 (\rho - 1)$ guarantees an error:
            \begin{equation*}
                \min_{k \in [0,N-1]} \expect{ \norms{\nabla \func{x^k}}} \lesssim \sqrt{\frac{F_0}{\eta N}}  +\frac{F_0}{\eta c N} + \frac{\sigma}{\sqrt{B}} +  \frac{\sigma^2}{cB}.
            \end{equation*}
        \end{proposition}
        \end{boxD}
        It is not hard to see that with a sufficiently large batch size, Proposition~\ref{cor:NC_ClipSGD} guarantees convergence to the desired accuracy $\varepsilon: \min_{k \in [0,N-1]} \expect{ \norms{\nabla \func{x^k}}} \leq \varepsilon$, thereby improving upon the prior result \citep[see Appendix C.4.1,][]{Koloskova_2023}, which only ensured convergence to an error floor (asymptote) independent of the batch size $B$.

        Next,~we~present~the~results~for~Normalized~SGD.
        \begin{boxA}
        \begin{theorem}[Non-convex, NSGD]\label{th:NC_NSGD}
            Let $f$ is $(L_0,L_1)$-smooth (Assumption~\ref{ass:L_0_L_1Smooth}) and gradient oracle is unbiased (Assumption~\ref{ass:unbiased}) and satisfies SGC (Assumption~\ref{ass:strong growth condition} with $\sigma = 0$), then NSGD (Algorithm~\ref{algo:clip_algo}) with any $\lambda > 0$, step size $\eta \leq \frac{\lambda}{L_0 + L_1 \lambda}$, batch size $B \geq 64 (\rho - 1)$ guarantees:
            \begin{equation*}
            \min_{k\in [0, N-1]}\expect{\norms{\nabla \func{x^k}}} \leq \frac{4F_0}{\eta N} + 6 \lambda,
        \end{equation*}
            where $N$ is the number of iterations, $F_0 = \func{x^0} - f^*$.
        \end{theorem}
        \end{boxA}
        The proof of Theorem~\ref{th:NC_NSGD} is provided in Appendix~\ref{app:NC_NSGD} and corresponds to the case $\sigma = 0$. 
        \vspace{-1em}
        \paragraph{Comparison to ClipSGD (Theorem~\ref{th:NC_ClipSGD}).} From Theorem~\ref{th:NC_NSGD}, we observe that Normalized SGD converges to hyperparameter $\lambda >0$; however, the smaller $\lambda$ is, the more iterations are required, as $\lambda$ directly influences the step size. Note that for NSGD to reach the desired accuracy $\varepsilon: \min_{k \in [0,N-1]} \expect{ \norms{\nabla \func{x^k}}} \leq \varepsilon$, it is necessary to take $\lambda \leq \frac{\varepsilon}{12}$. Then, by setting the maximum allowable hyperparameter $\lambda = \frac{\varepsilon}{12}$ and the maximum step size $\eta = \frac{\varepsilon}{ 12 L_0 + L_1 \varepsilon}$, we obtain the iteration complexity: $N = \Obound{\frac{L_0 F_0}{\varepsilon^2} + \frac{L_1 F_0}{\varepsilon}}$, which coincides with ClipSGD (see the discussion following Theorem~\ref{th:NC_ClipSGD} with $c= \lambda \simeq \varepsilon$). Thus, Normalized SGD asymptotically behaves like Clipped SGD with a small clipping radius $c \simeq \varepsilon$, demonstrating improved coefficients. However, if the clipping radius is sufficiently large $c \geq \varepsilon$, then NSGD improves the iteration complexity compared to ClipSGD (specifically the term involving $L_1$) without increasing the oracle call numbers. In other words, in this regime, Normalized SGD outperforms ClipSGD in iteration and oracle complexity. 
        \vspace{-1em}
        \paragraph{Comparison to the prior works.} First, when $L_1 = 0$, Theorem~\ref{th:NC_NSGD} achieves oracle complexity matching \textit{the deterministic counterpart with momentum} under $L$-smoothness \citep[see Theorem 1,][with $\sigma = 0$]{Cutkosky_2020}, up to an $\Obound{\rho}$ factor (since we employ mini-batches), i.e., $T = \Obound{\varepsilon^{-2}}$. We also note that due to the relaxed $(L_0, L_1)$-smoothness assumption (in particular, because $L_0 \lll L$ and $L_1 \lll L$, and includes an additive $\Obound{L_1 \varepsilon^{-1}}$ term) the iteration complexity in Theorem~\ref{th:NC_NSGD} improves upon that of \citep{Cutkosky_2020}. Second, since our analysis relies on the SGC, the complexity $T = \Obound{\rho L_0 \varepsilon^{-2} + \rho L_1 \varepsilon^{-1}}$  established in Theorem~\ref{th:NC_NSGD} is strictly better than the oracle complexity $T = N \cdot B = \Obound{L \varepsilon^{-4}}$ (with the batch size $B = \Obound{\varepsilon^{-2}}$) of \textit{NSGD under $L$-smoothness} \citep[see Theorem 1,][]{Zhao_2021}. Third, by employing mini-batches and leveraging the SGC, Theorem~\ref{th:NC_NSGD} improves both the iteration and oracle complexity of \textit{NSGD with momentum under $(L_0,L_1)$-smoothness} \citep[see Theorem 2,][]{Hubler_2024}. Unlike \citep{Hubler_2024}, our analysis separates the influence of $L_0$, $L_1$, which yields benefits~in~regimes~${L_0 \lll L_1}$.
        \begin{boxD}
        \begin{proposition}[Heavy-tailed]\label{cor:NC_NSGD_heavy_tailed}
            Let $f$ is $(L_0,L_1)$-smooth (Assumption~\ref{ass:L_0_L_1Smooth}) and gradient oracle is unbiased (Assumption~\ref{ass:unbiased}) and satisfies SGC under heavy-tailed noise (Assumption~\ref{prop:SGC_under_heavy tailed noise} with $\sigma = 0$), then NSGD with any $\lambda > 0$, $\eta \leq \frac{\lambda}{L_0 + L_1 \lambda}$ and $B \geq 2^{\frac{3p +1}{p-1}} (\rho-1)^{\frac{1}{p-1}}$ guarantees an error:
            \vspace{-0.5em}
            \begin{equation*}
            \vspace{-0.5em}
            \min_{k\in [0, N-1]}\expect{\norms{\nabla \func{x^k}}} \leq \frac{4F_0}{\eta N} + 6 \lambda,
        \end{equation*}
            where  $N$ is the number of iterations, $F_0 = \func{x^0} - f^*$.
        \end{proposition}
        \end{boxD}
        Proposition~\ref{cor:NC_NSGD_heavy_tailed} shows that the strong growth condition (SGC) produces a similar effect in the heavy-tailed setup. Specifically, Proposition~\ref{cor:NC_NSGD_heavy_tailed} establishes an oracle complexity for $\lambda = \varepsilon/12$~and~maximum~step~size: $\Obound{\rho^{\frac{1}{p-1}} L_0 F_0 \varepsilon^{-2} + \rho^{\frac{1}{p-1}} L_1 F_0 \varepsilon^{-1}}$. In comparison, \citep[see Corollary 3,][]{Hubler_2025} derived a complexity for normalized stochastic gradient descent (NSGD) under $L$-smoothness $\Obound{\frac{L F_0}{\varepsilon^2} \left( \frac{\sigma}{\varepsilon} \right)^{\frac{p}{p-1}}}$, which requires a substantially larger number of oracle calls.

        In Appendices~\ref{app:NC_NSGD} and \ref{app:NSGD_heavy_tailed_noise}, we provide proofs for the cases satisfying the generalized SGC (Assumption~\ref{ass:strong growth condition}) and the generalized $p$-th central moment condition (Assumption~\ref{prop:SGC_under_heavy tailed noise}), respectively.
\vspace{-1em}
\subsection{Convex functions}\label{subsec:Convex functions}
\vspace{-0.5em}

        In this subsection, we consider the class of convex and smoothness functions $f(\cdot,\xi)$.
        \begin{boxF}
        \begin{assumption}\label{ass:Convexity_Smoothness_L0_L1}
            Let $f(\cdot,\xi)$ be $(\cL_0, \cL_1)$-smooth and convex for almost every $ \xi\in \mathcal{D}$. That is, $\forall x,y,z \in \mathbb{R}^d$ such that $\norms{y-z} \leq \frac{1}{\cL_1}$, we have with $\cL_{\xi}(z) := \cL_0 + \cL_1 \norms{\nabla \func{z,\xi}}$:
            \vspace{-0.5em}
            \begin{equation*}
                 \func{x,\xi} + \dotprod{\nabla \func{x,\xi}}{y-x} \leq \func{y,\xi}   \leq \func{z,\xi} + \dotprod{\nabla \func{z,\xi}}{y-z} + \frac{\cL_{\xi}(z)}{2} \norms{y - z}^2.
            \end{equation*}
        \end{assumption}
        \end{boxF}
        We note that Assumption~\ref{ass:Convexity_Smoothness_L0_L1} is widely used in the literature under interpolation conditions \citep{Woodworth_2021, Gorbunov_2024,Alimisis_2025}. Assumption~\ref{ass:Convexity_Smoothness_L0_L1} can be seen to reduce to Assumption~\ref{ass:L_0_L_1Smooth} under Assumption~\ref{ass:strong growth condition} (SGC) with $\sigma = 0$, $L_0 = \cL_0$ and $L_1 = \sqrt{\rho}\cL_1$ (see Appendix~\ref{subsec:generalized_smoothness_conditions} for details). Below, we present the convergence results for ClipSGD.
        \begin{boxA}
        \begin{theorem}[Convex, ClipSGD]\label{th:Convex_ClipSGD}
            Let $f(x,\xi)$ be $(\cL_0,\cL_1)$-smooth and convex (Assumption~\ref{ass:Convexity_Smoothness_L0_L1}), gradient oracle is unbiased (Assumption~\ref{ass:unbiased}) and satisfies SGC (Assumption~\ref{ass:strong growth condition} with $\sigma = 0$). Then ClipSGD (Algorithm~\ref{algo:clip_algo}) with step size $\eta \leq [9 ( \cL_0 + \sqrt{\rho} \cL_1 c ) ]^{-1}$~and~batch~size~${B \geq 36 (\rho - 1)}$~yields:
            \vspace{-0.5em}
            \begin{equation*}
                \expect{\func{x^{N}}} - f^* \leq \max\left\{\left(1 - \frac{\eta c}{36 R_0}\right)^{N/2} F_0 , \frac{18 R_0^2}{N \eta}\right\},
            \end{equation*}
            % where $\Tilde{R} = \max_{k \in [0,N-1]} \norms{x^k - x^*}$, $F_0 = \func{x^0} - f^*$.
            where  $N$ is the number of iterations, $F_0 = \func{x^0} - f^*$ and $R_0 = \norms{x^0 - x^*}$.
        \end{theorem}
        \end{boxA}
        The proof of Theorem~\ref{th:Convex_ClipSGD} is provided in Appendix~\ref{app:Convex_ClipSGD}.  It is straightforward to see that if the step size is chosen as $\eta = [9 ( \cL_0 + \sqrt{\rho} \cL_1 c ) ]^{-1}$, the iteration complexity $N$ takes the following form: $\max \left\{\Obound{\frac{[\cL_0 + \sqrt{\rho} \cL_1 c ]R_0 }{c} \log \left( \frac{1}{\varepsilon} \right)}, \Obound{\frac{(\cL_0 + \sqrt{\rho} \cL_1 c) R_0^2}{\varepsilon}} \right\}$, where $R_0 = \norms{x^0 - x^*}$.
        % \vspace{-1em}
        \paragraph{Comparison to the prior works.} First, Theorem~\ref{th:Convex_ClipSGD} reveals two convergence regimes, closely matching the convergence rate of ClipGD in the convex \textit{deterministic setting} \citep[see Theorem 3.1,][]{Lobanov_2024_Linear_Convergence_Rate}. The key distinction is that due to the monotonicity of the gradient norm in the deterministic case, ClipGD initially exhibits a linear rate, which slows down to a sublinear rate after the threshold $c$ is reached. Second, Theorem~\ref{th:Convex_ClipSGD} holds for an arbitrary clipping radius $c > 0$. In contrast, \citep{Chezhegov_2025,Gaash_2025} require a specific, sufficiently large clipping radius $c$ to establish convergence guarantees for \textit{ClipSGD in the convex $(L_0, L_1)$-smooth setting}. This choice of clipping radius effectively ignores the term corresponding to the exponential decrease of the function (although we note that the sublinear term ultimately dominates). Third, the prior best known convergence rate of ClipSGD for convex $(L_0, L_1)$-smooth functions with \textit{an arbitrary $c$} \citep[see Theorem 3.1,][]{Lobanov_2025} requires strong assumptions: a uniform bounded gradient norm. We successfully addressed these issues through an improved analysis and the application of SGC. Fourth, Theorem~\ref{th:Convex_ClipSGD} provides guarantees in terms of the standard metric $\expect{f(x^N)} - f^*$, whereas the analysis in \citep{Gorbunov_2024} relies on a non-standard metric.

        We now present the convergence results for NSGD.
        \begin{boxA}
        \begin{theorem}[Convex, NSGD]\label{th:Convex_NSGD}
            Let $f(x,\xi)$ be $(\cL_0,\cL_1)$-smooth and convex (Assumption~\ref{ass:Convexity_Smoothness_L0_L1}), gradient oracle is unbiased (Assumption~\ref{ass:unbiased}) and satisfies SGC (Assumption~\ref{ass:strong growth condition} with $\sigma = 0$). Then NSGD (Algorithm~\ref{algo:clip_algo}) with $\lambda > 0$, step size $\eta \leq \frac{\lambda}{\cL_0 + \sqrt{\rho} \cL_1 \lambda}$ and batch size $B \geq 64 (\rho - 1)$ yields:
            \vspace{-0.5em}
            \begin{equation*}
                \expect{\func{x^{N}}} - f^* \leq \left( 1 - \frac{\eta}{4 R_0} \right)^N F_0  + 6 \lambda R_0,
            \end{equation*}
            where  $N$ is the number of iterations, $F_0 = \func{x^0} - f^*$ and $R_0 = \norms{x^0 - x^*}$.
        \end{theorem}
        \end{boxA}

        The proof of Theorem~\ref{th:Convex_NSGD} is provided in Appendix~\ref{app:Convex_NSGD}. Similar to the non-convex case, Theorem~\ref{th:Convex_NSGD} shows that NSGD converges to an error floor of $6 \lambda R_0$. To ensure convergence to a desired accuracy $\varepsilon$, i.e., $\mathbb{E}[f(x^N)] - f^* \leq \varepsilon$, the hyperparameter must satisfy $\lambda \leq \frac{\varepsilon}{12 R_0}$. Setting $\lambda$ to its maximum admissible value $\lambda = \frac{\varepsilon}{12 R_0}$ and the step size to $\eta = \frac{\varepsilon}{12 \cL_0 R_0 + \sqrt{\rho}\cL_1 \varepsilon}$, NSGD achieves the following iteration complexity:
        $
            N = \Obound{\left( \cL_0 R_0^2/\varepsilon +  \sqrt{\rho}\cL_1 R_0 \right) \log \frac{F_0}{\varepsilon}}
        $ with batch size $B = \Obound{\rho}$. 
        \vspace{-1em}
        \paragraph{Comparison to ClipSGD (Theorem~\ref{th:Convex_ClipSGD}).} Similar to the non-convex setting, the convergence behavior of NSGD mirrors that of ClipSGD with $c \simeq \varepsilon / R_0$. The distinction lies in the fact that for ClipSGD, the term involving $\cL_1$ differ between its two convergence regimes, yielding a bound proportional to $\max\left\{ \mathcal{O}\big(\sqrt{\rho}\cL_1 R_0 \log(1/\varepsilon)\big), \mathcal{O}\big(\sqrt{\rho}\cL_1 R_0 \big) \right\}$. Specifically, for small $\varepsilon$, the $\mathcal{O}\big(\sqrt{\rho}\cL_1 R_0  \log(1/\varepsilon)\big)$ term dominates in ClipSGD, which aligns with the rate of NSGD. However, for sufficiently large $\varepsilon$, the $\mathcal{O}\big(\sqrt{\rho}\cL_1 R_0 \big)$ term becomes dominant for ClipSGD, a rate that can potentially be worse than that of NSGD. We also note that for a large clipping radius $c \gtrsim \varepsilon / R_0$, NSGD is more efficient than ClipSGD in terms of both iteration $N$ and oracle $T = N \cdot B$ complexity.
        \vspace{-2em}
        \paragraph{Comparison to the prior works.} First, unlike ClipSGD, Theorem~\ref{th:Convex_NSGD} shows that the NSGD achieves the same efficiency as its \textit{deterministic counterpart} \citep[see Theorem 3.3,][]{Lobanov_2024_Linear_Convergence_Rate}. However, the analysis in \citep{Lobanov_2024_Linear_Convergence_Rate} introduces the hyperparameter under the condition $\lambda > 0: \|\nabla f(x^k)\| \geq \lambda$ (for all $k$). In this work, we explicitly characterize the convergence behavior for any $\lambda > 0$. Second, \citep[Theorem 4.1,][]{Lobanov_2025} relies on a strong assumption of bounded gradients and requires a sufficiently large batch size of order $B=\mathcal{O}(\varepsilon^{-3})$. In contrast, the results presented in Theorem~\ref{th:Convex_NSGD} eliminate the need for the bounded gradient. Furthermore, by leveraging the SGC, the required batch size is $B=\mathcal{O}(\rho)$.
        \begin{boxD}
        \begin{proposition}[Heavy-tailed]\label{cor:Convex_NSGD_HT}
            Let $f(\cdot,\xi)$ be $(\cL_0,\cL_1)$-smooth and convex (Assumption~\ref{ass:Convexity_Smoothness_L0_L1}), gradient oracle is unbiased (Assumption~\ref{ass:unbiased}) and satisfies SGC under heavy-tailed noise (Assumption~\ref{prop:SGC_under_heavy tailed noise} with $\sigma = 0$). Then NSGD (Algorithm~\ref{algo:clip_algo}) with hyperparameter $\lambda > 0$, step size $\eta \leq \frac{\lambda}{\cL_0 + \sqrt{\rho}\cL_1 \lambda}$ and batch size $B \geq 2^{\frac{3p +1}{p-1}} (\rho-1)^{\frac{1}{p-1}}$ guarantees convergence rate:
            \vspace{-0.5em}
            \begin{equation*}
                \expect{\func{x^{N}}} - f^* \leq \left( 1 - \frac{\eta}{4 R_0} \right)^N F_0 + 6\lambda R_0.
            \end{equation*}
        \end{proposition}
        \end{boxD}Proposition~\ref{cor:Convex_NSGD_HT} shows that in the convex setting, the SGC under heavy-tailed noise yields the effect: NSGD achieves efficiency matching that of its deterministic counterpart. The proof of Proposition~\ref{cor:Convex_NSGD_HT} is provided in Appendix~\ref{app:NSGD_heavy_tailed_noise}. We also present generalizations of Theorem~\ref{th:Convex_NSGD} to the case $\sigma \geq 0$ (generalized strong growth condition, Assumption~\ref{ass:strong growth condition}) in Appendix~\ref{app:Convex_NSGD}.

\section{Discussion and Limitation}\label{sec:Discussion and Limitation}
\vspace{-1em}
First, due to space constraints in the main text, we refer the reader to Appendix~\ref{app:Extending the Analysis to H_0 H_1 Smoothness}. There, we extend the analysis of Section~\ref{sec:Stochastic Setup via Constant Step Size} to convex $(\cH_0,\cH_1)$-smooth (which is strictly weaker than standard assumptions in optimization literature) functions and provide an improved analysis of GD with a warm-up step-size strategy under $(H_0,H_1)$-smoothness. The latter covers several regimes, including non-convex, convex, and strongly convex objectives, as well as under the PL condition.

Second, the results of Theorems~\ref{th:Convex_ClipSGD} and~\ref{th:Convex_NSGD}, together with Proposition~\ref{cor:Convex_NSGD_HT}, show that the iteration complexity scales as $N \sim \sqrt{\rho}$. This dependence stems from assuming convexity and smoothness at the level of individual realizations, i.e., for almost every $\xi \in \mathcal{D}$ (see Assumption~\ref{ass:Convexity_Smoothness_L0_L1}). It can be avoided if Assumption~\ref{ass:Convexity_Smoothness_L0_L1} is imposed directly on the mini-batch oracle associated with $\bxi$. Third, the GD step size in Theorems~\ref{th:GD_NC}-\ref{th:GD_strongly_convex} requires knowledge of $f^*$; however, this information is often readily available in modern machine learning applications \citep[e.g., for the squared loss in regression and the logistic loss in classification, $f^* = 0$][]{Bartlett_2006}.

Finally, as directions for future work, we identify the following avenues: extending the results of Subsection~\ref{subsec:Non convex functions} to $(H_0,H_1)$-smooth objectives;  generalizing Theorem~\ref{th:Convex_ClipSGD} to the case $\sigma>0$; extending our ClipSGD guarantees to heavy-tailed noise (Assumption~\ref{prop:SGC_under_heavy tailed noise}). Another promising direction is to understand how the results of Subsection~\ref{subsec:Convex functions} change when accelerated~methods~are~considered.
\vspace{-1em}
\section{Numerical Experiments}\label{sec:Numerical Experiments}
\vspace{-1em}
This section presents experiments to validate our theoretical results.

\begin{figure}[t]
    \centering
    \begin{subfigure}[b]{0.48\textwidth}
        \centering
        \includegraphics[width=\textwidth]{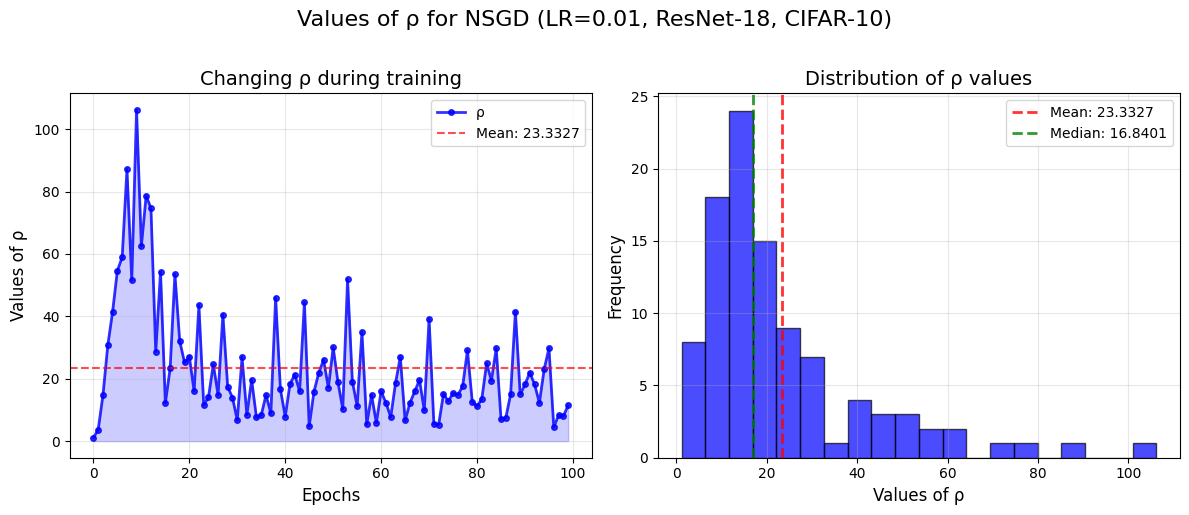}
    \end{subfigure}
    \hfill
    \begin{subfigure}[b]{0.48\textwidth}
        \centering
        \includegraphics[width=\textwidth]{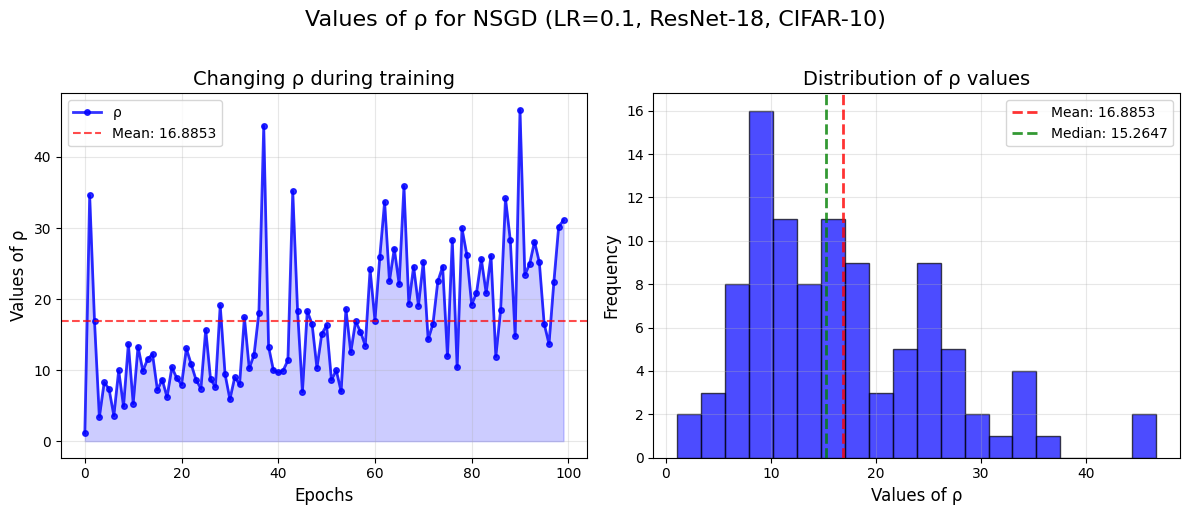}
    \end{subfigure}
    
    \vspace{0.5em}
    
    \begin{subfigure}[b]{0.48\textwidth}
        \centering
        \includegraphics[width=\textwidth]{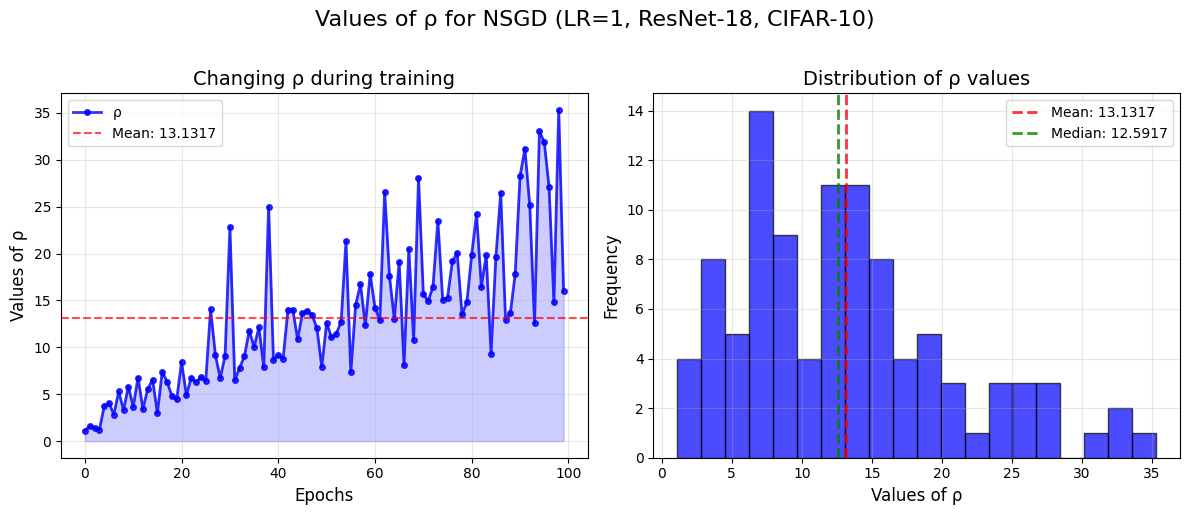}
    \end{subfigure}
    \hfill
    \begin{subfigure}[b]{0.48\textwidth}
        \centering
        \includegraphics[width=\textwidth]{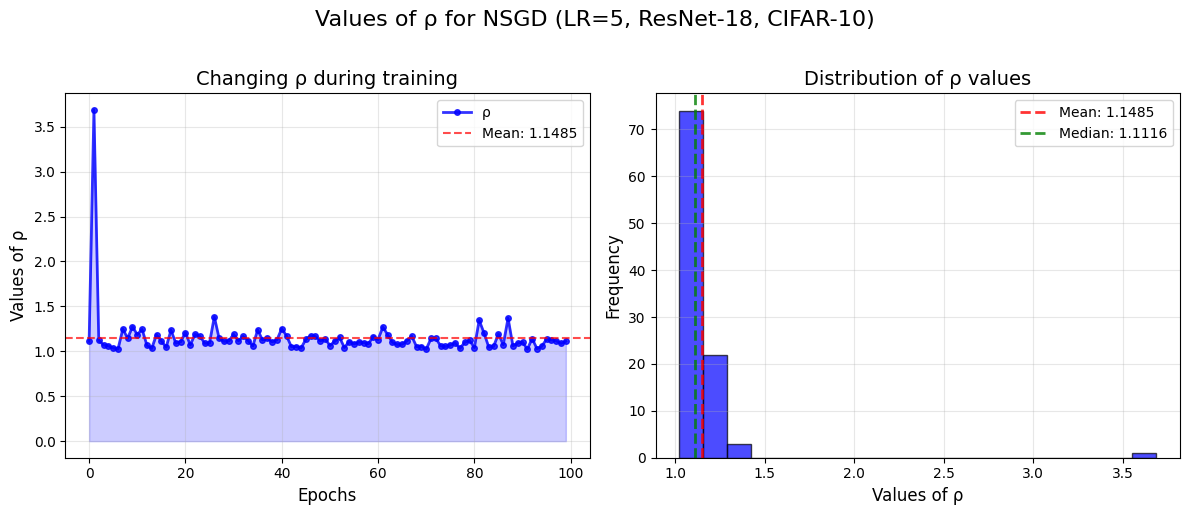}
    \end{subfigure}
    
    \vspace{0.5em}
    
    \begin{subfigure}[b]{0.48\textwidth}
        \centering
        \includegraphics[width=\textwidth]{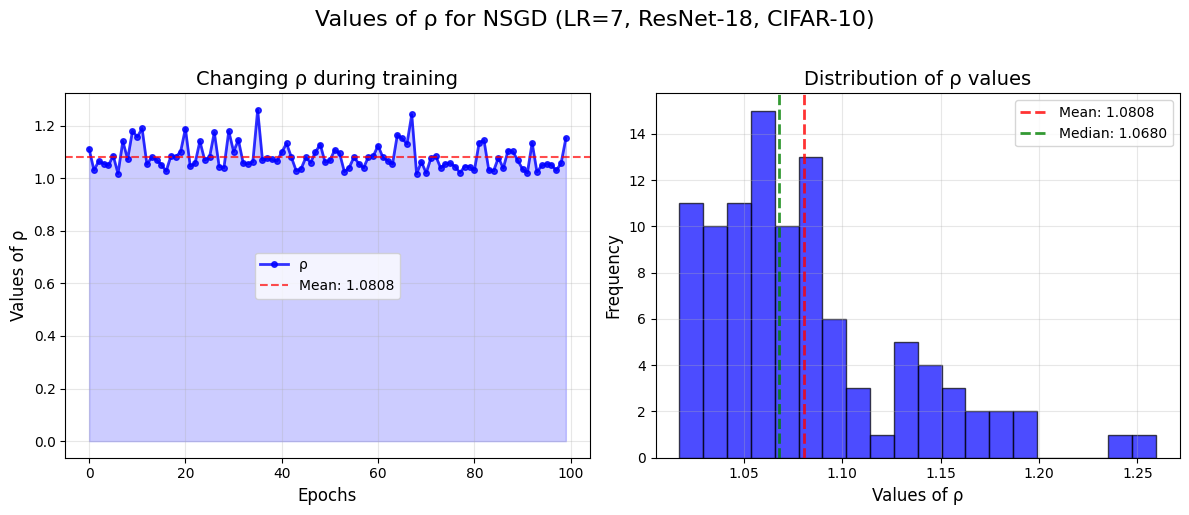}
    \end{subfigure}
    \hfill
    \begin{subfigure}[b]{0.48\textwidth}
        \centering
        \includegraphics[width=\textwidth]{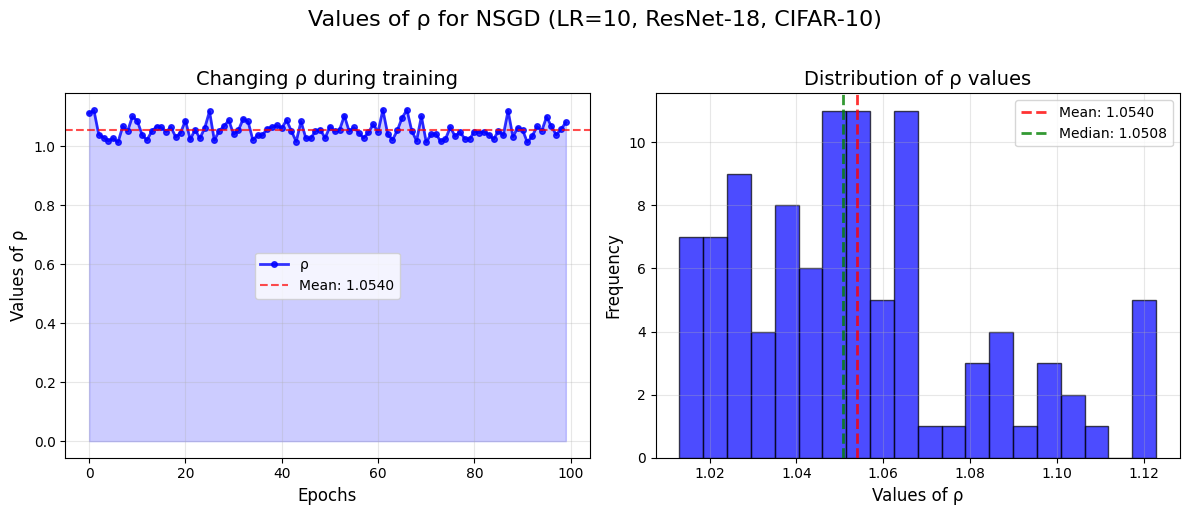}
    \end{subfigure}
    
    \caption{Illustration of the values of $\rho$ obtained from NSGD runs with different learning rates.}
    \label{fig:rho}
    \vspace{-1em}
\end{figure}

We train a ResNet-18 model on the CIFAR-10 dataset \cite{Krizhevsky_2009} for 100 epochs using ClipSGD and NSGD (the methods studied in this work; see Algorithm~\ref{algo:clip_algo} and Theorems~\ref{th:NC_ClipSGD},\ref{th:NC_NSGD}) and standard SGD (a method extensively analyzed under the strong growth condition \citep[Assumption~\ref{ass:strong growth condition}); see Theorem 3 in ][]{Vaswani_2019_Fast_and_faster}. The CIFAR-10 dataset \cite{Krizhevsky_2009} consists of 60,000 color images of size 32×32, including 50,000 training images and 10,000 test images. ResNet-18 \cite{He_2016} is a deep convolutional neural network composed of 18 layers. The total number of model parameters exceeds 11 million, which confirms that the network operates in a highly overparameterized regime. The considered task corresponds to a non-convex optimization problem arising in multi-class image classification with a deep neural network model.

Before proceeding to compare the best runs of the algorithms, we first empirically estimate the constant $\rho$ from the strong growth condition (see Assumption~\ref{ass:strong growth condition}) for the given model and dataset. To estimate $\rho$, we run NSGD with different learning rates (LR) $[0.01, 0.1, 1, 5, 7, 10]$ and a batch size of $B = 128$. At each epoch, we compute $\rho$ using the following finite-sum approximation: $\rho \simeq \left(1/n \sum_{\xi = 1}^n \norms{\nabla f_\xi(x)}^2\right) / \left(\norms{1/n \sum_{\xi = 1}^n \nabla f_\xi(x) }^2\right)$. where $n$ is the total number of data points. The resulting values of $\rho$ for different learning rates are illustrated in Figure~\ref{fig:rho}.

In Figure~\ref{fig:rho}, we present the evolution of $\rho$ over epochs, as well as its distribution in a frequency–value format. In all experiments, we use a total of $100$ epochs. It is easy to observe that as the learning rate increases, the values of $\rho$ tend to decrease. Additionally, for (LR) $= 0.01$, we observe a clear trend: as the method approaches a solution, the values of $\rho$ decrease. Across all runs, the maximum observed value of $\rho$ is $106.16$ (see (LR) $=0.01$). Therefore, in the subsequent experiments, we choose $B = 128$, since the batch size should exceed the empirically estimated value of $\rho$.  

In Figure~\ref{fig:Best_comparison}, we compare the best-performing runs of ClipSGD, NSGD and SGD, that is, the runs with tuned hyperparameters. For ClipSGD we tested the following learning rates (LR) $[0.01, 0.1, 0.5, 1, 5, 7, 10]$, and the clipping radius (R) $[0.0001, 0.001, 0.01, 0.1, 0.25]$; For NSGD we tested the following learning rates (LR) $[0.01, 0.1, 0.5, 1, 5, 7, 10]$; For SGD we tested the following learning rates (LR) $[6.25 \cdot 10^{-5}, 0.000125, 0.000625, 0.001, 0.00125, 0.01, 0.1]$. It is easy to observe that even with carefully tuned parameters, ClipSGD and NSGD demonstrate superior performance compared to SGD in overparameterized models. We also observe that ClipSGD and NSGD exhibit essentially identical convergence behavior. This can be explained by the fact that, for NSGD, the learning rate (LR) matches (LR) $\cdot$ (R) in ClipSGD (corresponding to the regime of large stochastic gradients). Finally, we emphasize that the best-performing run of SGD is achieved with (LR)$= 0.01$, which is $100$ times smaller than the effective step size used in ClipSGD ((LR) $\cdot$ (R) in the large-gradient regime) and in NSGD. This observation indicates that, even with batching, the (LR) for SGD is smaller by approximately a factor of $\rho$ compared to ClipSGD and NSGD \citep[see][for results with $\eta \leq (\rho L)^{-1}$ in the non-batch setting]{Vaswani_2019_Fast_and_faster}. This shows that, empirically, deep learning training exhibits the same effect predicted by Theorems \ref{th:NC_ClipSGD} and \ref{th:NC_NSGD}: larger learning rates can be safely used for ClipSGD and NSGD than for standard SGD,~helping~explain~their~practical~effectiveness.
%Which shows that empirically for deep learning training we observe the same effect as predicted by our theory in Theorems \ref{th:NC_ClipSGD} and \ref{th:NC_NSGD}, highlighting that larger (LR) can be safely used for ClipSGD and NSGD compared to standard SGD, and explaining effectiveness of ClipSGD and NSGD in practice.

\begin{figure}[t]
    \centering
    \includegraphics[width=1\linewidth]{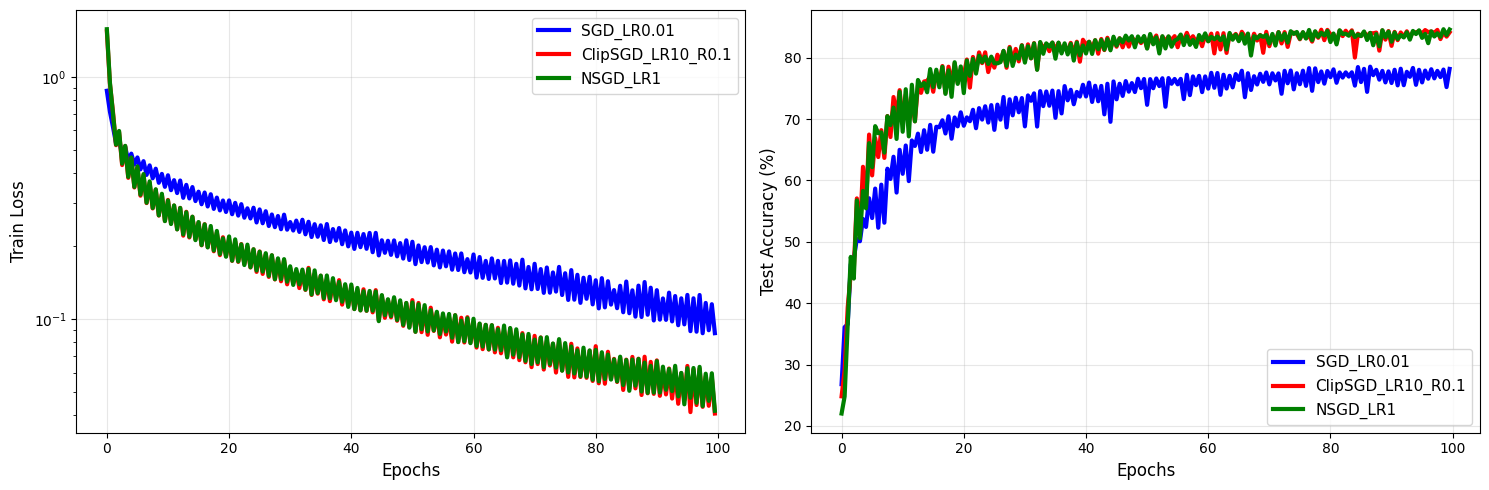}
    \caption{Comparison of ClipSGD, NSGD and SGD with tuned hyperparameters.}
    \label{fig:Best_comparison}
    \vspace{-1em}
\end{figure}

We would like to emphasize once again that all experiments were conducted with a batch size of $B = 128$, for which we empirically estimated $\rho = 106.16$. This demonstrates that our results hold under a mild assumption on the batch size.
\vspace{-1em}
\section{Conclusion}\label{sec:Conclusion}
\vspace{-0.5em}
%In this work, we propose %and theoretically analyze 
% In this work, we propose a novel unified framework that generalizes various existing smoothness conditions (Section~\ref{sec:Problem Setup and New Assumption}). We analyze the convergence of ClipSGD and NSGD with constant step size in the overparameterization regime (Section~\ref{sec:Stochastic Setup via Constant Step Size}), improving upon prior results for both non-convex and convex functions. Notably, under a mild batch size assumption, we prove that both ClipSGD and NSGD eliminate clipping bias and converge to the true optimum. Finally, we derive improved convergence rates for Gradient Descent using a warm-up step size strategy (Section~\ref{sec:Deterministic Setup via Warm-Up Step Size}).
% \section*{Impact Statement}
% Our work focuses on the theoretical analysis of clipped and normalized SGD, and there are no specific societal consequences that must be highlighted here.

We developed a convergence theory for clipping-based stochastic methods in the interpolation regime. Under the strong growth condition and relaxed $(L_0,L_1)$-smoothness, we showed that ClipSGD and NSGD avoid the bias typically associated with clipping and, with a mild mini-batch requirement, achieve nearly deterministic convergence behavior. Our results cover both non-convex and convex objectives, allow arbitrary clipping radius for ClipSGD, and identify regimes where NSGD improves the iteration complexity of ClipSGD without increasing oracle complexity. We further extended the NSGD analysis to heavy-tailed gradients and provided complementary results for $(H_0,H_1)$-smooth objectives and GD with warm-up step sizes. Experiments on overparameterized ResNet-18 models trained on CIFAR-10 support the theory, showing that the required batch sizes are mild and that clipping-based methods can use substantially larger effective step sizes than standard SGD.

\section*{Acknowledgments}
The authors are grateful to Alexander Gasnikov for useful discussions and to Eduard Gorbunov for his careful reading of an earlier version of the article and for pointing out several inaccuracies. AL thanks Yuriy Dorn for his support. This work was done in part while AK was visiting the Simons Institute for the Theory of Computing.

%% file: NeurIPS_26/4_appendix.tex
\newpage
\appendix

%%%%%%%%%%%%%%%%%%%%%%%%%%%%%%%%%%%%%%%%%%%%%%%%%%%%%%%%%%%%%%%%%%%%%%%%%%%%%%%
\vbox{%
    \hsize\textwidth
    \linewidth\hsize
    \vskip 0.1in
    \vskip 0.19in
    \vskip -\parskip
    \hrule height 1pt
    \vskip 0.09in
    \vskip 0.09in
    %\centering
    \begin{center}
        \LARGE \bf APPENDIX 
    \end{center}
    %\vskip 0.1in
    \vskip 0.19in
    \vskip -\parskip
    \hrule height 1pt
    \vskip 0.09in}
%%%%%%%%%%%%%%%%%%%%%%%%%%%%%%%%%%%%%%%%%%%%%%%%%%%%%%%%%%%%%%%%%%%%%%%%%%%%%%%

%%%%%%%%%%%%%%%%%%%%%%%%%%%%%%%%%%%%%%%%%%%%%%%%%%%%%%%%%%%%%%%%%%%%%%%%%%%%%%%
%%%%%%%%%%%%%%%%%%%%%%%%%%%%%%%%%%%%%%%%%%%%%%%%%%%%%%%%%%%%%%%%%%%%%%%%%%%%%%%

\tableofcontents

\section{Auxiliary Results}\label{app:Auxiliary Results}

In this section, we provide auxiliary technical results that are used in our analysis.

\paragraph{Basic inequalities.} For all $a,b \in \mathbb{R}^d$ ($d \geq 1$) and $\lambda>0$, the following inequalities hold:
\begin{equation}
    \label{eq:quadratic_difference}
    \dotprod{a}{b}  = \frac{\norms{a}^2}{2} + \frac{\norms{b}^2}{2} - \frac{\norms{b - a}^2}{2},
\end{equation}
\begin{equation}
    \label{eq:scalar_product_bound}
    \dotprod{a}{b} \leq \| a \| \cdot \| b \|,
\end{equation}
\begin{equation}\label{eq:Fenchel_Young_inequality}
    \dotprod{a}{b} \leq \frac{\norms{a}^2}{2 \lambda} + \frac{\lambda \norms{b}^2}{2}.
\end{equation}

\paragraph{$(L_0,L_1)$-smoothness.} Assumption \ref{ass:L_0_L_1Smooth} has the equivalent form $\forall x,y \in \mathbb{R}^d$ with $\norms{y-x} \leq \frac{1}{L_1}$:
\begin{equation}
    \label{eq:L0L1_descent_inequality}
    f(y) - f(x) \leq \dotprod{\nabla f(x)}{ y - x} + \frac{L_0 + L_1 \norms{\nabla f(x)}}{2} \norms{y-x}^2.
\end{equation}

\begin{assumption}[Convexity]\label{ass:Convexity}
The function $f$ is convex, i.e., for all $x,y\in\mathbb{R}^d$,
\begin{equation}
    \label{eq:Convexity}
    \func{x} - \func{y} \leq \dotprod{\nabla \func{x}}{ x - y}, \quad\quad\quad \forall x,y \in \mathbb{R}^d. 
\end{equation} 
\end{assumption}

% \paragraph{Variance decomposition.} If $\xi$ is random vector in $\mathbb{R}^d$ with bounded second moment, then 
% \begin{equation}\label{eq:Variance_decomposition}
%     \expect{\norms{\xi - \expect{\xi}}^2}  = \expect{\norms{\xi + a}^2} - \norms{\expect{\xi} + a}^2,
% \end{equation}
% for any deterministic vector $a\in \mathbb{R}^d$.

\subsection{On the generalized strong growth condition}\label{app:On the generalized strong growth condition}
\paragraph{Generalized strong growth condition.} In this work, we assume that the generalized strong growth condition holds, i.e.
\begin{equation*}
    \expect{\norms{\nabla \func{x, \xi}}^2} \leq \rho \norms{\nabla \func{x}}^2 + \sigma^2. 
\end{equation*}

Subtracting the norm of the full gradient from both sides, we have that
\begin{equation*}
    \expect{\norms{\nabla \func{x, \xi}}^2} - \norms{\nabla \func{x}}^2 \leq (\rho - 1) \norms{\nabla \func{x}}^2 + \sigma ^2.
\end{equation*}

We know that
\begin{align*}
    \expect{\norms{\nabla \func{x, \xi} - \nabla \func{x}}^2} &= \expect{\norms{\nabla \func{x, \xi}}^2 - 2 \dotprod{\nabla \func{x, \xi}}{\nabla \func{x}} + \norms{\nabla \func{x}}^2}\\
    &= \expect{\norms{\nabla \func{x, \xi}}^2} - \norms{\nabla \func{x}}^2.
\end{align*}

Therefore, on the left-hand side we now have exactly the variance, and hence it is equivalent to
\begin{equation}
    \label{eq:gradient_noise_variance}
    \expect{\norms{\nabla \func{x, \xi} - \nabla \func{x}}^2} \leq (\rho - 1) \norms{\nabla \func{x}}^2 + \sigma ^2.
\end{equation}

\paragraph{Strong growth condition.} We note that interpolation regime is achieved when $\sigma = 0$, i.e.
\begin{equation*}
    \expect{\norms{\nabla \func{x, \xi}}^2} \leq \rho \norms{\nabla \func{x}}^2. 
\end{equation*}
From \eqref{eq:gradient_noise_variance}, we see that in this case the variance has the following upper bound:
\begin{equation*}
    \expect{\norms{\nabla \func{x, \xi} - \nabla \func{x}}^2} \leq (\rho - 1) \norms{\nabla \func{x}}^2.
\end{equation*}
\textbf{This means that $\rho$ cannot be smaller than 1!} If $\rho$ is smaller than $1$, it means that expectation of the positive number has to be negative and no function satisfies this.

\subsection{On the generalized smoothness conditions}\label{subsec:generalized_smoothness_conditions}

\paragraph{$(L_0,L_1)$-Smoothness.} In this work, we rely on two smoothness assumptions: Assumption~\ref{ass:L_0_L_1Smooth} for the non-convex setting (see Section~\ref{subsec:Non convex functions}) and Assumption~\ref{ass:Convexity_Smoothness_L0_L1} for the convex setting (see Section~\ref{subsec:Convex functions}). It is straightforward to verify that, under the strong growth condition (Assumption~\ref{ass:strong growth condition} with $\sigma = 0$), Assumption~\ref{ass:Convexity_Smoothness_L0_L1} implies Assumption~\ref{ass:L_0_L_1Smooth} with $L_0 = \cL_0$ and $L_1 = \sqrt{\rho} \cL_1$:
\begin{align}\label{eq:new_smoothness_with_rho}
    \norms{\nabla \func{y} - \nabla \func{x}} &\overset{\circledOne}{=} \norms{\expect{\nabla \func{y,\xi} - \nabla \func{x,\xi}}}\nonumber \\
    &\overset{\circledTwo}{\leq} \expect{\norms{\nabla \func{y,\xi} - \nabla \func{x,\xi}}}\nonumber\\
    &\overset{\circledThree}{\leq} \left( \cL_0 + \cL_1 \expect{\norms{\nabla \func{x,\xi}}}  \right) \norms{y - x}\nonumber\\
    &= \left( \cL_0 + \cL_1 \expect{\left[ \norms{\nabla \func{x,\xi}}^2 \right]^{1/2}}  \right) \norms{y - x}\nonumber\\
    &\overset{\circledFour}{\leq} \left( \cL_0 + \cL_1 \left[ \expect{\norms{\nabla \func{x,\xi}}^2 }\right]^{1/2}  \right) \norms{y - x}\nonumber\\
    &\overset{\circledFive}{\leq} \left( \cL_0 + \sqrt{\rho} \cL_1 \norms{\nabla \func{x}}  \right) \norms{y - x},
\end{align}
where in $\circledOne$ we used the unbiasedness relation
$\nabla \func{x} = \expect{\nabla \func{x,\xi}}$ and
$\nabla \func{y} = \expect{\nabla \func{y,\xi}}$ (Assumption~\ref{ass:unbiased}), in $\circledTwo$, we applied Jensen's inequality to the norm, in $\circledThree$, we used the $(\cL_0,\cL_1)$-smoothness assumption
for each realization $\xi$ (Assumption~\ref{ass:Convexity_Smoothness_L0_L1}), in $\circledFour$, we used Jensen's inequality, namely
$\expect{\norms{Z}} \leq \left(\expect{\norms{Z}^2}\right)^{1/2}$, in $\circledFive$, we applied the strong growth condition (Assumption~\ref{ass:strong growth condition}). 

Similarly, Assumption~\ref{ass:Convexity_Smoothness_L0_L1} also implies the
deterministic descent inequality \eqref{eq:L0L1_descent_inequality} with $L_0 = \cL_0$ and $L_1 = \sqrt{\rho} \cL_1$,
\begin{align*}
    \func{y}-\func{x} =
    \expect{\func{y,\xi}-\func{x,\xi}}
    &\overset{\circledOne}{\leq}
    \expect{
        \dotprod{\nabla \func{x,\xi}}{y-x}
        +
        \frac{
            \cL_0+\cL_1\norms{\nabla \func{x,\xi}}
        }{2}
        \norms{y-x}^2
    }
    \nonumber\\
    &=
    \dotprod{
        \expect{\nabla \func{x,\xi}}
    }{
        y-x
    }
    +
    \frac{
        \cL_0+\cL_1\expect{\norms{\nabla \func{x,\xi}}}
    }{2}
    \norms{y-x}^2
    \nonumber\\
    &\overset{\circledTwo}{=}
    \dotprod{\nabla \func{x}}{y-x}
    +
    \frac{
        \cL_0+\cL_1\expect{\norms{\nabla \func{x,\xi}}}
    }{2}
    \norms{y-x}^2
    \nonumber\\
    &\overset{\circledThree}{\leq}
    \dotprod{\nabla \func{x}}{y-x}
    +
    \frac{
        \cL_0+\cL_1
        \left[
            \expect{\norms{\nabla \func{x,\xi}}^2}
        \right]^{1/2}
    }{2}
    \norms{y-x}^2
    \nonumber\\
    &\overset{\circledFour}{\leq}
    \dotprod{\nabla \func{x}}{y-x}
    +
    \frac{
        \cL_0+\sqrt{\rho}\cL_1\norms{\nabla \func{x}}
    }{2}
    \norms{y-x}^2,
    \label{eq:new_descent_with_rho}
\end{align*}
where in $\circledOne$ we used
$(\cL_0,\cL_1)$-smoothness inequality from
Assumption~\ref{ass:Convexity_Smoothness_L0_L1}; in $\circledTwo$ we used
unbiasedness relation
$\nabla \func{x}=\expect{\nabla \func{x,\xi}}$
(Assumption~\ref{ass:unbiased}); in $\circledThree$ we used Jensen's
inequality,
$\expect{\norms{Z}}\leq(\expect{\norms{Z}^2})^{1/2}$; and in
$\circledFour$ we applied the strong growth condition
(Assumption~\ref{ass:strong growth condition} with $\sigma=0$).

\paragraph{$(H_0,H_1)$-Smoothness.} In this work, we extend our analysis to $(H_0,H_1)$-smooth functions (see Assumption~\ref{ass:H_0H_1_Smooth}). Unlike $(L_0,L_1)$-smoothness, the $(H_0,H_1)$-smoothness condition depends on the function gap rather than on the gradient norm. As a result, taking expectation directly yields the corresponding deterministic condition. Consequently, if $(\cH_0,\cH_1)$-smoothness (see Assumption~\ref{ass:H_0H_1_Smooth_for_xi}) holds for each realization, then the deterministic $(H_0,H_1)$-smoothness condition in Assumption~\ref{ass:H_0H_1_Smooth} is also satisfied with the same constants $H_0 = \cH_0$ and $H_1 = \cH_1$.

\subsection{Auxiliary lemma for \texorpdfstring{$(\cL_0,\cL_1)$}{TEXT}-smoothness}
\begin{boxN}
\begin{lemma}
    Let $f(\cdot,\xi)$ be $(\cL_0,\cL_1)$-smooth and convex for each realization $\xi \in \mathcal{D}$ (Assumption~\ref{ass:Convexity_Smoothness_L0_L1}), and let the interpolation condition (Definition~\ref{def:interpolation}) hold, then~${\forall x \in \mathbb{R}^d}$ we have upper bound for norm of mini-batch gradient  $\norms{\nabla \func{x,\bxi}} = \norms{\frac{1}{B} \sum_{i=1}^B \nabla \func{x,\xi_i}}$:
    \begin{equation}\label{eq:Auxiliary_lemma}
        \norms{\nabla \func{x,\bxi}}^2 \leq 2 \left( \cL_0 + \cL_1 \norms{\nabla \func{x, \bxi}} \right) \dotprod{\nabla \func{x, \bxi}}{x - x^*}.
    \end{equation}
\end{lemma}
\end{boxN}
\begin{proof}
    By the $(\cL_0,\cL_1)$-smoothness of $f(\cdot,\xi)$, for every realization 
    $\xi \in \mathcal{D}$ (see Assumption~\ref{ass:Convexity_Smoothness_L0_L1}) we have the following descent inequality:
    \begin{equation*}
        \func{y,\xi}
        \leq
        \func{x,\xi}
        +
        \dotprod{\nabla \func{x,\xi}}{y-x}
        +
        \frac{\cL_0+\cL_1\norms{\nabla \func{x,\xi}}}{2}
        \norms{y-x}^2,
    \end{equation*}
    whenever the locality condition required by $(\cL_0,\cL_1)$-smoothness is satisfied.
    Let
    \[
        y
        =
        x
        -
        \frac{\nabla \func{x,\xi}}
        {\cL_0+\cL_1\norms{\nabla \func{x,\xi}}}.
    \]
    Then, for $\cL_1>0$,
    \begin{equation*}
        \norms{y-x}
        =
        \frac{\norms{\nabla \func{x,\xi}}}
        {\cL_0+\cL_1\norms{\nabla \func{x,\xi}}}
        \leq
        \frac{1}{\cL_1},
    \end{equation*}
    and hence the locality condition is satisfied. If $\cL_1=0$, the corresponding 
    descent inequality is global and the argument below remains valid.

    By interpolation (see Definition~\ref{def:interpolation}), $x^*$ is a minimizer of each realization $f(\cdot,\xi)$.
    Therefore, $\func{x^*,\xi} \leq \func{y,\xi}$. Applying the descent inequality 
    with the above choice of $y$, we obtain
    \begin{align*}
        \func{x^*,\xi}
        &\leq
        \func{x,\xi}
        -
        \frac{\norms{\nabla \func{x,\xi}}^2}
        {\cL_0+\cL_1\norms{\nabla \func{x,\xi}}}
        +
        \frac{\cL_0+\cL_1\norms{\nabla \func{x,\xi}}}{2}
        \left(
            \frac{\norms{\nabla \func{x,\xi}}}
            {\cL_0+\cL_1\norms{\nabla \func{x,\xi}}}
        \right)^2 \\
        &=
        \func{x,\xi}
        -
        \frac{\norms{\nabla \func{x,\xi}}^2}
        {2\left(\cL_0+\cL_1\norms{\nabla \func{x,\xi}}\right)}.
    \end{align*}
    Hence,
    \begin{equation*}
        \func{x,\xi}
        -
        \func{x^*,\xi}
        \geq
        \frac{\norms{\nabla \func{x,\xi}}^2}
        {2\left(\cL_0+\cL_1\norms{\nabla \func{x,\xi}}\right)}.
    \end{equation*}
    On the other hand, by convexity of $f(\cdot,\xi)$ (see Assumption~\ref{ass:Convexity_Smoothness_L0_L1}),
    \begin{equation*}
        \func{x,\xi}
        -
        \func{x^*,\xi}
        \leq
        \dotprod{\nabla \func{x,\xi}}{x-x^*}.
    \end{equation*}
    Combining the last two inequalities gives, for each realization $\xi$,
    \begin{equation*}
        \dotprod{\nabla \func{x,\xi}}{x-x^*}
        \geq
        \frac{\norms{\nabla \func{x,\xi}}^2}
        {2\left(\cL_0+\cL_1\norms{\nabla \func{x,\xi}}\right)}.
    \end{equation*}

    Averaging this inequality over the mini-batch, we get
    \begin{align*}
        \dotprod{\nabla \func{x,\bxi}}{x-x^*}
        &=
        \frac{1}{B}
        \sum_{i=1}^B
        \dotprod{\nabla \func{x,\xi_i}}{x-x^*} \geq
        \frac{1}{2B}
        \sum_{i=1}^B
        \frac{\norms{\nabla \func{x,\xi_i}}^2}
        {\cL_0+\cL_1\norms{\nabla \func{x,\xi_i}}}.
    \end{align*}

    Now define, for $t\geq 0$,
    \[
        \phi(t)
        =
        \frac{t^2}{\cL_0+\cL_1 t}.
    \]
    This function is nondecreasing and convex, since
    \begin{align*}
        \phi'(t)
        &=
        \frac{t(2\cL_0+\cL_1 t)}
        {(\cL_0+\cL_1 t)^2}
        \geq 0, \\
        \phi''(t)
        &=
        \frac{2\cL_0^2}
        {(\cL_0+\cL_1 t)^3}
        \geq 0.
    \end{align*}
    Therefore, by Jensen's inequality and then by the triangle inequality together 
    with the monotonicity of $\phi$, we have
    \begin{align*}
        \frac{1}{B}
        \sum_{i=1}^B
        \phi\left(\norms{\nabla \func{x,\xi_i}}\right)
        &\geq
        \phi\left(
            \frac{1}{B}
            \sum_{i=1}^B
            \norms{\nabla \func{x,\xi_i}}
        \right) \\
        &\geq
        \phi\left(
            \norms{
                \frac{1}{B}
                \sum_{i=1}^B
                \nabla \func{x,\xi_i}
            }
        \right) \\
        &=
        \phi\left(\norms{\nabla \func{x,\bxi}}\right) \\
        &=
        \frac{\norms{\nabla \func{x,\bxi}}^2}
        {\cL_0+\cL_1\norms{\nabla \func{x,\bxi}}}.
    \end{align*}
    Consequently,
    \begin{equation*}
        \dotprod{\nabla \func{x,\bxi}}{x-x^*}
        \geq
        \frac{\norms{\nabla \func{x,\bxi}}^2}
        {2\left(\cL_0+\cL_1\norms{\nabla \func{x,\bxi}}\right)}.
    \end{equation*}
    Rearranging the last inequality yields
    \begin{equation*}
        \norms{\nabla \func{x,\bxi}}^2
        \leq
        2\left(
            \cL_0
            +
            \cL_1\norms{\nabla \func{x,\bxi}}
        \right)
        \dotprod{\nabla \func{x,\bxi}}{x-x^*}.
    \end{equation*}
\end{proof}

\subsection{Monotone behavior of the distance to optimum for ClipSGD}\label{app:Monotone behavior of the distance to optimum for ClipSGD}

\begin{boxN}
\begin{lemma}\label{lem:monotone_clipsgd}
Let $f(\cdot,\xi)$ be $(\cL_0,\cL_1)$-smooth and convex for each realization $\xi$ (Assumption~\ref{ass:Convexity_Smoothness_L0_L1}), and let the interpolation condition from Definition~\ref{def:interpolation} hold. Let $R_0:=\norms{x^0-x^*}$ and let $\rho\geq1$ be the strong growth parameter. Then ClipSGD with step size
$
    \eta
    \leq
    \left[
        9
        \left(
            \cL_0
            +
            \sqrt{\rho}\cL_1 c
        \right)
    \right]^{-1}
$ guarantees
\begin{equation*}
    \norms{x^{k+1}-x^*}^2
    \leq
    \norms{x^k-x^*}^2.
\end{equation*}
Consequently,
\begin{equation*}
    \norms{x^k-x^*}
    \leq
    \norms{x^0-x^*}
    =
    R_0
    \quad
    \text{for all } k\geq 0.
\end{equation*}
\end{lemma}
\end{boxN}

\begin{proof}
We prove the statement by induction. The claim is trivial for $k=0$. Assume that
\begin{equation*}
    \norms{x^k-x^*}
    \leq
    R_0.
\end{equation*}
We show that
\begin{equation*}
    \norms{x^{k+1}-x^*}
    \leq
    \norms{x^k-x^*}.
\end{equation*}

For a mini-batch $\bxi^k$, define
\begin{equation*}
    F_k(\bxi^k)
    :=
    \func{x^k,\bxi^k}
    -
    \func{x^*,\bxi^k}.
\end{equation*}
By the interpolation condition, $x^*$ minimizes each stochastic realization almost surely, and hence it also minimizes the mini-batch objective. Therefore,
\begin{equation}\label{eq:Fk_polozitelen}
    F_k(\bxi^k)\geq 0.
\end{equation}
Moreover, by convexity \eqref{eq:Convexity},
\begin{equation}
    \label{eq:convexity_batch_L0L1_Clip}
    F_k(\bxi^k)
    \leq
    \dotprod{\nabla \func{x^k,\bxi^k}}{x^k-x^*}.
\end{equation}

The ClipSGD update has the form
\begin{equation*}
    x^{k+1}
    =
    x^k
    -
    \eta\alpha_{\bxi^k}\nabla \func{x^k,\bxi^k},
\end{equation*}
where
\begin{equation*}
    \alpha_{\bxi^k}
    =
    \min
    \left\{
        1,
        \frac{c}{\norms{\nabla \func{x^k,\bxi^k}}}
    \right\}.
\end{equation*}
Then
\begin{align}
    \norms{x^{k+1}-x^*}^2
    &=
    \norms{
        x^k-x^*
        -
        \eta\alpha_{\bxi^k}\nabla \func{x^k,\bxi^k}
    }^2
    \nonumber \\
    &=
    \norms{x^k-x^*}^2
    -
    2\eta\alpha_{\bxi^k}
    \dotprod{\nabla \func{x^k,\bxi^k}}{x^k-x^*}
    +
    \eta^2\alpha_{\bxi^k}^2
    \norms{\nabla \func{x^k,\bxi^k}}^2.
    \label{eq:ClipSGD_L0L1_monotone_start}
\end{align}

By the self-bounding property of convex $(\cL_0,\cL_1)$-smooth functions applied to the mini-batch objective, and using $\rho\geq 1$, we have
\begin{align}
    \label{eq:self_bounding_L0L1_Clip}
    \norms{\nabla \func{x^k,\bxi^k}}^2
    &\overset{\eqref{eq:Auxiliary_lemma}}{\leq}
    2
    \left(
        \cL_0
        +
        \cL_1\norms{\nabla \func{x^k,\bxi^k}}
    \right)
    \dotprod{\nabla \func{x^k,\bxi^k}}{x^k-x^*}\nonumber\\
    &\leq 2
    \left(
        \cL_0
        +
        \sqrt{\rho}\cL_1\norms{\nabla \func{x^k,\bxi^k}}
    \right)
    \dotprod{\nabla \func{x^k,\bxi^k}}{x^k-x^*}.
\end{align}
Substituting~\eqref{eq:self_bounding_L0L1_Clip} into~\eqref{eq:ClipSGD_L0L1_monotone_start}, we obtain
\begin{align}
    \norms{x^{k+1}-x^*}^2
    &\leq
    \norms{x^k-x^*}^2
    -
    2\eta\alpha_{\bxi^k}
    \dotprod{\nabla \func{x^k,\bxi^k}}{x^k-x^*}
    \nonumber \\
    &\quad\quad\quad
    +
    2\eta^2\alpha_{\bxi^k}^2
    \left(
        \cL_0
        +
        \sqrt{\rho}\cL_1\norms{\nabla \func{x^k,\bxi^k}}
    \right)
   \dotprod{\nabla \func{x^k,\bxi^k}}{x^k-x^*}.
    \label{eq:ClipSGD_L0L1_monotone_main}
\end{align}

We now consider two cases.

\textbf{Case 1: $\norms{\nabla \func{x^k,\bxi^k}}\leq c$.}
In this case, $\alpha_{\bxi^k}=1$. Using~\eqref{eq:ClipSGD_L0L1_monotone_main}, we get
\begin{align*}
    \norms{x^{k+1}-x^*}^2
    &\leq
    \norms{x^k-x^*}^2
    -
    2\eta \dotprod{\nabla \func{x^k,\bxi^k}}{x^k-x^*}\\
    &\quad\quad\quad +
    2\eta^2
    \left(
        \cL_0+\sqrt{\rho}\cL_1 c
    \right)
    \dotprod{\nabla \func{x^k,\bxi^k}}{x^k-x^*}
    \\
    &\leq
    \norms{x^k-x^*}^2
    -
    2\eta \dotprod{\nabla \func{x^k,\bxi^k}}{x^k-x^*}
    +
    \frac{2}{9}\eta \dotprod{\nabla \func{x^k,\bxi^k}}{x^k-x^*}
    \\
    &=
    \norms{x^k-x^*}^2
    -
    \frac{16}{9}\eta \dotprod{\nabla \func{x^k,\bxi^k}}{x^k-x^*}\\
    &\overset{\eqref{eq:convexity_batch_L0L1_Clip}}{\leq} \norms{x^k-x^*}^2
    -
    \frac{16}{9}\eta F_k(\bxi^k)\\
    &\overset{\eqref{eq:Fk_polozitelen}}{\leq}
    \norms{x^k-x^*}^2.
\end{align*}

\textbf{Case 2: $\norms{\nabla \func{x^k,\bxi^k}}>c$.}
In this case,
\begin{equation*}
    \alpha_{\bxi^k}
    =
    \frac{c}{\norms{\nabla \func{x^k,\bxi^k}}}.
\end{equation*}
Since $\alpha_{\bxi^k}\leq 1$, we have
\begin{align*}
    \alpha_{\bxi^k}
    \left(
        \cL_0
        +
        \sqrt{\rho}\cL_1\norms{\nabla \func{x^k,\bxi^k}}
    \right)
    &=
    \alpha_{\bxi^k}\cL_0
    +
    \sqrt{\rho}\cL_1 c
    \\
    &\leq
    \cL_0
    +
    \sqrt{\rho}\cL_1 c.
\end{align*}
Using this bound in~\eqref{eq:ClipSGD_L0L1_monotone_main}, we obtain
\begin{align*}
    \norms{x^{k+1}-x^*}^2
    &\leq
    \norms{x^k-x^*}^2
    -
    2\eta\alpha_{\bxi^k}\dotprod{\nabla \func{x^k,\bxi^k}}{x^k-x^*}
    \\
    &\quad\quad\quad
    +
    2\eta^2\alpha_{\bxi^k}
    \left(
        \cL_0+\sqrt{\rho}\cL_1 c
    \right)
    \dotprod{\nabla \func{x^k,\bxi^k}}{x^k-x^*}
    \\
    &\leq
    \norms{x^k-x^*}^2
    -
    2\eta\alpha_{\bxi^k}\dotprod{\nabla \func{x^k,\bxi^k}}{x^k-x^*}
    +
    \frac{2}{9}\eta\alpha_{\bxi^k}\dotprod{\nabla \func{x^k,\bxi^k}}{x^k-x^*}
    \\
    &=
    \norms{x^k-x^*}^2
    -
    \frac{16}{9}\eta\alpha_{\bxi^k}\dotprod{\nabla \func{x^k,\bxi^k}}{x^k-x^*}\\
    &\overset{\eqref{eq:convexity_batch_L0L1_Clip}}{\leq} \norms{x^k-x^*}^2
    -
    \frac{16}{9}\eta \alpha_{\bxi^k} F_k(\bxi^k)\\
    &\overset{\eqref{eq:Fk_polozitelen}}{\leq}
    \norms{x^k-x^*}^2.
\end{align*}

Combining both cases, we conclude that
\begin{equation*}
    \norms{x^{k+1}-x^*}^2
    \leq
    \norms{x^k-x^*}^2.
\end{equation*}
Thus, by induction,
\begin{equation*}
    \norms{x^k-x^*}
    \leq
    \norms{x^0-x^*}
    =
    R_0
    \quad
    \text{for all } k\geq 0.
\end{equation*}

Finally, the same step-size condition implies
\begin{equation*}
    \norms{x^{k+1}-x^k}
    =
    \eta\norms{\clip{\nabla \func{x^k,\bxi^k}}}
    \leq
    \eta c
    \leq
    \frac{1}{\sqrt{\rho}\cL_1},
\end{equation*}
which ensures that the ClipSGD update stays within the locality radius required by the corresponding $(L_0,L_1)$-smoothness condition with $L_1=\sqrt{\rho}\cL_1$.
This completes the proof.
\end{proof}

\subsection{Monotone behavior of the distance to optimum for NSGD}\label{app:Monotone behavior of the distance to optimum for NSGD}

\begin{boxN}
\begin{lemma}\label{lem:monotone_nsgd}
Let $f(\cdot,\xi)$ be $(\cL_0,\cL_1)$-smooth and convex for each realization $\xi$ (Assumption~\ref{ass:Convexity_Smoothness_L0_L1}), and let the interpolation condition from Definition~\ref{def:interpolation} hold. Let $R_0:=\norms{x^0-x^*}$ and let $\rho\geq1$ be the strong growth parameter. Then NSGD with any $\lambda>0$ and step size
$
\eta
    \leq
    \frac{\lambda}{
        \cL_0
        +
        \sqrt{\rho}\cL_1\lambda
    }
$
guarantees
\begin{equation*}
    \norms{x^{k+1}-x^*}^2
    \leq
    \norms{x^k-x^*}^2.
\end{equation*}
Consequently,
\begin{equation*}
    \norms{x^k-x^*}
    \leq
    \norms{x^0-x^*}
    =
    R_0
    \quad
    \text{for all } k\geq 0.
\end{equation*}
\end{lemma}
\end{boxN}

\begin{proof}
We prove the statement by induction. The claim is trivial for $k=0$. Assume that
\begin{equation*}
    \norms{x^k-x^*}
    \leq
    R_0.
\end{equation*}
We show that
\begin{equation*}
    \norms{x^{k+1}-x^*}
    \leq
    \norms{x^k-x^*}.
\end{equation*}

For a mini-batch $\bxi^k$, define
\begin{equation*}
    F_k(\bxi^k)
    :=
    \func{x^k,\bxi^k}
    -
    \func{x^*,\bxi^k}.
\end{equation*}
By the interpolation condition, $x^*$ minimizes each stochastic realization almost surely, and hence it also minimizes the mini-batch objective. Therefore,
\begin{equation}\label{eq:Fk_polozitelen_FOR_NSGD}
    F_k(\bxi^k)\geq 0.
\end{equation}
Moreover, by convexity,
\begin{equation}
    \label{eq:convexity_batch_L0L1_NSGD}
    F_k(\bxi^k)
    \leq
    \dotprod{\nabla \func{x^k,\bxi^k}}{x^k-x^*}.
\end{equation}

Introduce the NSGD direction
\begin{equation*}
    \Gf{x^k,\bxi^k}
    \coloneqq
    \frac{\nabla \func{x^k,\bxi^k}}
    {\norms{\nabla \func{x^k,\bxi^k}}+\lambda}.
\end{equation*}
Then the update rule of NSGD can be written as
\begin{equation*}
    x^{k+1}
    =
    x^k
    -
    \eta \Gf{x^k,\bxi^k}.
\end{equation*}
Hence,
\begin{align}
    \norms{x^{k+1}-x^*}^2
    &=
    \norms{x^k-x^*-\eta\Gf{x^k,\bxi^k}}^2
    \nonumber \\
    &=
    \norms{x^k-x^*}^2
    -
    2\eta
    \dotprod{\Gf{x^k,\bxi^k}}{x^k-x^*}
    +
    \eta^2\norms{\Gf{x^k,\bxi^k}}^2
    \nonumber \\
    &=
    \norms{x^k-x^*}^2
    -
    \frac{2\eta}{
        \norms{\nabla \func{x^k,\bxi^k}}+\lambda
    }
    \dotprod{\nabla \func{x^k,\bxi^k}}{x^k-x^*}
    \nonumber \\
    &\quad\quad\quad
    +
    \eta^2
    \left(
        \frac{
            \norms{\nabla \func{x^k,\bxi^k}}
        }{
            \norms{\nabla \func{x^k,\bxi^k}}+\lambda
        }
    \right)^2.
    \label{eq:NSGD_L0L1_monotone_start}
\end{align}

By the self-bounding property of convex $(\cL_0,\cL_1)$-smooth functions applied to the mini-batch objective, and using $\rho\geq 1$, we have
\begin{align}
    \label{eq:self_bounding_L0L1_NSGD}
    \norms{\nabla \func{x^k,\bxi^k}}^2
    &\overset{\eqref{eq:Auxiliary_lemma}}{\leq}
    2
    \left(
        \cL_0
        +
        \cL_1\norms{\nabla \func{x^k,\bxi^k}}
    \right)
    \dotprod{\nabla \func{x^k,\bxi^k}}{x^k-x^*}\nonumber\\
    &\leq 2
    \left(
        \cL_0
        +
        \sqrt{\rho}\cL_1\norms{\nabla \func{x^k,\bxi^k}}
    \right)
    \dotprod{\nabla \func{x^k,\bxi^k}}{x^k-x^*}..
\end{align}
Substituting~\eqref{eq:self_bounding_L0L1_NSGD} into~\eqref{eq:NSGD_L0L1_monotone_start}, we obtain
\begin{align}
    \norms{x^{k+1}-x^*}^2
    &\leq
    \norms{x^k-x^*}^2
    -
    \frac{2\eta}{
        \norms{\nabla \func{x^k,\bxi^k}}+\lambda
    }
    \dotprod{\nabla \func{x^k,\bxi^k}}{x^k-x^*}
    \nonumber \\
    &\quad\quad\quad
    +
    \frac{
        2\eta^2
        \left(
            \cL_0
            +
            \sqrt{\rho}\cL_1\norms{\nabla \func{x^k,\bxi^k}}
        \right)
    }{
        \left(
            \norms{\nabla \func{x^k,\bxi^k}}+\lambda
        \right)^2
    }
   \dotprod{\nabla \func{x^k,\bxi^k}}{x^k-x^*}
    \nonumber \\
    &=
    \norms{x^k-x^*}^2\nonumber\\
    &\quad\quad\quad-
    \frac{
        2\eta \dotprod{\nabla \func{x^k,\bxi^k}}{x^k-x^*}
    }{
        \norms{\nabla \func{x^k,\bxi^k}}+\lambda
    }
    \left(
        1
        -
        \eta
        \frac{
            \cL_0
            +
            \sqrt{\rho}\cL_1\norms{\nabla \func{x^k,\bxi^k}}
        }{
            \norms{\nabla \func{x^k,\bxi^k}}+\lambda
        }
    \right).
    \label{eq:NSGD_L0L1_monotone_main}
\end{align}

It remains to show that the term in parentheses is non-negative. For any $G\geq 0$,
\begin{align*}
    \frac{
        \cL_0+\sqrt{\rho}\cL_1\lambda
    }{\lambda}
    -
    \frac{
        \cL_0+\sqrt{\rho}\cL_1G
    }{G+\lambda}
    =
    \frac{
        \cL_0G+\sqrt{\rho}\cL_1\lambda^2
    }{
        \lambda(G+\lambda)
    }
    \geq
    0.
\end{align*}
Applying this inequality with $G=\norms{\nabla \func{x^k,\bxi^k}}$, and using the step-size condition, gives
\begin{equation*}
    1
    -
    \eta
    \frac{
        \cL_0+\sqrt{\rho}\cL_1\norms{\nabla \func{x^k,\bxi^k}}
    }{
        \norms{\nabla \func{x^k,\bxi^k}}+\lambda
    }
    \geq
    0.
\end{equation*}
Substituting this into~\eqref{eq:NSGD_L0L1_monotone_main}, as well as using $F_k(\bxi^k)
    \leq
    \dotprod{\nabla \func{x^k,\bxi^k}}{x^k-x^*}$ from \eqref{eq:convexity_batch_L0L1_NSGD} $F_k(\bxi^k)\geq 0$ from \eqref{eq:Fk_polozitelen_FOR_NSGD}, we conclude that
\begin{equation*}
    \norms{x^{k+1}-x^*}^2
    \leq
    \norms{x^k-x^*}^2.
\end{equation*}
Thus, by induction,
\begin{equation*}
    \norms{x^k-x^*}
    \leq
    \norms{x^0-x^*}
    =
    R_0
    \quad
    \text{for all } k\geq 0.
\end{equation*}

Finally, the same step-size condition implies
\begin{equation*}
    \norms{x^{k+1}-x^k}
    =
    \eta
    \norms{
        \frac{\nabla \func{x^k,\bxi^k}}
        {\norms{\nabla \func{x^k,\bxi^k}}+\lambda}
    }
    \leq
    \eta
    \leq
    \frac{1}{\sqrt{\rho}\cL_1},
\end{equation*}
which ensures that the NSGD update stays within the locality radius required by the corresponding $(L_0,L_1)$-smoothness condition with $L_1=\sqrt{\rho}\cL_1$.
This completes the proof.
\end{proof}

\subsection{Variance for mini-batch stochastic gradient}

If we use the mini-batch stochastic gradient $\nabla f(x, \bxi) = \frac{1}{B} \sum_{i = 1}^B \nabla f(x, \xi_i)$ instead of the stochastic gradient $\nabla \func{x, \xi}$, then the variance \eqref{eq:gradient_noise_variance} takes the following form:
    \begin{align}
            \expect{\norms{\nabla \func{x,\bxi} - \nabla \func{x}}^2} &\overset{\circledOne}{=} \expect{\norms{ \frac{1}{B} \sum_{i=1}^B \nabla \func{x,\xi_i} - \nabla \func{x}}^2} \nonumber\\
            &\overset{\circledTwo}{=} \expect{\norms{ \frac{1}{B} \sum_{i=1}^B \left( \nabla \func{x,\xi_i} - \nabla \func{x} \right)}^2}\nonumber\\
            &= \frac{1}{B^2}  \expect{\norms{ \sum_{i=1}^B \left( \nabla \func{x,\xi_i} - \nabla \func{x} \right)}^2}\nonumber\\
            &\overset{\circledThree}{=} \frac{1}{B^2}  \sum_{i=1}^B \sum_{j=1}^B\expect {\dotprod{ \nabla \func{x,\xi_i} - \nabla \func{x} }{\nabla \func{x,\xi_j} - \nabla \func{x}}}\nonumber\\
            &= \frac{1}{B^2}  \sum_{i=1}^B \expect{\norms{   \nabla \func{x,\xi_i} - \nabla \func{x} }^2} \nonumber\\
            &\quad\quad\quad+  \frac{1}{B^2} \sum_{i \neq j}^B \dotprod{ \expect{\nabla \func{x,\xi_i}} - \nabla \func{x} }{\expect{\nabla \func{x,\xi_j}} - \nabla \func{x}} \nonumber\\
            &\overset{\circledFour}{=}  \frac{1}{B^2} \sum_{i=1}^B \expect{\norms{   \nabla \func{x,\xi_i} - \nabla \func{x} }^2}\nonumber\\
            &\overset{\eqref{eq:gradient_noise_variance}}{\leq} \frac{1}{B^2} \sum_{i=1}^B \left[ (\rho - 1) \norms{\nabla \func{x}}^2 + \sigma^2 \right]\nonumber\\
            &= \frac{\rho - 1}{B} \norms{\nabla \func{x}}^2 + \frac{\sigma^2}{B},\label{eq:variance_with_batch_size}
        \end{align}
        where in $\circledOne$ we used $\nabla \func{x, \bxi} = \frac{1}{B} \sum_{i =1}^B \nabla \func{x, \xi_i}$, in $\circledTwo$ we used $\nabla \func{x} = \frac{1}{B} \sum_{i =1}^B \nabla \func{x}$, in $\circledThree$ we utilized the definition of the Euclidean norm, in $\circledFour$ we used that $\expect{\nabla \func{x, \xi_i}} = \nabla \func{x}$.

\subsection{Motivation for the generalized strong growth condition under heavy-tailed noise}\label{app:Motivation_heavy}

\begin{boxC}
\begin{example}
    Consider the function $f(x) = \frac{1}{2} \norms{x}^2$ with gradient oracle $\nabla \func{x,\xi} = \left(Z + 1\right) x$ (where $Z$ is a random variable with a symmetric Pareto distribution). Then, for $\alpha \in (p,2]$ we obtain the strong growth condition under heavy-tailed noise $\expect{\norms{\nabla f(x, \xi)}^p} \leq 2^{p-1} \rho \norms{\nabla f(x)}^p$ with $\rho = \frac{\alpha}{\alpha - p} + 1 < \infty$,~and~${\expect{\norms{\nabla f(x, \xi)}^2} = \infty}$.
\end{example}
\end{boxC}
\begin{proof}
    We start by estimating the $p$-th moment $\expect{\norms{\nabla \func{x, \xi}}^p} \overset{\circledOne}{=} \expect{\norms{Zx + x}^p}$:
    \begin{align*}
         \expect{\norms{Zx + x}^p} &\overset{\circledTwo}{\leq}  2^{p-1}\left(\expect{|Z|^p}+ 1 \right) \norms{x}^p \overset{\circledThree}{=} 2^{p-1} \left( \alpha \int_{1}^{\infty} |z|^{p-\alpha-1} dz + 1 \right) \norms{x}^p \\
         &\overset{\circledFour}{=} 2^{p-1} \underbrace{\left(\frac{ \alpha}{\alpha - p} + 1 \right)}_{=:\rho < \infty} \norms{\nabla \func{x}}^p,
    \end{align*}
    where in $\circledOne$ we used $\nabla \func{x,\xi} = \left(Z + 1\right) x$, in $\circledTwo$ we applied Jensen's inequality for finite form, in $\circledThree$ we applied the definition of expectation with probability density function of a symmetric Pareto random variable $f_Z(z) = \frac{\alpha}{2} |z|^{-(\alpha + 1)}$ with $|z| \geq 1$ and $|z+1| \leq 2 |z|$ , in $\circledFour$ we noted that $\int_{1}^{\infty} |z|^{p - \alpha - 1}  dz =  \frac{1}{\alpha - p} < \infty$ since $p < \alpha$, and used $\nabla \func{x} = x$.

    To analyze the second moment, we follow the same steps, substituting $p$ with $2$. In this case, we obtain $\int_{1}^{\infty} |z|^{2 - \alpha - 1}  dz = \infty$, since $\alpha \leq (p,2]$. Consequently, the second moment is unbounded: $\expect{\norms{\nabla \func{x, \xi}}^2} = \infty$.
    
\end{proof}

\section{Missing Proofs for Stochastic Part (ClipSGD \texorpdfstring{$\&$}{TEXT} NSGD)}

\subsection{Missing proofs for clipped stochastic gradient descent method}

    This section provides the proofs for ClipSGD in the non-convex and convex settings. The analysis follows the same general line as \cite{Koloskova_2023}, but the stochastic noise is controlled through the generalized strong growth condition and a mini-batch gradient. The main proof ingredient is an alignment argument: once the batch size is proportional to the growth parameter $\rho$, the clipped stochastic direction remains sufficiently aligned with the full gradient.

    \subsubsection{Non-convex setup under \texorpdfstring{$(L_0,L_1)$}{TEXT}-smooth and generalized strong growth conditions}\label{app:NC_ClipSGD}
 
        The proof is divided into two regimes: $c \leq  \frac{6 \sqrt{2} \sigma}{\sqrt{B}}$ and $c >  \frac{ 6 \sqrt{2} \sigma}{\sqrt{B}}$. In each regime, we further split according to the deterministic gradient norm $\norms{\nabla \func{x^k}}$.

\begin{proof}
Throughout the proof, we use the shorthand
\[
        F_k := \func{x^k}-f^*.
    \]
In the one-step estimates below, expectations are taken with respect to the fresh mini-batch $\bxi^k$ and are written explicitly as conditional expectations given the current iterate $x^k$. After taking full expectation, we return to the notation $\expect{\cdot}$.

        We start from the descent inequality for $(L_0,L_1)$-smooth functions:
        \begin{align}
            \func{x^{k+1}} - \func{x^k} &\overset{\circledOne}{\leq} \dotprod{\nabla \func{x^k}}{x^{k+1} - x^k} + \frac{L_0 + L_1 \norms{\nabla \func{x^k}}}{2} \norms{x^{k+1} - x^{k}}^2 \nonumber \\
            &\overset{\circledTwo}{=} - \eta \dotprod{\nabla \func{x^k}}{\clip{\nabla \func{x^k, \bxi^k}}} \nonumber\\
            &\quad \quad \quad+ \frac{\eta^2 \left(L_0 + L_1 \norms{\nabla \func{x^k}}\right)}{2} \norms{\clip{\nabla \func{x^k, \bxi^k}}}^2, 
            \label{eq:NC_clipSGD_smooth}
        \end{align}
        Here, in~$\circledOne$ we used the descent inequality~\eqref{eq:L0L1_descent_inequality}; its locality condition $\norms{x^{k+1}-x^k}\leq \frac{1}{L_1}$ is verified in Remark~\ref{rem:Step_size_nc_ClipSGD}. In $\circledTwo$ we used the ClipSGD update $x^{k+1} = x^k - \eta \cdot \clip{\nabla \func{x^k, \bxi^k}}$.

        We now split the proof into two regimes.
        
        \paragraph{First regime: $c \leq  \frac{6 \sqrt{2} \sigma}{\sqrt{B}}$.} We consider two cases: ${\norms{\nabla \func{x^k}} \geq \frac{6 \sqrt{2} \sigma}{\sqrt{B}}}$ and ${\norms{\nabla \func{x^k}} < \frac{6 \sqrt{2} \sigma}{\sqrt{B}}}$.

        \fbox{1.1) The case ${\norms{\nabla \func{x^k}} \geq \frac{6 \sqrt{2} \sigma}{\sqrt{B}}}$.} 

        Define the bad event
        \[
            \delta = \mathbbm{1} {\left\{  \norms{\nabla \func{x^k,\bxi^k} - \nabla \func{x^k}} > 3 \sqrt{\frac{\rho - 1}{B} \norms{\nabla \func{x^k}}^2+ \frac{\sigma^2}{B}} \right\}}.
        \]
        We decompose the first term in~\eqref{eq:NC_clipSGD_smooth} according to this event:
        \begin{align}
            \expectcond{- \eta \dotprod{\nabla \func{x^k}}{\clip{\nabla \func{x^k, \bxi^k}}}}{x^k} &\overset{\circledast}{=}  \expectcond{- \eta \alpha_{\bxi^k} \dotprod{\nabla \func{x^k}}{\nabla \func{x^k, \bxi^k}}}{x^k}\nonumber\\
            &\hspace{-8em}\leq \mathbb{P}\left(\delta=0\mid x^k\right) \underbrace{\expectcond{- \eta \alpha_{\bxi^k} \dotprod{\nabla \func{x^k}}{\nabla \func{x^k, \bxi^k}}}{x^k,\delta=0}}_{=: T_1} \nonumber\\
            &\hspace{-8em}\quad\quad\quad+ \mathbb{P}\left(\delta=1\mid x^k\right) \underbrace{\expectcond{- \eta \alpha_{\bxi^k} \dotprod{\nabla \func{x^k}}{\nabla \func{x^k, \bxi^k}}}{x^k,\delta=1}}_{=: T_2},\label{eq:NC_clipSGD_T1_and_T2_FIRST_part}
        \end{align}
        where in $\circledast $ we used $\clip{\nabla \func{x^k, \bxi^k}} = \alpha_{\bxi^k} \cdot \nabla \func{x^k, \bxi^k}$. We first estimate the probability of the bad event. By Markov's inequality,
        \begin{align}
            \mathbb{P}\left(\delta=1\mid x^k\right) &= \mathbb{P}\left(\norms{\nabla \func{x^k,\bxi^k} - \nabla \func{x^k}}^2 > 9 \left(\frac{\rho - 1}{B} \norms{\nabla \func{x^k}}^2 +\frac{\sigma^2}{B}\right)\mid x^k\right)\nonumber\\
            &\leq \frac{\expectcond{\norms{\nabla \func{x^k,\bxi^k} - \nabla \func{x^k}}^2}{x^k}}{9 (\rho - 1) \norms{\nabla \func{x^k}}^2 + \sigma^2} \cdot B \nonumber\\
            &\overset{\eqref{eq:variance_with_batch_size}}{\leq} \frac{1}{9}.
            \label{eq:NC_clipSGD_prob1_FIRST_part}
        \end{align}
        Consequently, $\mathbb{P}\left(\delta=0\mid x^k\right) = 1 - \mathbb{P}\left(\delta=1\mid x^k\right) \geq \frac{8}{9}$. We next bound the contribution of the good event:
        \begin{align}
            &\expectcond{- \eta \alpha_{\bxi^k} \dotprod{\nabla \func{x^k}}{\nabla \func{x^k, \bxi^k}}}{x^k,\delta=0} \nonumber\\
            &=  \expectcond{- \eta \alpha_{\bxi^k} \dotprod{\nabla \func{x^k}}{\nabla \func{x^k, \bxi^k} \pm \nabla \func{x^k}}}{x^k,\delta=0} \nonumber\\
            &=  \expectcond{- \eta \alpha_{\bxi^k} \norms{\nabla \func{x^k}}^2 - \eta \alpha_{\bxi^k} \dotprod{\nabla \func{x^k}}{\nabla \func{x^k, \bxi^k} - \nabla \func{x^k}}}{x^k,\delta=0}\nonumber\\
            &\overset{\eqref{eq:scalar_product_bound}}{\leq}  \expectcond{- \eta \alpha_{\bxi^k} \norms{\nabla \func{x^k}}^2 + \eta \alpha_{\bxi^k} \norms{\nabla \func{x^k}}\norms{\nabla \func{x^k, \bxi^k} - \nabla \func{x^k}}}{x^k,\delta=0}\nonumber\\
            &\overset{\circledOne}{\leq} \expectcond{- \eta \alpha_{\bxi^k} \norms{\nabla \func{x^k}}^2 + 3 \eta \alpha_{\bxi^k} \norms{\nabla \func{x^k}} \sqrt{\frac{\rho - 1}{B} \norms{\nabla \func{x^k}}^2 + \frac{\sigma^2}{B}}}{x^k,\delta=0}\nonumber\\
            &\overset{\circledTwo}{\leq} \expectcond{- \eta \alpha_{\bxi^k} \norms{\nabla \func{x^k}}^2 + 3 \eta \alpha_{\bxi^k}  \sqrt{\frac{\rho - 1}{B} + \frac{1}{72}} \cdot \norms{\nabla \func{x^k}}^2}{x^k,\delta=0}\nonumber\\
            &\overset{\circledThree}{\leq} \expectcond{- \eta \alpha_{\bxi^k} \norms{\nabla \func{x^k}}^2 + 3\eta \alpha_{\bxi^k} \sqrt{\frac{1}{72} + \frac{1}{72}} \cdot \norms{\nabla \func{x^k}}^2}{x^k,\delta=0}\nonumber\\
            &= \expectcond{- \eta \alpha_{\bxi^k} \norms{\nabla \func{x^k}}^2 + \frac{\eta \alpha_{\bxi^k}}{2} \norms{\nabla \func{x^k}}^2}{x^k,\delta=0}\nonumber\\
            &=\expectcond{- \frac{\eta \alpha_{\bxi^k}}{2}  \norms{\nabla \func{x^k}}^2}{x^k,\delta=0}\nonumber\\
            &\overset{\circledFour}{\leq} - \frac{\eta c}{4}  \norms{\nabla \func{x^k}},
            \label{eq:NC_clipSGD_T1_FIRST_part}
        \end{align}
        Here, in $\circledOne$ we used the definition of the event $\{\delta=0\}$, in $\circledTwo$ we used $\frac{\sigma}{\sqrt{B}} \leq \frac{\norms{\nabla \func{x^k}}}{6 \sqrt{2}}$, and in $\circledThree$ we used $B \geq 72 (\rho - 1)$. In $\circledFour$ we used the bound $\norms{\nabla \func{x^k, \bxi^k}}\leq 2\norms{\nabla \func{x^k}}$, which implies $\alpha_{\bxi^k}\geq c/(2\norms{\nabla \func{x^k}})$. On the complementary event, we only use the boundedness induced by clipping:
        \begin{align}
            \expectcond{- \eta \alpha_{\bxi^k} \dotprod{\nabla \func{x^k}}{\nabla \func{x^k, \bxi^k}}}{x^k,\delta=1} \overset{\eqref{eq:scalar_product_bound}}{\leq} &\eta \norms{\nabla \func{x^k}} \expectcond{ \alpha_{\bxi^k} \norms{\nabla \func{x^k, \bxi^k}}}{x^k,\delta=1} \nonumber\\
            &\overset{\circledOne}{\leq} \eta c \norms{\nabla \func{x^k}},\label{eq:NC_clipSGD_T2_FIRST_part}
        \end{align}
        where in $\circledOne$ we used $\norms{\clip{\nabla \func{x^k, \bxi^k}}} \leq c$.

        Combining the good and bad events yields
        \begin{align}
            &\expectcond{- \eta \dotprod{\nabla \func{x^k}}{\clip{\nabla \func{x^k, \bxi^k}}}}{x^k} \overset{\eqref{eq:NC_clipSGD_T1_and_T2_FIRST_part}}{\leq} \mathbb{P}\left(\delta=0\mid x^k\right) \underbrace{\expectcond{- \eta \alpha_{\bxi^k} \dotprod{\nabla \func{x^k}}{\nabla \func{x^k, \bxi^k}}}{x^k,\delta=0}}_{=: T_1} \nonumber\\
            &\quad\quad\quad+ \mathbb{P}\left(\delta=1\mid x^k\right) \underbrace{\expectcond{- \eta \alpha_{\bxi^k} \dotprod{\nabla \func{x^k}}{\nabla \func{x^k, \bxi^k}}}{x^k,\delta=1}}_{=: T_2}\nonumber\\
            &\overset{\eqref{eq:NC_clipSGD_T1_FIRST_part}, \eqref{eq:NC_clipSGD_T2_FIRST_part}}{\leq}
            -\frac{\eta c}{4} \norms{\nabla \func{x^k}}\cdot \mathbb{P}\left(\delta=0\mid x^k\right)  + \eta c \norms{\nabla \func{x^k}}\cdot \mathbb{P}\left(\delta=1\mid x^k\right)\nonumber\\
            &= -\frac{\eta c}{4} \norms{\nabla \func{x^k}}\cdot \left( 1 -  \mathbb{P}\left(\delta=1\mid x^k\right)\right)  + \eta c \norms{\nabla \func{x^k}}\cdot \mathbb{P}\left(\delta=1\mid x^k\right)\nonumber\\
            &\overset{\eqref{eq:NC_clipSGD_prob1_FIRST_part}}{\leq} -\eta c \left( \frac{8}{9} \cdot\frac{1}{4} -  \frac{1}{9} \right)\norms{\nabla \func{x^k}}\nonumber\\
            &= -\frac{\eta c}{9} \norms{\nabla \func{x^k}}.
            \label{eq:NC_clipSGD_first_term_smooth_FIRST_part}
        \end{align}
        Substituting this estimate into \eqref{eq:NC_clipSGD_smooth}, we obtain:
        \begin{align}
            \expectcond{F_{k+1}}{x^k} - F_k &\overset{\eqref{eq:NC_clipSGD_smooth}}{\leq} 
            - \expectcond{\eta \dotprod{\nabla \func{x^k}}{\clip{\nabla \func{x^k, \bxi^k}}}}{x^k} \nonumber\\
            &\quad\quad\quad+ \frac{\eta^2 \left(L_0 + L_1 \norms{\nabla \func{x^k}} \right)}{2} \expectcond{\norms{\clip{\nabla \func{x^k, \bxi^k}}}^2}{x^k}\nonumber\\
            &\overset{\eqref{eq:NC_clipSGD_first_term_smooth_FIRST_part}}{\leq} -\frac{\eta c}{9} \norms{\nabla \func{x^k}} + \frac{\eta^2 \left(L_0 + L_1 \norms{\nabla \func{x^k}} \right)}{2} \expectcond{\norms{\clip{\nabla \func{x^k, \bxi^k}}}^2}{x^k}\nonumber\\
            &\overset{\circledOne}{\leq} -\frac{\eta c}{9} \norms{\nabla \func{x^k}} + \frac{\eta^2 \left(L_0 + L_1 \norms{\nabla \func{x^k}} \right)}{2} \cdot c^2\nonumber\\
            &\overset{\circledTwo}{\leq} -\frac{\eta c}{9} \norms{\nabla \func{x^k}} + \frac{\eta^2 \left(L_0 + L_1 c \right)}{2} \cdot c \cdot \norms{\nabla \func{x^k}} \nonumber\\
            &\overset{\circledThree}{\leq} -\frac{\eta c}{9} \norms{\nabla \func{x^k}} + \frac{\eta c}{18} \norms{\nabla \func{x^k}} \nonumber\\
            &= -\frac{\eta c}{18} \norms{\nabla \func{x^k}}, \nonumber
        \end{align}
        where in $\circledOne$ we used $\norms{\clip{\nabla \func{x^k, \bxi^k}}} \leq c$, in $\circledTwo$ we used $\norms{\nabla \func{x^k}} \geq c$, in $\circledThree$ we chose $\eta \leq \left[9 \left( L_0 + L_1 c  \right)  \right]^{-1}$.

        Hence, in the first case,
        \begin{equation}
            \norms{\nabla \func{x^k}} \leq \frac{18 \left( F_k - \expectcond{F_{k+1}}{x^k}  \right)}{\eta c}.
            \label{eq:NC_clipSGD_case1_Itog_FIRST_part}
        \end{equation}

        \fbox{1.2) The case ${\norms{\nabla \func{x^k}} < \frac{6 \sqrt{2} \sigma}{\sqrt{B}}}$.} 

        This case is immediate from the case condition:
        \begin{equation*}
            \norms{\nabla \func{x^k}} \leq \frac{6 \sqrt{2} \sigma}{\sqrt{B}}.
        \end{equation*}

        \paragraph{Combining the first regime.}
        Since both terms on the right-hand side are non-negative, the bounds from the two cases imply, for all $k\in[0,N-1]$,
        \begin{equation*}
            \norms{\nabla \func{x^k}} \leq \frac{18 \left( F_k - \expectcond{F_{k+1}}{x^k}  \right)}{\eta c} + \frac{6 \sqrt{2} \sigma}{\sqrt{B}}.
        \end{equation*}
        Taking full expectation, summing over all indices $k \in [0,N-1]$, and using telescoping of $\expect{F_k}$ gives
        \begin{equation}
            \min_{k \in [0,N-1]} \expect{\norms{\nabla \func{x^k}}} \leq \frac{1}{N} \sum_{k =0}^{N-1} \expect{\norms{\nabla \func{x^k}}} \leq \frac{18 F_0}{\eta c N} + \frac{6 \sqrt{2} \sigma}{\sqrt{B}}. 
            \label{eq:NC_clipSGD_ITOG_First_Part}
        \end{equation}

        \paragraph{Second regime: $c >  \frac{6 \sqrt{2} \sigma}{\sqrt{B}}$.} We now consider two cases: ${\norms{\nabla \func{x^k}} \geq c}$ and ${\norms{\nabla \func{x^k}} < c}$.

        \fbox{2.1) The case ${\norms{\nabla \func{x^k}} \geq c}$.} 

        This case is covered by the alignment argument used in \fbox{1.1)}. Indeed, the condition of \fbox{1.1)} is
        \begin{equation*}
            \norms{\nabla \func{x^k}} \geq \frac{6 \sqrt{2} \sigma}{\sqrt{B}} \geq c,
        \end{equation*}
        while the present case satisfies
        \begin{equation*}
            \norms{\nabla \func{x^k}} \geq c > \frac{6 \sqrt{2} \sigma}{\sqrt{B}}.
        \end{equation*}
        Thus, the same proof applies without repetition.
        
        Therefore,
        \begin{equation}
           c \cdot \norms{\nabla \func{x^k}} \overset{\eqref{eq:NC_clipSGD_case1_Itog_FIRST_part}}{\leq} \frac{18 \left( F_k - \expectcond{F_{k+1}}{x^k}  \right)}{\eta}.
            \label{eq:NC_clipSGD_case1_Itog_SECOND_part}
        \end{equation}

        \fbox{2.2) The case ${\norms{\nabla \func{x^k}} < c}$.}

        In this case, $\nabla \func{x^k} = \clip{\nabla \func{x^k}}$. From \eqref{eq:NC_clipSGD_smooth}, we have:
        \begin{align*}
            \expectcond{\func{x^{k+1}}}{x^k} - \func{x^k} &\overset{\eqref{eq:NC_clipSGD_smooth}}{\leq} - \eta \expectcond{\dotprod{\nabla \func{x^k}}{\clip{\nabla \func{x^k, \bxi^k}}}}{x^k} \\
            &\quad\quad\quad+ \frac{\eta^2 \left(L_0 + L_1 \norms{\nabla \func{x^k}} \right)}{2} \expectcond{\norms{\clip{\nabla \func{x^k, \bxi^k}}}^2}{x^k}\\
            &\overset{\circledOne}{\leq} - \eta \expectcond{\dotprod{\nabla \func{x^k}}{\clip{\nabla \func{x^k, \bxi^k}}}}{x^k} \\
            &\quad\quad\quad+ \frac{\eta^2 \left(L_0 + L_1 c \right)}{2} \expectcond{\norms{\clip{\nabla \func{x^k, \bxi^k}}}^2}{x^k}\\
            &\overset{\circledTwo}{=} -\frac{\eta}{2} \norms{\nabla \func{x^k}}^2 - \frac{\eta}{2} \expectcond{\norms{\clip{\nabla \func{x^k, \bxi^k}}}^2}{x^k} \\
            &\quad\quad\quad+ \frac{\eta}{2} \expectcond{\norms{\clip{\nabla \func{x^k, \bxi^k}} - \nabla \func{x^k}}^2}{x^k}\\
            &\quad \quad \quad+ \frac{\eta^2 \left(L_0 + L_1 c \right)}{2} \expectcond{\norms{\clip{\nabla \func{x^k, \bxi^k}}}^2}{x^k}\\
            &\overset{\circledThree}{=} -\frac{\eta}{2} \norms{\nabla \func{x^k}}^2 + \frac{\eta}{2} \expectcond{\norms{\clip{\nabla \func{x^k, \bxi^k}} - \clip{\nabla \func{x^k}}}^2}{x^k}\\
            &\quad \quad \quad- \frac{\eta}{2} \left( 1 -\eta \left(L_0 + L_1 c \right) \right) \expectcond{\norms{\clip{\nabla \func{x^k, \bxi^k}}}^2}{x^k}\\
            &\overset{\circledFour}{\leq} -\frac{\eta}{2} \norms{\nabla \func{x^k}}^2 + \frac{\eta}{2} \expectcond{\norms{\clip{\nabla \func{x^k, \bxi^k}} - \clip{\nabla \func{x^k}}}^2}{x^k},
        \end{align*}
        where in $\circledOne$ we used $\norms{\nabla \func{x^k}} < c$, in $\circledTwo$ we used inequality \eqref{eq:quadratic_difference} with $a =  \nabla \func{x^{k}}$ and $b =  \clip{\nabla \func{x^k, \bxi^k}}$, in $\circledThree$ we used $\nabla \func{x^k} = \clip{\nabla \func{x^k}}$, in $\circledFour$ we chose $\eta \leq \left[9 \left( L_0 + L_1 c  \right)  \right]^{-1}$.

        Since clipping is the projection onto the Euclidean ball of radius $c$, it is a Lipschitz operator with Lipschitz constant $1$. Therefore,
        \begin{align}
            \expectcond{\func{x^{k+1}}}{x^k} - \func{x^k} &\leq -\frac{\eta}{2} \norms{\nabla \func{x^k}}^2 + \frac{\eta}{2} \expectcond{\norms{\nabla \func{x^k, \bxi^k} - \nabla \func{x^k}}^2}{x^k}\nonumber\\
            &\overset{\eqref{eq:variance_with_batch_size}}{\leq} -\frac{\eta}{2} \norms{\nabla \func{x^k}}^2 + \frac{\eta }{2} \left(  \frac{(\rho - 1)}{B} \norms{\nabla \func{x^k}}^{2} + \frac{\sigma^2}{B} \right) \nonumber\\
            &\overset{\circledOne}{\leq}  -\frac{\eta}{2} \norms{\nabla \func{x^k}}^2 + \frac{\eta}{144} \norms{\nabla \func{x^k}}^{2} + \frac{\eta\sigma^2}{2 B}\nonumber\\
            &\leq  -\frac{\eta}{2} \norms{\nabla \func{x^k}}^2 + \frac{\eta}{9} \norms{\nabla \func{x^k}}^{2} + \frac{\eta\sigma^2}{2 B}\nonumber\\
            &=  -\frac{7\eta}{18} \norms{\nabla \func{x^k}}^2 + \frac{\eta\sigma^2}{2 B}\nonumber\\
            &\leq  -\frac{\eta}{18} \norms{\nabla \func{x^k}}^2 + \frac{\eta\sigma^2}{2 B}\nonumber,
        \end{align}
        where in $\circledOne$ we chose $B \geq 72 (\rho - 1)$. 

        Hence, in the second case,
        \begin{equation}
            \norms{\nabla \func{x^k}}^2 \leq \frac{18 \left( F_k - \expectcond{F_{k+1}}{x^k} \right)}{\eta} + \frac{9 \sigma^2}{B},
            \label{eq:NC_clipSGD_case2_Itog_SECOND_part}
        \end{equation}

        \paragraph{Combining the second regime.}
        Define $\cT_1 := \left\{ k \in [0,N-1] \;| \;\norms{\nabla \func{x^k}} \geq c \right\}$ and $\cT_2 := [0,N-1] \setminus \cT_1$. Taking full expectation and summing the corresponding bounds over $k=0,\ldots,N-1$, we obtain:
        \begin{align*}
            \frac{1}{N} &\left( \sum_{k \in \cT_1} c \cdot \expect{\norms{\nabla \func{x^k}}} + \sum_{k \in \cT_2} \expect{\norms{\nabla \func{x^k}}^2} \right) \\
            &\overset{\eqref{eq:NC_clipSGD_case1_Itog_SECOND_part}, \eqref{eq:NC_clipSGD_case2_Itog_SECOND_part}}{\leq} \frac{1}{N} \left(\sum_{k \in \cT_1} \frac{18 \left( \expect{F_k} - \expect{F_{k+1}} \right)}{\eta} +  \sum_{k \in \cT_2} \frac{18 \left( \expect{F_k} - \expect{F_{k+1}} \right)}{\eta} + \sum_{k \in \cT_2} \frac{9 \sigma^2}{B} \right)\\
            & = \frac{1}{N} \left(  \sum_{k=0}^{N-1} \frac{18 \left( \expect{F_k} - \expect{F_{k+1}} \right)}{\eta} + \sum_{k \in \cT_2} \frac{9 \sigma^2}{B} \right)\\
            &\leq  \frac{18  F_0 }{\eta N} +  \frac{9 \sigma^2}{B}.
        \end{align*}
        In particular,
        \begin{equation*}
            \frac{1}{N} \sum_{k \in \cT_1} c \cdot \expect{\norms{\nabla \func{x^k}}} \leq \frac{18 F_0}{\eta N} +  \frac{9 \sigma^2}{B},
        \end{equation*}
        and
        \begin{equation*}
            \frac{1}{N} \sum_{k \in \cT_2}  \expect{\norms{\nabla \func{x^k}}^2} \leq \frac{18 F_0}{\eta N} +  \frac{9 \sigma^2}{B},
        \end{equation*}
        Using $x^2 \geq 2 \epsilon x - \epsilon^2$ for any $\epsilon,x > 0$, and defining
        \[
            A = \frac{18 F_0}{\eta N} +  \frac{18 \sigma^2}{B},
        \]
        we get:
        \begin{equation*}
           \frac{1}{N} \sum_{k \in \cT_2} \left( 2 \epsilon \expect{\norms{\nabla \func{x^k}}} - \epsilon^2  \right) \leq \frac{1}{N} \sum_{k \in \cT_2}  \expect{\norms{\nabla \func{x^k}}}^2 \overset{\circledOne}{\leq} \frac{1}{N} \sum_{k \in \cT_2}  \expect{\norms{\nabla \func{x^k}}^2} \leq A,
        \end{equation*}
        where in $\circledOne$ we applied Jensen's inequality, and thus we have
        \begin{equation*}
            \frac{1}{N} \sum_{k \in \cT_2}  \expect{\norms{\nabla \func{x^k}}} \leq \frac{A}{2\epsilon} + \frac{\epsilon}{2}.
        \end{equation*}
        Choosing $\epsilon = \sqrt{A}$, we get
        \begin{equation*}
            \frac{1}{N} \sum_{k \in \cT_2}  \expect{\norms{\nabla \func{x^k}}} \leq \sqrt{A} \leq \sqrt{\frac{18 F_0}{\eta N} +  \frac{18 \sigma^2}{B}} \overset{\circledOne}{\leq} \sqrt{\frac{18 F_0}{\eta N}} + \frac{6 \sqrt{2} \sigma}{\sqrt{B}},
        \end{equation*}
        where in $\circledOne$ we used $\sqrt{a+b} \leq \sqrt{a} + \sqrt{b}$ for $a,b \geq 0$.
        
        Combining the contributions of $\cT_1$ and $\cT_2$ gives:
        \begin{align}
            \min_{k \in [0,N-1]} \expect{ \norms{\nabla \func{x^k}}} &\leq \frac{1}{N} \sum_{k =0}^{N-1} \expect{ \norms{\nabla \func{x^k}}}  \nonumber\\
            &= \frac{1}{N} \left(  \sum_{k \in \cT_1}  \expect{ \norms{\nabla \func{x^k}}} +  \sum_{k \in \cT_2} \expect{\norms{\nabla \func{x^k}}} \right)\nonumber\\
            &\leq \sqrt{\frac{18 F_0}{\eta N}}  +\frac{18 F_0}{\eta c N} + \frac{6 \sqrt{2} \sigma}{\sqrt{B}} +  \frac{18 \sigma^2}{2cB}.
            \label{eq:NC_clipSGD_ITOG_SECOND_Part}
        \end{align}

        The second-regime bound \eqref{eq:NC_clipSGD_ITOG_SECOND_Part} also covers the first-regime bound \eqref{eq:NC_clipSGD_ITOG_First_Part}. This proves the stated convergence rate for step size $\eta \leq \left[9 \left( L_0 + L_1 c  \right)  \right]^{-1}$ and batch size $B \geq 72 (\rho - 1)$.
\end{proof}

        \begin{remark}
            \label{rem:Step_size_nc_ClipSGD}
            In the first step of~\eqref{eq:NC_clipSGD_smooth}, we use the descent lemma for $(L_0,L_1)$-smooth functions~\eqref{eq:L0L1_descent_inequality} \citep[see Appendix~A.1 in][]{Zhang_2020_Improved}; the condition $\norms{x^{k+1}-x^k}\leq 1/L_1$ required for its application is guaranteed by the chosen step size $\eta \leq \left[9 \left( L_0 + L_1 c  \right)  \right]^{-1}$:
            \begin{equation*}
                \norms{x^{k+1}-x^k} \overset{\circledOne}{=} \eta \norms{\clip{\nabla \func{x^k, \bxi^k}}} \overset{\circledTwo}{\leq} \eta c \overset{\circledThree}{\leq} \frac{c}{9 \left( L_0 + L_1 c  \right)} \overset{\circledFour}{\leq} \frac{c}{9 L_1 c} \overset{\circledFour}{\leq} \frac{1}{L_1},
            \end{equation*}
            where in $\circledOne$ we used $x^{k+1} = x^k - \eta \cdot \clip{\nabla \func{x^k, \bxi^k}}$, in $\circledTwo$ we used $\norms{\clip{\nabla \func{x^k, \bxi^k}}} = \norms{\min\left\{1, \frac{c}{\norms{\nabla \func{x^k, \bxi^k}}} \right\} \nabla \func{x^k, \bxi^k}} \leq c$, in $\circledThree$ we used $\eta \leq \left[9 \left( L_0 + L_1 c  \right)  \right]^{-1}$, where in $\circledFour$ we used the fact that, for $0<b\leq a$, one has $\frac{1}{a}\leq \frac{1}{b}$.
        \end{remark}

    \subsubsection{Convex setup under \texorpdfstring{$(\cL_0,\cL_1)$}{TEXT}-smooth and strong growth conditions}\label{app:Convex_ClipSGD}

       We now prove the convex ClipSGD result. The proof follows the same alignment argument as in the non-convex case, but convexity allows us to convert gradient-norm bounds into a recursion for the optimality gap.
    
\begin{proof}
Throughout the proof, we use the shorthand
\[
        F_k := \func{x^k}-f^*.
    \]
In the one-step estimates below, expectations are taken with respect to the fresh mini-batch $\bxi^k$ and are written explicitly as conditional expectations given the current iterate $x^k$. After taking full expectation, we return to the notation $\expect{\cdot}$.
By convexity, the Cauchy--Schwarz inequality, and the monotonicity of the distance to the optimum from Lemma~\ref{lem:monotone_clipsgd}, we have
\begin{equation}
    F_k
    \leq
    \dotprod{\nabla \func{x^k}}{x^k-x^*}
    \leq
    \norms{\nabla \func{x^k}}\norms{x^k-x^*}
    \leq
    R_0\norms{\nabla \func{x^k}}.
    \label{eq:function_gap_gradient_relation_ClipSGD}
\end{equation}

       We start from the descent inequality for $(L_0,L_1)$-smooth functions with constants $L_0 = \cL_0$ and $L_1 = \sqrt{\rho} \cL_1$ (see \eqref{eq:new_smoothness_with_rho} in Section~\ref{subsec:generalized_smoothness_conditions}):
        \begin{align}
            \func{x^{k+1}} - \func{x^k} &\overset{\circledOne}{\leq} \dotprod{\nabla \func{x^k}}{x^{k+1} - x^k} + \frac{\cL_0 + \sqrt{\rho} \cL_1 \norms{\nabla \func{x^k}}}{2} \norms{x^{k+1} - x^{k}}^2 \nonumber \\
            &\overset{\circledTwo}{=} - \eta \dotprod{\nabla \func{x^k}}{\clip{\nabla \func{x^k, \bxi^k}}} \nonumber\\
            &\quad \quad \quad+ \frac{\eta^2 \left(\cL_0 + \sqrt{\rho} \cL_1 \norms{\nabla \func{x^k}}\right)}{2} \norms{\clip{\nabla \func{x^k, \bxi^k}}}^2 , \label{eq:Convex_clipSGD_smooth}
        \end{align}
        Here, in~$\circledOne$ we used the descent inequality~\eqref{eq:L0L1_descent_inequality}; its locality condition $\norms{x^{k+1}-x^k}\leq \frac{1}{\sqrt{\rho} \cL_1}$ is verified in Remark~\ref{rem:Step_size_convex_ClipSGD}. In $\circledTwo$ we used the ClipSGD update $x^{k+1} = x^k - \eta \cdot \clip{\nabla \func{x^k, \bxi^k}}$.

        We now consider two cases depending on the deterministic gradient norm.

        \fbox{1) The case $\norms{\nabla \func{x^k}} \geq c$.} 

        Define the bad event
        \[
            \delta = \mathbbm{1}_{\left\{  \norms{\nabla \func{x^k,\bxi^k} - \nabla \func{x^k}} > 3 \sqrt{\frac{\rho - 1}{B}} \norms{\nabla \func{x^k}} \right\}}.
        \]
        We decompose the first term in~\eqref{eq:Convex_clipSGD_smooth} according to this event:
        \begin{align}
            \expectcond{- \eta \dotprod{\nabla \func{x^k}}{\clip{\nabla \func{x^k, \bxi^k}}}}{x^k} &\overset{\circledast}{=}  \expectcond{- \eta \alpha_{\bxi^k} \dotprod{\nabla \func{x^k}}{\nabla \func{x^k, \bxi^k}}}{x^k}\nonumber\\
            &\hspace{-8em}\leq \mathbb{P}\left(\delta=0\mid x^k\right) \underbrace{\expectcond{- \eta \alpha_{\bxi^k} \dotprod{\nabla \func{x^k}}{\nabla \func{x^k, \bxi^k}}}{x^k,\delta=0}}_{=: T_1} \nonumber\\
            &\hspace{-8em}\quad\quad\quad+ \mathbb{P}\left(\delta=1\mid x^k\right) \underbrace{\expectcond{- \eta \alpha_{\bxi^k} \dotprod{\nabla \func{x^k}}{\nabla \func{x^k, \bxi^k}}}{x^k,\delta=1}}_{=: T_2},\label{eq:Convex_clipSGD_T1_and_T2}
        \end{align}
        where in $\circledast $ we used $\clip{\nabla \func{x^k, \bxi^k}} = \alpha_{\bxi^k} \cdot \nabla \func{x^k, \bxi^k}$. We first estimate the probability of the bad event. By Markov's inequality,
        \begin{align}
            \mathbb{P}\left(\delta=1\mid x^k\right) &= \mathbb{P}\left(\norms{\nabla \func{x^k,\bxi^k} - \nabla \func{x^k}}^2 > 9 \frac{\rho - 1}{B} \norms{\nabla \func{x^k}}^2\mid x^k\right) \nonumber\\
            &\leq \frac{\expectcond{\norms{\nabla \func{x^k,\bxi^k} - \nabla \func{x^k}}^2}{x^k}}{9 (\rho - 1) \norms{\nabla \func{x^k}}^2} \cdot B \nonumber\\
            &\overset{\eqref{eq:variance_with_batch_size}}{\leq} \frac{1}{9}.
            \label{eq:Convex_clipSGD_prob1}
        \end{align}
        Consequently, $\mathbb{P}\left(\delta=0\mid x^k\right) = 1 - \mathbb{P}\left(\delta=1\mid x^k\right) \geq \frac{8}{9}$. We next bound the contribution of the good event:
        \begin{align}
            &\expectcond{- \eta \alpha_{\bxi^k} \dotprod{\nabla \func{x^k}}{\nabla \func{x^k, \bxi^k}}}{x^k,\delta=0} \nonumber\\
            &=  \expectcond{- \eta \alpha_{\bxi^k} \dotprod{\nabla \func{x^k}}{\nabla \func{x^k, \bxi^k} \pm \nabla \func{x^k}}}{x^k,\delta=0} \nonumber\\
            &=  \expectcond{- \eta \alpha_{\bxi^k} \norms{\nabla \func{x^k}}^2 - \eta \alpha_{\bxi^k} \dotprod{\nabla \func{x^k}}{\nabla \func{x^k, \bxi^k} - \nabla \func{x^k}}}{x^k,\delta=0}\nonumber\\
            &\overset{\eqref{eq:scalar_product_bound}}{\leq}  \expectcond{- \eta \alpha_{\bxi^k} \norms{\nabla \func{x^k}}^2 + \eta \alpha_{\bxi^k} \norms{\nabla \func{x^k}}\norms{\nabla \func{x^k, \bxi^k} - \nabla \func{x^k}}}{x^k,\delta=0}\nonumber\\
            &\overset{\circledOne}{\leq} \expectcond{- \eta \alpha_{\bxi^k} \norms{\nabla \func{x^k}}^2 + 3 \eta \alpha_{\bxi^k} \sqrt{\frac{\rho - 1}{B}} \norms{\nabla \func{x^k}}^2}{x^k,\delta=0}\nonumber\\
            &\overset{\circledTwo}{\leq} \expectcond{- \eta \alpha_{\bxi^k} \norms{\nabla \func{x^k}}^2 + \frac{\eta \alpha_{\bxi^k}}{2}  \norms{\nabla \func{x^k}}^2}{x^k,\delta=0}\nonumber\\
            &=\expectcond{- \frac{\eta \alpha_{\bxi^k}}{2}  \norms{\nabla \func{x^k}}^2}{x^k,\delta=0}\nonumber\\
            &\overset{\circledThree}{\leq} - \frac{\eta c}{4}  \norms{\nabla \func{x^k}},
            \label{eq:Convex_clipSGD_T1}
        \end{align}
        Here, in $\circledOne$ we used the definition of the event $\{\delta=0\}$, in $\circledTwo$ we used $B \geq 36 (\rho - 1)$, and in $\circledThree$ we used $\norms{\nabla \func{x^k, \bxi^k}}\leq 2\norms{\nabla \func{x^k}}$, which implies $\alpha_{\bxi^k}\geq c/(2\norms{\nabla \func{x^k}})$. On the complementary event, we only use the boundedness induced by clipping:
        
        \begin{align}
            \expectcond{- \eta \alpha_{\bxi^k} \dotprod{\nabla \func{x^k}}{\nabla \func{x^k, \bxi^k}}}{x^k,\delta=1} &\overset{\eqref{eq:scalar_product_bound}}{\leq} \eta \norms{\nabla \func{x^k}} \expectcond{ \alpha_{\bxi^k} \norms{\nabla \func{x^k, \bxi^k}}}{x^k,\delta=1} \nonumber\\
            &\overset{\circledOne}{\leq} \eta c \norms{\nabla \func{x^k}},\label{eq:Convex_clipSGD_T2}
        \end{align}
        where in $\circledOne$ we used $\norms{\clip{\nabla \func{x^k, \bxi^k}}} \leq c$.

        Combining the good and bad events yields
        \begin{align}
            &\expectcond{- \eta \dotprod{\nabla \func{x^k}}{\clip{\nabla \func{x^k, \bxi^k}}}}{x^k} \overset{\eqref{eq:Convex_clipSGD_T1_and_T2}}{\leq} \mathbb{P}\left(\delta=0\mid x^k\right) \underbrace{\expectcond{- \eta \alpha_{\bxi^k} \dotprod{\nabla \func{x^k}}{\nabla \func{x^k, \bxi^k}}}{x^k,\delta=0}}_{=: T_1} \nonumber\\
            &\quad\quad\quad+ \mathbb{P}\left(\delta=1\mid x^k\right) \underbrace{\expectcond{- \eta \alpha_{\bxi^k} \dotprod{\nabla \func{x^k}}{\nabla \func{x^k, \bxi^k}}}{x^k,\delta=1}}_{=: T_2}\nonumber\\
            &\overset{\eqref{eq:Convex_clipSGD_T1}, \eqref{eq:Convex_clipSGD_T2}}{\leq}
            -\frac{\eta c}{4} \norms{\nabla \func{x^k}}\cdot \mathbb{P}\left(\delta=0\mid x^k\right)  + \eta c \norms{\nabla \func{x^k}}\cdot \mathbb{P}\left(\delta=1\mid x^k\right)\nonumber\\
            &= -\frac{\eta c}{4} \norms{\nabla \func{x^k}}\cdot \left( 1 -  \mathbb{P}\left(\delta=1\mid x^k\right)\right)  + \eta c \norms{\nabla \func{x^k}}\cdot \mathbb{P}\left(\delta=1\mid x^k\right)\nonumber\\
            &\overset{\eqref{eq:Convex_clipSGD_prob1}}{\leq} -\eta c \left( \frac{8}{9} \cdot\frac{1}{4} -  \frac{1}{9} \right)\norms{\nabla \func{x^k}}\nonumber\\
            &= -\frac{\eta c}{9} \norms{\nabla \func{x^k}}.
            \label{eq:Convex_clipSGD_first_term_smooth}
        \end{align}
        Substituting this estimate into \eqref{eq:Convex_clipSGD_smooth}, we obtain:
        \begin{align}
            \expectcond{F_{k+1}}{x^k} - F_k &\overset{\eqref{eq:Convex_clipSGD_smooth}}{\leq} 
            - \expectcond{\eta \dotprod{\nabla \func{x^k}}{\clip{\nabla \func{x^k, \bxi^k}}}}{x^k}\nonumber\\
            &\quad \quad \quad+ \frac{\eta^2 \left(\cL_0 + \sqrt{\rho} \cL_1  \norms{\nabla \func{x^k}} \right)}{2} \expectcond{\norms{\clip{\nabla \func{x^k, \bxi^k}}}^2}{x^k}\nonumber\\
            &\overset{\eqref{eq:Convex_clipSGD_first_term_smooth}}{\leq} -\frac{\eta c}{9} \norms{\nabla \func{x^k}} + \frac{\eta^2 \left(\cL_0 + \sqrt{\rho} \cL_1 \norms{\nabla \func{x^k}} \right)}{2} \expectcond{\norms{\clip{\nabla \func{x^k, \bxi^k}}}^2}{x^k}\nonumber\\
            &\overset{\circledOne}{\leq} -\frac{\eta c}{9} \norms{\nabla \func{x^k}} + \frac{\eta^2 \left(\cL_0 + \sqrt{\rho} \cL_1  \norms{\nabla \func{x^k}} \right)}{2} \cdot c^2\nonumber\\
            &\overset{\circledTwo}{\leq} -\frac{\eta c}{9} \norms{\nabla \func{x^k}} + \frac{\eta^2 \left(\cL_0 + \sqrt{\rho} \cL_1  c \right)}{2} \cdot c \cdot \norms{\nabla \func{x^k}} \nonumber\\
            &\overset{\circledThree}{\leq} -\frac{\eta c}{9} \norms{\nabla \func{x^k}} + \frac{\eta c}{18} \norms{\nabla \func{x^k}} \nonumber\\
            &= -\frac{\eta c}{18} \norms{\nabla \func{x^k}}\nonumber\\
            &\overset{\eqref{eq:function_gap_gradient_relation_ClipSGD}}{\leq} -\frac{\eta c}{18 R_0} F_k, \label{eq:Convex_clipSGD_case1_descent}
        \end{align}
        where in $\circledOne$ we used $\norms{\clip{\nabla \func{x^k, \bxi^k}}} \leq c$, in $\circledTwo$ we used $\norms{\nabla \func{x^k}} \geq c$, in $\circledThree$ we chose $\eta \leq \left[9 \left( \cL_0 + \sqrt{\rho} \cL_1 c  \right)  \right]^{-1}$.

        Hence, in the first case,
        \begin{equation}
            \expectcond{F_{k+1}}{x^k} \leq \left( 1 - \frac{\eta c}{18 R_0} \right) F_k.
            \label{eq:Convex_clipSGD_case1_ITOG}
        \end{equation}

        \fbox{2) The case $\norms{\nabla \func{x^k}} < c$.}

        In this case, $\nabla \func{x^k} = \clip{\nabla \func{x^k}}$. From \eqref{eq:Convex_clipSGD_smooth}, we have:
        \begin{align*}
            \expectcond{\func{x^{k+1}}}{x^k} - \func{x^k} &\overset{\eqref{eq:Convex_clipSGD_smooth}}{\leq} - \eta \expectcond{\dotprod{\nabla \func{x^k}}{\clip{\nabla \func{x^k, \bxi^k}}}}{x^k} \nonumber\\
            &\quad \quad \quad+ \frac{\eta^2 \left(\cL_0 + \sqrt{\rho} \cL_1 \norms{\nabla \func{x^k}} \right)}{2} \expectcond{\norms{\clip{\nabla \func{x^k, \bxi^k}}}^2}{x^k}\\
            &\overset{\circledOne}{\leq} - \eta \expectcond{\dotprod{\nabla \func{x^k}}{\clip{\nabla \func{x^k, \bxi^k}}}}{x^k} \nonumber\\
            &\quad \quad \quad+ \frac{\eta^2 \left(\cL_0 + \sqrt{\rho} \cL_1  c \right)}{2} \expectcond{\norms{\clip{\nabla \func{x^k, \bxi^k}}}^2}{x^k}\\
            &\overset{\circledTwo}{=} -\frac{\eta}{2} \norms{\nabla \func{x^k}}^2 - \frac{\eta}{2} \expectcond{\norms{\clip{\nabla \func{x^k, \bxi^k}}}^2}{x^k} \\
            &\quad\quad\quad+ \frac{\eta}{2} \expectcond{\norms{\clip{\nabla \func{x^k, \bxi^k}} - \nabla \func{x^k}}^2}{x^k}\\
            &\quad \quad \quad+ \frac{\eta^2 \left(\cL_0 + \sqrt{\rho} \cL_1  c \right)}{2} \expectcond{\norms{\clip{\nabla \func{x^k, \bxi^k}}}^2}{x^k}\\
            &\overset{\circledThree}{=} -\frac{\eta}{2} \norms{\nabla \func{x^k}}^2 + \frac{\eta}{2} \expectcond{\norms{\clip{\nabla \func{x^k, \bxi^k}} - \clip{\nabla \func{x^k}}}^2}{x^k}\\
            &\quad \quad \quad- \frac{\eta}{2} \left( 1 -\eta \left(\cL_0 + \sqrt{\rho} \cL_1 c \right) \right) \expectcond{\norms{\clip{\nabla \func{x^k, \bxi^k}}}^2}{x^k}\\
            &\overset{\circledFour}{\leq} -\frac{\eta}{2} \norms{\nabla \func{x^k}}^2 + \frac{\eta}{2} \expectcond{\norms{\clip{\nabla \func{x^k, \bxi^k}} - \clip{\nabla \func{x^k}}}^2}{x^k},
        \end{align*}
        where in $\circledOne$ we used $\norms{\nabla \func{x^k}} < c$, in $\circledTwo$ we used inequality \eqref{eq:quadratic_difference} with $a = \nabla \func{x^{k}}$ and $b =  \clip{\nabla \func{x^k, \bxi^k}}$, in $\circledThree$ we used $\nabla \func{x^k} = \clip{\nabla \func{x^k}}$, in $\circledFour$ we chose $\eta \leq \left[9 \left( \cL_0 + \sqrt{\rho} \cL_1 c  \right)  \right]^{-1}$.

        Since clipping is the projection onto the Euclidean ball of radius $c$, it is a Lipschitz operator with Lipschitz constant $1$. Therefore,
        \begin{align}
            \expectcond{\func{x^{k+1}}}{x^k} - \func{x^k} &\leq -\frac{\eta}{2} \norms{\nabla \func{x^k}}^2 + \frac{\eta}{2} \expectcond{\norms{\nabla \func{x^k, \bxi^k} - \nabla \func{x^k}}^2}{x^k}\nonumber\\
            &\overset{\eqref{eq:variance_with_batch_size}}{\leq} -\frac{\eta}{2} \norms{\nabla \func{x^k}}^2 + \frac{\eta (\rho - 1)}{2 B} \norms{\nabla \func{x^k}}^{2}\nonumber\\
            &\overset{\circledOne}{\leq}  -\frac{\eta}{2} \norms{\nabla \func{x^k}}^2 + \frac{\eta}{72} \norms{\nabla \func{x^k}}^{2}\nonumber\\
            &\leq  -\frac{\eta}{2} \norms{\nabla \func{x^k}}^2 + \frac{\eta}{9} \norms{\nabla \func{x^k}}^{2}\nonumber\\
            &=  -\frac{7\eta}{18} \norms{\nabla \func{x^k}}^2\nonumber\\
            &\leq  -\frac{\eta}{9} \norms{\nabla \func{x^k}}^2\nonumber\\
            &\overset{\eqref{eq:function_gap_gradient_relation_ClipSGD}}{\leq}  -\frac{\eta}{9 R_0^2} \left( \func{x^k} - f^*\right)^2,
            \label{eq:Convex_clipSGD_case2_descent}
        \end{align}
        where in $\circledOne$ we chose $B \geq 36 (\rho - 1)$. 

        \textbf{In summary.} Let $\mathcal{F}_k$ be the sigma-algebra generated by the history up to
$x^k$, and define
\begin{equation*}
    G_k:=\norms{\nabla \func{x^k}},
    \qquad
    I_k:=\mathbbm{1}_{\{G_k\geq c\}}.
\end{equation*}
The random variable $I_k$ is $\mathcal{F}_k$-measurable. Set
\begin{equation*}
    a:=\frac{\eta c}{18R_0},
    \qquad
    b:=\frac{\eta}{9R_0^2}.
\end{equation*}
The two conditional estimates derived above can be combined as
\begin{equation}
    \label{eq:combined_conditional_recursion}
    \expectcond{F_{k+1}}{\mathcal{F}_k}
    \leq
    F_k-aF_k I_k-bF_k^2(1-I_k).
\end{equation}
Indeed, on the event $I_k=1$ this is exactly the large-gradient estimate,
whereas on the event $I_k=0$ it is exactly the small-gradient estimate.

Taking expectations in~\eqref{eq:combined_conditional_recursion}, we obtain
\begin{equation}
    \label{eq:expected_combined_recursion}
    A_{k+1}
    \leq
    A_k
    -
    a\expect{F_k I_k}
    -
    b\expect{F_k^2(1-I_k)},
    \qquad
    A_k:=\expect{F_k}.
\end{equation}

We now use the following elementary inequality. For any nonnegative random
variable $X$ and any indicator random variable $I$,
\begin{equation}
    \label{eq:elementary_indicator_inequality}
    a\expect{XI}
    +
    b\expect{X^2(1-I)}
    \geq
    \min\left\{
        \frac{a}{2}\expect{X},
        b\left(\expect{X}\right)^2
    \right\}.
\end{equation}
To prove it, let $A=\expect{X}$ and $r=\expect{X(1-I)}$. Since
\begin{equation*}
    \expect{X^2(1-I)}
    =
    \expect{\left(X(1-I)\right)^2}
    \geq
    \left(\expect{X(1-I)}\right)^2
    =
    r^2,
\end{equation*}
we have
\begin{equation*}
    a\expect{XI}
    +
    b\expect{X^2(1-I)}
    \geq
    a(A-r)+br^2.
\end{equation*}
Minimizing the right-hand side over $r\in[0,A]$, we get: if
$A\leq a/(2b)$, the minimum is attained at $r=A$ and equals $bA^2$; if
$A\geq a/(2b)$, the minimum is attained at $r=a/(2b)$ and equals
\begin{equation*}
    aA-\frac{a^2}{4b}
    \geq
    \frac{a}{2}A.
\end{equation*}
This proves~\eqref{eq:elementary_indicator_inequality}.

Applying~\eqref{eq:elementary_indicator_inequality} to $X=F_k$ and $I=I_k$,
inequality~\eqref{eq:expected_combined_recursion} gives
\begin{equation}
    \label{eq:deterministic_gap_recursion}
    A_{k+1}
    \leq
    A_k
    -
    \min\left\{
        \frac{a}{2}A_k,
        bA_k^2
    \right\}.
\end{equation}
In particular, $(A_k)_{k\geq0}$ is non-increasing.

Define
\begin{equation*}
    \alpha:=\frac{a}{2}=\frac{\eta c}{36R_0},
    \qquad
    \beta:=b=\frac{\eta}{9R_0^2}.
\end{equation*}
Then~\eqref{eq:deterministic_gap_recursion} becomes
\begin{equation}
    \label{eq:alpha_beta_gap_recursion}
    A_{k+1}
    \leq
    A_k-\min\{\alpha A_k,\beta A_k^2\}.
\end{equation}

Let
\begin{equation*}
    \tau
    :=
    \min\left\{
        k\in\{0,\ldots,N-1\}
        :
        A_k<\frac{\alpha}{\beta}
    \right\},
\end{equation*}
with the convention $\tau=N$ if the set is empty. Since $(A_k)$ is
non-increasing, the indices $k<\tau$ correspond to the linear part of
\eqref{eq:alpha_beta_gap_recursion}, and the indices $k\geq\tau$ correspond
to the quadratic part.

If $\tau>N/2$, then for at least $N/2$ iterations,
\begin{equation*}
    A_{k+1}\leq (1-\alpha)A_k.
\end{equation*}
Using monotonicity for the remaining iterations, we get
\begin{equation}
    \label{eq:linear_phase_bound}
    A_N
    \leq
    (1-\alpha)^{N/2}A_0
    =
    \left(
        1-\frac{\eta c}{36R_0}
    \right)^{N/2}F_0.
\end{equation}

If $\tau\leq N/2$, then for $k=\tau,\ldots,N-1$,
\begin{equation*}
    A_{k+1}\leq A_k-\beta A_k^2.
\end{equation*}
If $A_N=0$, the claim is trivial. Otherwise,
\begin{equation*}
    \frac{1}{A_{k+1}}-\frac{1}{A_k}
    =
    \frac{A_k-A_{k+1}}{A_kA_{k+1}}
    \geq
    \beta\frac{A_k}{A_{k+1}}
    \geq
    \beta.
\end{equation*}
Summing from $k=\tau$ to $N-1$ and using $N-\tau\geq N/2$, we obtain
\begin{equation}
    \label{eq:quadratic_phase_bound}
    A_N
    \leq
    \frac{2}{\beta N}
    =
    \frac{18R_0^2}{\eta N}.
\end{equation}

Combining~\eqref{eq:linear_phase_bound} and
\eqref{eq:quadratic_phase_bound}, we get
\begin{equation*}
    \expect{\func{x^N}}-f^*
    =
    A_N
    \leq
    \max\left\{
        \left(
            1-\frac{\eta c}{36R_0}
        \right)^{N/2}F_0,
        \frac{18R_0^2}{\eta N}
    \right\}.
\end{equation*}

\end{proof}

    \begin{remark}
            \label{rem:Step_size_convex_ClipSGD}
            In the first step of~\eqref{eq:Convex_clipSGD_smooth}, we use the descent lemma for $(L_0,L_1)$-smooth functions~\eqref{eq:L0L1_descent_inequality} \citep[see Appendix~A.1 in][]{Zhang_2020_Improved}; the condition $\norms{x^{k+1}-x^k}\leq \frac{1}{\sqrt{\rho} \cL_1}$ required for its application is guaranteed by the chosen step size $\eta \leq \left[9 \left( \cL_0 + \sqrt{\rho} \cL_1 c  \right)  \right]^{-1}$:
            \begin{equation*}
                \norms{x^{k+1}-x^k} \overset{\circledOne}{=} \eta \norms{\clip{\nabla \func{x^k, \bxi^k}}} \overset{\circledTwo}{\leq} \eta c \overset{\circledThree}{\leq} \frac{c}{9 \left( \cL_0 + \sqrt{\rho} \cL_1 c  \right)} \overset{\circledFour}{\leq} \frac{c}{9 \sqrt{\rho} \cL_1 c} \overset{\circledFour}{\leq} \frac{1}{\sqrt{\rho} \cL_1},
            \end{equation*}
            where in $\circledOne$ we used $x^{k+1} = x^k - \eta \cdot \clip{\nabla \func{x^k, \bxi^k}}$, in $\circledTwo$ we used $\norms{\clip{\nabla \func{x^k, \bxi^k}}} = \norms{\min\left\{1, \frac{c}{\norms{\nabla \func{x^k, \bxi^k}}} \right\} \nabla \func{x^k, \bxi^k}} \leq c$, in $\circledThree$ we used $\eta \leq \left[9 \left( \cL_0 + \sqrt{\rho} \cL_1 c  \right)  \right]^{-1}$, where in $\circledFour$ we used the fact that, for $0<b\leq a$, one has $\frac{1}{a}\leq \frac{1}{b}$.
        \end{remark}

\subsection{Missing proofs for normalized stochastic gradient descent method}

This section gives the corresponding proofs for NSGD. The hyperparameter $\lambda>0$ plays a role analogous to the clipping radius, but it also appears directly in the normalization denominator. As a result, the admissible step size and the final error floor depend on $\lambda$.

    \subsubsection{Non-convex setup under \texorpdfstring{$(L_0,L_1)$}{TEXT}-smooth and generalized strong growth conditions}\label{app:NC_NSGD}
\begin{proof}
For simplicity of presentation, define
\[
    \Gf{x^k, \bxi^k}
    \coloneqq
    \frac{\nabla \func{x^k, \bxi^k}}{\norms{\nabla \func{x^k, \bxi^k}} + \lambda},
    \qquad
        F_k := \func{x^k}-f^*.
    \]
In the one-step estimates below, expectations are taken with respect to the fresh mini-batch $\bxi^k$ and are written explicitly as conditional expectations given the current iterate $x^k$. After taking full expectation, we return to the notation $\expect{\cdot}$.

     We start from the descent inequality for $(L_0,L_1)$-smooth functions:
     \begin{align}
            \func{x^{k+1}} - \func{x^k} &\overset{\circledOne}{\leq} \dotprod{\nabla \func{x^k}}{x^{k+1} - x^k} + \frac{L_0 + L_1 \norms{\nabla \func{x^k}}}{2} \norms{x^{k+1} - x^k}^2\nonumber\\
            &\overset{\circledTwo}{=} -\eta \dotprod{\nabla \func{x^k}}{\Gf{x^k, \bxi^k}} + \frac{\eta^2 \left(L_0 + L_1 \norms{\nabla \func{x^k}} \right)}{2} \norms{\Gf{x^k, \bxi^k}}^2\nonumber\\
            &\overset{\circledThree}{\leq} -\eta \frac{\dotprod{\nabla \func{x^k}}{\nabla \func{x^k, \bxi^k}}}{\norms{\nabla \func{x^k, \bxi^k}} + \lambda} + \frac{\eta^2 \left(L_0 + L_1 \norms{\nabla \func{x^k}} \right)}{2}\nonumber\\ 
            &= -\eta \frac{\dotprod{\nabla \func{x^k} \pm \nabla \func{x^k, \bxi^k} }{\nabla \func{x^k, \bxi^k}}}{\norms{\nabla \func{x^k, \bxi^k}} + \lambda} + \frac{\eta^2 \left(L_0 + L_1 \norms{\nabla \func{x^k}} \right)}{2}\nonumber\\ 
            &= -\eta \frac{\norms{\nabla \func{x^k, \bxi^k}}^2}{\norms{\nabla \func{x^k, \bxi^k}} + \lambda} -\eta \frac{\dotprod{\nabla \func{x^k} - \nabla \func{x^k, \bxi^k} }{\nabla \func{x^k, \bxi^k}}}{\norms{\nabla \func{x^k, \bxi^k}} + \lambda} \nonumber\\
            &\quad\quad\quad+ \frac{\eta^2 \left(L_0 + L_1 \norms{\nabla \func{x^k}} \right)}{2} \nonumber\\
            &\overset{\circledFour}{\leq} -\eta \frac{\norms{\nabla \func{x^k, \bxi^k}}^2}{\norms{\nabla \func{x^k, \bxi^k}} + \lambda} +  \eta \frac{\norms{\nabla \func{x^k} - \nabla \func{x^k, \bxi^k} } \norms{\nabla \func{x^k, \bxi^k}}}{\norms{\nabla \func{x^k, \bxi^k}} + \lambda} \nonumber\\
            &\quad\quad\quad+ \frac{\eta^2 \left(L_0 + L_1 \norms{\nabla \func{x^k}} \right)}{2} \nonumber\\
            &\overset{\circledThree}{\leq} -\eta \frac{\norms{\nabla \func{x^k, \bxi^k}}^2}{\norms{\nabla \func{x^k, \bxi^k}} + \lambda} +  \eta \norms{\nabla \func{x^k} - \nabla \func{x^k, \bxi^k} } \nonumber\\
            &\quad\quad\quad+ \frac{\eta^2 \left(L_0 + L_1 \norms{\nabla \func{x^k}} \right)}{2} \nonumber\\
            &= -\eta \norms{\nabla \func{x^k, \bxi^k}} + \eta \frac{\lambda \norms{\nabla \func{x^k, \bxi^k}}}{\norms{\nabla \func{x^k, \bxi^k}} + \lambda} +  \eta \norms{\nabla \func{x^k} - \nabla \func{x^k, \bxi^k} } \nonumber\\
            &\quad\quad\quad+ \frac{\eta^2 \left(L_0 + L_1 \norms{\nabla \func{x^k}} \right)}{2} \nonumber\\
            &\overset{\circledThree}{\leq} -\eta \norms{\nabla \func{x^k, \bxi^k}} + \eta \lambda +  \eta \norms{\nabla \func{x^k} - \nabla \func{x^k, \bxi^k} } \nonumber\\
            &\quad\quad\quad+ \frac{\eta^2 \left(L_0 + L_1 \norms{\nabla \func{x^k}} \right)}{2} \nonumber\\
            &\overset{\circledFive}{\leq}  -\eta \norms{\nabla \func{x^k}} + 2 \eta \norms{\nabla \func{x^k, \bxi^k} - \nabla \func{x^k}}  \nonumber\\
            &\quad\quad\quad+ \frac{\eta^2 \left(L_0 + L_1 \norms{\nabla \func{x^k}} \right)}{2} + \eta \lambda \nonumber\\
            &\overset{\circledSix}{\leq} -\eta \norms{\nabla \func{x^k}} + 2 \eta \norms{\nabla \func{x^k, \bxi^k} - \nabla \func{x^k}}  \nonumber\\
            &\quad\quad\quad+ \frac{\eta}{2} \left( \norms{\nabla \func{x^k}} + \lambda \right) + \eta \lambda \nonumber\\
            &= -\frac{\eta}{2} \norms{\nabla \func{x^k}} + 2 \eta \norms{\nabla \func{x^k, \bxi^k} - \nabla \func{x^k}} + \frac{3}{2} \eta \lambda,
            \label{eq:NC_NSGD_smooth}
        \end{align}
        Here, in~$\circledOne$ we used the descent inequality~\eqref{eq:L0L1_descent_inequality}; its locality condition $\norms{x^{k+1}-x^k}\leq \frac{1}{L_1}$ is verified in Remark~\ref{rem:Step_size_NC_NSGD}. In $\circledTwo$ we used the NSGD update $x^{k+1} = x^k - \eta \cdot \Gf{x^k, \bxi^k}$. In $\circledThree$ we used $\norms{\Gf{x^k, \bxi^k}}\leq 1$, in $\circledFour$ the Cauchy--Schwarz inequality~\eqref{eq:scalar_product_bound}, and in $\circledFive$ the triangle inequality. Finally, in $\circledSix$ we used the step-size condition $\eta \leq \frac{\lambda}{L_0 + L_1 \lambda} $ together with $\frac{\lambda}{L_0 + L_1 \lambda} \leq \frac{\norms{\nabla \func{x^k}} + \lambda}{L_0 + L_1 \norms{\nabla \func{x^k}}}$, which follows from
        \begin{align*}
            \frac{\norms{\nabla \func{x^k}} + \lambda}{L_0 + L_1 \norms{\nabla \func{x^k}}} - \frac{\lambda}{L_0 + L_1 \lambda} &= \frac{(\norms{\nabla \func{x^k}} + \lambda) (L_0 + L_1 \lambda) - \lambda (L_0 + L_1 \norms{\nabla \func{x^k}})}{(L_0 + L_1 \norms{\nabla \func{x^k}}) (L_0 + L_1 \lambda)} \\
            &= \frac{L_0 \norms{\nabla \func{x^k}} + \lambda^2 L_1}{(L_0 + L_1 \norms{\nabla \func{x^k}}) (L_0 + L_1 \lambda)} \geq 0,
        \end{align*}
        where $L_0, L_1, \lambda ,\norms{\nabla \func{x^k}} \geq 0$.

        Taking conditional expectation in \eqref{eq:NC_NSGD_smooth} given $x^k$, we obtain:
        \begin{align}
            \expectcond{F_{k+1}}{x^k} - F_k &\overset{\eqref{eq:NC_NSGD_smooth}}{\leq} -\frac{\eta}{2} \norms{\nabla \func{x^k}} + 2 \eta \expectcond{\norms{\nabla \func{x^k, \bxi^k} - \nabla \func{x^k}}}{x^k} + \frac{3}{2} \eta \lambda\nonumber\\
            &= -\frac{\eta}{2} \norms{\nabla \func{x^k}} + 2 \eta \expectcond{\left(\norms{\nabla \func{x^k, \bxi^k} - \nabla \func{x^k}}^2 \right)^{1/2}}{x^k} + \frac{3}{2} \eta \lambda\nonumber\\
            &\overset{\circledOne}{\leq} -\frac{\eta}{2} \norms{\nabla \func{x^k}} + 2 \eta \left(\expectcond{\norms{\nabla \func{x^k, \bxi^k} - \nabla \func{x^k}}^2}{x^k} \right)^{1/2} + \frac{3}{2} \eta \lambda\nonumber\\
            &\overset{\eqref{eq:variance_with_batch_size}}{\leq} -\frac{\eta}{2} \norms{\nabla \func{x^k}} + 2 \eta \left(\frac{1}{B} \left[ (\rho - 1) \norms{\nabla \func{x^k}}^2 + \sigma^2 \right] \right)^{1/2} + \frac{3}{2} \eta \lambda\nonumber\\
            &\overset{\circledTwo}{\leq} -\frac{\eta}{2} \norms{\nabla \func{x^k}} + 2 \eta \sqrt{\frac{(\rho - 1)}{B}}   \norms{\nabla \func{x^k}} + \frac{2\eta \sigma}{\sqrt{B}} + \frac{3}{2} \eta \lambda\nonumber\\
            &\overset{\circledThree}{\leq} -\frac{\eta}{2} \norms{\nabla \func{x^k}} + \frac{\eta}{4}   \norms{\nabla \func{x^k}} + \frac{2\eta \sigma}{\sqrt{B}} + \frac{3}{2} \eta \lambda\nonumber\\
            &= -\frac{\eta}{4} \norms{\nabla \func{x^k}} + \frac{2\eta \sigma}{\sqrt{B}} + \frac{3}{2} \eta \lambda, \label{NC_NSGD_SeMI_Itog}
        \end{align}
        where in $\circledOne$ we applied Jensen's inequality, in $\circledTwo$ we used $\sqrt{a^2 + b^2} \leq a + b$ when $a,b \geq 0$, in $\circledThree$ we chose $B \geq 64 (\rho - 1)$. Thus, we have obtained the inequality:
        \begin{equation*}
            \label{eq:NC_NSGD_ITOG}
            \norms{\nabla \func{x^k}} \overset{\eqref{NC_NSGD_SeMI_Itog}}{\leq} \frac{4 \left(F_k - \expectcond{F_{k+1}}{x^k}\right)}{\eta} + \frac{8 \sigma}{\sqrt{B}} + 6 \lambda.
        \end{equation*}

        Taking full expectation and summing this inequality over all indices $0 \leq k \leq N-1$ gives:
        \begin{equation*}
            \min_{k\in [0, N-1]}\expect{\norms{\nabla \func{x^k}}} \leq \frac{1}{N} \sum_{k = 0}^{N-1} \expect{\norms{\nabla \func{x^k}}} \leq \underbrace{\frac{4F_0}{\eta N} + \frac{8 \sigma}{\sqrt{B}}}_{\circledOne} + \underbrace{6 \lambda}_{\circledTwo}.
        \end{equation*}
        To guarantee convergence to the desired accuracy $\min_{k \in [0, N-1]} \mathbb{E}\left[|\nabla f(x^k)|^2\right] \leq \varepsilon$, it is sufficient to bound the first two terms and the third term by $\varepsilon/2$. This yields the constraint:
        \begin{equation*}
            \circledTwo:\quad  \lambda \leq \frac{\varepsilon}{12}.
        \end{equation*}
        Taking $\lambda = \frac{\varepsilon}{12}$ and the maximum admissible step size $\eta = \frac{\lambda}{L_0 + L_1 \lambda} = \left[ \frac{12 L_0}{\varepsilon} + L_1 \right]^{-1}$, we obtain the following sufficient condition under batch size $B \geq 64 (\rho - 1)$:
        \begin{equation*}
            \circledOne: \quad \frac{4F_0}{\eta N} + \frac{8 \sigma}{\sqrt{B}} = \frac{48 L_0}{\varepsilon N} F_0 + \frac{4 L_1}{N} F_0 + \frac{8 \sigma}{\sqrt{B}} \leq \frac{\varepsilon}{2}.
        \end{equation*}

        In the case $\sigma = 0$, the above conditions are satisfied by
        \begin{equation*}
            \lambda = \frac{\varepsilon}{12}; \quad N \geq \frac{48 L_0}{\varepsilon^2} F_0 + \frac{4 L_1}{\varepsilon} F_0; \quad B \geq 64 (\rho - 1),
        \end{equation*}
        and the corresponding number of oracle calls $\# T$:
        \begin{equation*}
            T = N \cdot B = \mathcal{O} \left( \frac{(\rho - 1) L_0}{\varepsilon^2} F_0 + \frac{(\rho - 1) L_1}{\varepsilon} F_0 \right).
        \end{equation*}
\end{proof}

    \begin{remark}
            \label{rem:Step_size_NC_NSGD}
            In the first step of~\eqref{eq:NC_NSGD_smooth}, we use the descent lemma for $(L_0,L_1)$-smooth functions~\eqref{eq:L0L1_descent_inequality} \citep[see Appendix~A.1 in][]{Zhang_2020_Improved}; the condition $\norms{x^{k+1}-x^k}\leq 1/L_1$ required for its application is guaranteed by the chosen step size $\eta \leq \frac{\lambda}{L_0 + L_1 \lambda}$:
            \begin{equation*}
                \norms{x^{k+1}-x^k} \overset{\circledOne}{=} \eta \norms{\Gf{x^k, \bxi^k}} \overset{\circledTwo}{\leq} \eta \overset{\circledThree}{\leq} \frac{\lambda}{L_0 + L_1 \lambda} \overset{\circledFour}{\leq} \frac{\lambda}{ L_1 \lambda} = \frac{1}{L_1},
            \end{equation*}
            where in $\circledOne$ we used $x^{k+1} = x^k - \eta \cdot \Gf{x^k, \bxi^k}$, in $\circledTwo$ we used $\norms{\Gf{x^k, \bxi^k}} = \norms{\frac{\nabla \func{x^k, \bxi^k}}{\norms{\nabla \func{x^k, \bxi^k}} + \lambda}} \leq 1$, in $\circledThree$ we used $\eta \leq \frac{\lambda}{L_0 + L_1 \lambda}$, where in $\circledFour$ we used the fact that, for $0<b\leq a$, one has $\frac{1}{a}\leq \frac{1}{b}$.
        \end{remark}

    \subsubsection{Convex setup under \texorpdfstring{$(\cL_0,\cL_1)$}{TEXT}-smooth and strong growth conditions}\label{app:Convex_NSGD}
    
    We now prove the convex NSGD result. The proof follows the non-convex argument up to the one-step descent estimate, and then uses convexity to obtain a recursion for the optimality gap.

\begin{proof}
For simplicity of presentation, define
\[
    \Gf{x^k, \bxi^k}
    \coloneqq
    \frac{\nabla \func{x^k, \bxi^k}}{\norms{\nabla \func{x^k, \bxi^k}} + \lambda},
    \qquad
        F_k := \func{x^k}-f^*.
    \]
In the one-step estimates below, expectations are taken with respect to the fresh mini-batch $\bxi^k$ and are written explicitly as conditional expectations given the current iterate $x^k$. After taking full expectation, we return to the notation $\expect{\cdot}$.
By convexity, the Cauchy--Schwarz inequality, and the monotonicity of the distance to the optimum from Lemma~\ref{lem:monotone_nsgd}, we have
\begin{equation}
    F_k
    \leq
    \dotprod{\nabla \func{x^k}}{x^k-x^*}
    \leq
    \norms{\nabla \func{x^k}}\norms{x^k-x^*}
    \leq
    R_0\norms{\nabla \func{x^k}}.
    \label{eq:function_gap_gradient_relation_NSGD}
\end{equation}

    We start from the descent inequality for $(L_0,L_1)$-smooth functions with constants $L_0 = \cL_0$ and $L_1 = \sqrt{\rho} \cL_1$ (see Section~\ref{subsec:generalized_smoothness_conditions}):
     \begin{align}
            \func{x^{k+1}} - \func{x^k} &\overset{\circledOne}{\leq} \dotprod{\nabla \func{x^k}}{x^{k+1} - x^k} + \frac{\cL_0 + \sqrt{\rho} \cL_1 \norms{\nabla \func{x^k}}}{2} \norms{x^{k+1} - x^k}^2\nonumber\\
            &\overset{\circledTwo}{=} -\eta \dotprod{\nabla \func{x^k}}{\Gf{x^k, \bxi^k}} + \frac{\eta^2 \left(\cL_0 + \sqrt{\rho} \cL_1 \norms{\nabla \func{x^k}} \right)}{2} \norms{\Gf{x^k, \bxi^k}}^2\nonumber\\
            &\overset{\circledThree}{\leq} -\eta \frac{\dotprod{\nabla \func{x^k}}{\nabla \func{x^k, \bxi^k}}}{\norms{\nabla \func{x^k, \bxi^k}} + \lambda} + \frac{\eta^2 \left(\cL_0 + \sqrt{\rho} \cL_1 \norms{\nabla \func{x^k}} \right)}{2}\nonumber\\ 
            &= -\eta \frac{\dotprod{\nabla \func{x^k} \pm \nabla \func{x^k, \bxi^k} }{\nabla \func{x^k, \bxi^k}}}{\norms{\nabla \func{x^k, \bxi^k}} + \lambda} + \frac{\eta^2 \left(\cL_0 + \sqrt{\rho} \cL_1 \norms{\nabla \func{x^k}} \right)}{2}\nonumber\\ 
            &= -\eta \frac{\norms{\nabla \func{x^k, \bxi^k}}^2}{\norms{\nabla \func{x^k, \bxi^k}} + \lambda} -\eta \frac{\dotprod{\nabla \func{x^k} - \nabla \func{x^k, \bxi^k} }{\nabla \func{x^k, \bxi^k}}}{\norms{\nabla \func{x^k, \bxi^k}} + \lambda} \nonumber\\
            &\quad\quad\quad+ \frac{\eta^2 \left(\cL_0 + \sqrt{\rho} \cL_1 \norms{\nabla \func{x^k}} \right)}{2} \nonumber\\
            &\overset{\circledFour}{\leq} -\eta \frac{\norms{\nabla \func{x^k, \bxi^k}}^2}{\norms{\nabla \func{x^k, \bxi^k}} + \lambda} +  \eta \frac{\norms{\nabla \func{x^k} - \nabla \func{x^k, \bxi^k} } \norms{\nabla \func{x^k, \bxi^k}}}{\norms{\nabla \func{x^k, \bxi^k}} + \lambda} \nonumber\\
            &\quad\quad\quad+ \frac{\eta^2 \left(\cL_0 + \sqrt{\rho} \cL_1 \norms{\nabla \func{x^k}} \right)}{2} \nonumber\\
            &\overset{\circledThree}{\leq} -\eta \frac{\norms{\nabla \func{x^k, \bxi^k}}^2}{\norms{\nabla \func{x^k, \bxi^k}} + \lambda} +  \eta \norms{\nabla \func{x^k} - \nabla \func{x^k, \bxi^k} } \nonumber\\
            &\quad\quad\quad+ \frac{\eta^2 \left(\cL_0 + \sqrt{\rho} \cL_1 \norms{\nabla \func{x^k}} \right)}{2} \nonumber\\
            &= -\eta \norms{\nabla \func{x^k, \bxi^k}} + \eta \frac{\lambda \norms{\nabla \func{x^k, \bxi^k}}}{\norms{\nabla \func{x^k, \bxi^k}} + \lambda} +  \eta \norms{\nabla \func{x^k} - \nabla \func{x^k, \bxi^k} } \nonumber\\
            &\quad\quad\quad+ \frac{\eta^2 \left(\cL_0 + \sqrt{\rho} \cL_1 \norms{\nabla \func{x^k}} \right)}{2} \nonumber\\
            &\overset{\circledThree}{\leq} -\eta \norms{\nabla \func{x^k, \bxi^k}} + \eta \lambda +  \eta \norms{\nabla \func{x^k} - \nabla \func{x^k, \bxi^k} } \nonumber\\
            &\quad\quad\quad+ \frac{\eta^2 \left(\cL_0 + \sqrt{\rho} \cL_1 \norms{\nabla \func{x^k}} \right)}{2} \nonumber\\
            &\overset{\circledFive}{\leq}  -\eta \norms{\nabla \func{x^k}} + 2 \eta \norms{\nabla \func{x^k, \bxi^k} - \nabla \func{x^k}}  \nonumber\\
            &\quad\quad\quad+ \frac{\eta^2 \left(\cL_0 + \sqrt{\rho} \cL_1 \norms{\nabla \func{x^k}} \right)}{2} + \eta \lambda \nonumber\\
            &\overset{\circledSix}{\leq} -\eta \norms{\nabla \func{x^k}} + 2 \eta \norms{\nabla \func{x^k, \bxi^k} - \nabla \func{x^k}}  \nonumber\\
            &\quad\quad\quad+ \frac{\eta}{2} \left( \norms{\nabla \func{x^k}} + \lambda \right) + \eta \lambda \nonumber\\
            &= -\frac{\eta}{2} \norms{\nabla \func{x^k}} + 2 \eta \norms{\nabla \func{x^k, \bxi^k} - \nabla \func{x^k}} + \frac{3}{2} \eta \lambda,
            \label{eq:Convex_NSGD_smooth}
        \end{align}
        Here, in~$\circledOne$ we used the descent inequality~\eqref{eq:L0L1_descent_inequality}; its locality condition $\norms{x^{k+1}-x^k}\leq \frac{1}{L_1}$ with $L_1 = \sqrt{\rho} \cL_1$ is verified in Remark~\ref{rem:Step_size_Convex_NSGD}. In $\circledTwo$ we used the NSGD update $x^{k+1} = x^k - \eta \cdot \Gf{x^k, \bxi^k}$. In $\circledThree$ we used $\norms{\Gf{x^k, \bxi^k}}\leq 1$, in $\circledFour$ the Cauchy--Schwarz inequality~\eqref{eq:scalar_product_bound}, and in $\circledFive$ the triangle inequality. Finally, in $\circledSix$ we used the step-size condition $\eta \leq \frac{\lambda}{\cL_0 + \sqrt{\rho} \cL_1 \lambda} $ together with $\frac{\lambda}{\cL_0 + \sqrt{\rho} \cL_1 \lambda} \leq \frac{\norms{\nabla \func{x^k}} + \lambda}{\cL_0 + \sqrt{\rho} \cL_1 \norms{\nabla \func{x^k}}}$, which follows from
        \begin{align*}
            \frac{\norms{\nabla \func{x^k}} + \lambda}{\cL_0 + \sqrt{\rho} \cL_1 \norms{\nabla \func{x^k}}} - \frac{\lambda}{\cL_0 + \sqrt{\rho} \cL_1 \lambda} &= \frac{(\norms{\nabla \func{x^k}} + \lambda) (\cL_0 + \sqrt{\rho} \cL_1 \lambda) - \lambda (\cL_0 + \sqrt{\rho} \cL_1 \norms{\nabla \func{x^k}})}{(\cL_0 + \sqrt{\rho} \cL_1 \norms{\nabla \func{x^k}}) (\cL_0 + \sqrt{\rho} \cL_1 \lambda)} \\
            &= \frac{\cL_0 \norms{\nabla \func{x^k}} + \lambda^2 \sqrt{\rho} \cL_1}{(\cL_0 + \sqrt{\rho} \cL_1 \norms{\nabla \func{x^k}}) (\cL_0 + \sqrt{\rho} \cL_1 \lambda)} \geq 0,
        \end{align*}
        where $\cL_0, \cL_1, \lambda ,\norms{\nabla \func{x^k}} \geq 0$ and $\rho \geq 1$.

        Taking conditional expectation in \eqref{eq:Convex_NSGD_smooth} given $x^k$, we obtain:
        \begin{align}
            \expectcond{F_{k+1}}{x^k} - F_k &\overset{\eqref{eq:Convex_NSGD_smooth}}{\leq} -\frac{\eta}{2} \norms{\nabla \func{x^k}} + 2 \eta \expectcond{\norms{\nabla \func{x^k, \bxi^k} - \nabla \func{x^k}}}{x^k} + \frac{3}{2} \eta \lambda\nonumber\\
            &= -\frac{\eta}{2} \norms{\nabla \func{x^k}} + 2 \eta \expectcond{\left(\norms{\nabla \func{x^k, \bxi^k} - \nabla \func{x^k}}^2 \right)^{1/2}}{x^k} + \frac{3}{2} \eta \lambda\nonumber\\
            &\overset{\circledOne}{\leq} -\frac{\eta}{2} \norms{\nabla \func{x^k}} + 2 \eta \left(\expectcond{\norms{\nabla \func{x^k, \bxi^k} - \nabla \func{x^k}}^2}{x^k} \right)^{1/2} + \frac{3}{2} \eta \lambda\nonumber\\
            &\overset{\eqref{eq:variance_with_batch_size}}{\leq} -\frac{\eta}{2} \norms{\nabla \func{x^k}} + 2 \eta \left(\frac{1}{B} \left[ (\rho - 1) \norms{\nabla \func{x^k}}^2 + \sigma^2 \right] \right)^{1/2} + \frac{3}{2} \eta \lambda\nonumber\\
            &\overset{\circledTwo}{\leq} -\frac{\eta}{2} \norms{\nabla \func{x^k}} + 2 \eta \sqrt{\frac{(\rho - 1)}{B}}   \norms{\nabla \func{x^k}} + \frac{2\eta \sigma}{\sqrt{B}} + \frac{3}{2} \eta \lambda\nonumber\\
            &\overset{\circledThree}{\leq} -\frac{\eta}{2} \norms{\nabla \func{x^k}} + \frac{\eta}{4}   \norms{\nabla \func{x^k}} + \frac{2\eta \sigma}{\sqrt{B}} + \frac{3}{2} \eta \lambda\nonumber\\
            &= -\frac{\eta}{4} \norms{\nabla \func{x^k}} + \frac{2\eta \sigma}{\sqrt{B}} + \frac{3}{2} \eta \lambda, \label{eq:Convex_NSGD_SeMI_Itog}
        \end{align}
        where in $\circledOne$ we applied Jensen's inequality, in $\circledTwo$ we used $\sqrt{a^2 + b^2} \leq a + b$ when $a,b \geq 0$, in $\circledThree$ we chose $B \geq 64 (\rho - 1)$. Thus, we have obtained the inequality:
        \begin{align}
            \expectcond{F_{k+1}}{x^k} &\overset{\eqref{eq:Convex_NSGD_SeMI_Itog}}{\leq} F_k - \frac{\eta}{4} \norms{\nabla \func{x^k}} + \frac{2 \eta \sigma}{\sqrt{B}} + \frac{3 \eta}{2}\lambda \nonumber\\
            &\overset{\eqref{eq:function_gap_gradient_relation_NSGD}}{\leq} F_k - \frac{\eta}{4 R_0} F_k + \frac{2 \eta \sigma}{\sqrt{B}} + \frac{3\eta}{2}\lambda  \nonumber\\
            &=\left( 1 - \frac{\eta}{4 R_0} \right) F_k + \frac{2 \eta \sigma}{\sqrt{B}} + \frac{3\eta}{2}\lambda.
            \label{eq:Convex_NSGD_ITOG}
        \end{align}

        Taking full expectation in \eqref{eq:Convex_NSGD_ITOG} and applying the resulting recursion, we get:
        \begin{equation*}
            \expect{\func{x^{N}}} - f^* \leq \underbrace{\left( 1 - \frac{\eta}{4 R_0} \right)^N F_0}_{\circledOne} + \underbrace{6 \lambda R_0}_{\circledTwo}.
        \end{equation*}
        To guarantee convergence to the desired accuracy $\expect{\func{x^{N}}} - f^* \leq \varepsilon$, it is sufficient to bound the two terms above by $\varepsilon/2$. This yields the constraint:
        \begin{equation*}
            \circledTwo:\quad  \lambda \leq \frac{\varepsilon}{12 R_0}.
        \end{equation*}
        Taking $\lambda = \frac{\varepsilon}{12 R_0}$ and the maximum admissible step size $\eta = \frac{\lambda}{\cL_0 + \sqrt{\rho} \cL_1 \lambda} = \left[ \frac{12 \cL_0 R_0}{\varepsilon} + \sqrt{\rho} \cL_1 \right]^{-1}$, we obtain the following sufficient condition under batch size $B \geq 64 (\rho - 1)$:
        \begin{equation*}
            \circledOne: \quad \left( 1 - \frac{\eta}{4 R_0} \right)^N F_0 = \left( 1 - \frac{1}{ \frac{48 \cL_0 R_0^2}{\varepsilon} + 4 \sqrt{\rho} \cL_1 R_0} \right)^N F_0 \leq \frac{\varepsilon}{2}.
        \end{equation*}
        Thus, NSGD has the following values for the hyperparameter $\lambda$, the number of iterations $\# N$, and the batch size $\# B$ for convergence to the desired accuracy $\varepsilon$:
        \begin{equation*}
            \lambda = \frac{\varepsilon}{12 R_0}; \quad N \geq \left(\frac{48 \cL_0 R_0^2}{\varepsilon} + 4 \sqrt{\rho} \cL_1 R_0 \right) \log \frac{2 F_0}{\varepsilon}; \quad B \geq 64 (\rho - 1),
        \end{equation*}
        and the corresponding number of oracle calls $\# T$:
        \begin{equation*}
            T = N \cdot B = \Tilde{\mathcal{O}} \left( \frac{(\rho - 1) \cL_0 R_0^2}{\varepsilon} + (\rho - 1) \sqrt{\rho} \cL_1 R_0 \right), 
        \end{equation*}
        where we have used the notation $\Tilde{\mathcal{O}} \left( \cdot \right)$ to hide logarithmic factors.

\end{proof}

    \begin{remark}
            \label{rem:Step_size_Convex_NSGD}
            In the first step of~\eqref{eq:Convex_NSGD_smooth}, we use the descent lemma for $(\cL_0,\cL_1)$-smooth functions~\eqref{eq:L0L1_descent_inequality} \citep[see Appendix~A.1 in][]{Zhang_2020_Improved}; the condition $\norms{x^{k+1}-x^k}\leq \frac{1}{\sqrt{\rho}\cL_1}$ required for its application is guaranteed by the chosen step size $\eta \leq \frac{\lambda}{\cL_0 + \sqrt{\rho} \cL_1 \lambda}$:
            \begin{equation*}
                \norms{x^{k+1}-x^k} \overset{\circledOne}{=} \eta \norms{\Gf{x^k, \bxi^k}} \overset{\circledTwo}{\leq} \eta \overset{\circledThree}{\leq} \frac{\lambda}{\cL_0 + \sqrt{\rho} \cL_1 \lambda} \overset{\circledFour}{\leq} \frac{\lambda}{ \sqrt{\rho} \cL_1 \lambda} = \frac{1}{\sqrt{\rho} \cL_1},
            \end{equation*}
            where in $\circledOne$ we used $x^{k+1} = x^k - \eta \cdot \Gf{x^k, \bxi^k}$, in $\circledTwo$ we used $\norms{\Gf{x^k, \bxi^k}} \leq 1$, in $\circledThree$ we used $\eta \leq \frac{\lambda}{\cL_0 + \sqrt{\rho} \cL_1 \lambda}$, and in $\circledFour$ we used the fact that, for $0<b\leq a$, one has $\frac{1}{a}\leq \frac{1}{b}$.
        \end{remark}

    \subsubsection{Modification to the proof to include heavy-tailed noise}\label{app:NSGD_heavy_tailed_noise}

    If we assume that the stochastic noise is heavy-tailed, then the generalized strong growth condition \eqref{eq:strong growth condition} takes the following form for $\rho, \sigma \geq 0$ and $p \in (1,2)$:
    \begin{align*}
        \expect{\norms{\nabla \func{x, \xi}}^p} &=  \expect{\norms{\nabla \func{x, \xi} - \nabla \func{x} + \nabla \func{ x}}^p}\\
        &\overset{\circledOne}{\leq} 2^{p- 1} \expect{\norms{\nabla \func{x, \xi} - \nabla \func{x}}^p} + 2^{p- 1} \norms{\nabla \func{x}}^p\\
        &\overset{\text{Ass.~\ref{prop:SGC_under_heavy tailed noise}}}{\leq} 2^{p- 1} (\rho - 1) \norms{\nabla \func{x}}^p + 2^{p- 1} \sigma^p + 2^{p- 1} \norms{\nabla \func{x}}^p\\
        &= 2^{p- 1} \rho \norms{\nabla \func{x}}^p + 2^{p- 1} \sigma^p.
    \end{align*} 
    where in $\circledOne$ we applied Finite form of Jensen's inequality.

    If we use a mini-batch stochastic gradient $\nabla \func{x, \bxi} = \frac{1}{B} \sum_{i=1}^B \nabla \func{x, \xi_i}$, then Lemma~\ref{lem:heavy_tailed_batch} should be applied.
    
\begin{boxN}
    \begin{lemma}[\cite{Kornilov_2023,Hubler_2025}] \label{lem:heavy_tailed_batch}
        Let $p \in [1, 2]$, and $X_1, \dots, X_n \in \mathbb{R}^d$ be a martingale difference sequence (MDS), i.e.,
        \begin{equation*}
            \mathbb{E}[X_j \mid X_{j-1}, \dots, X_1] = 0 \quad \text{a.s. for all } j = 1, \dots, n,
        \end{equation*}
        satisfying $\mathbb{E}[\|X_j\|^p] < \infty$ for all $j = 1, \dots, n$. Define $S_n := \sum_{j=1}^n X_j$, then
        \begin{equation}
        \label{eq:heavy_tailed_batch}
            \mathbb{E}[\|S_n\|^p] \leq 2 \sum_{j=1}^n \mathbb{E}[\|X_j\|^p].
        \end{equation}
    \end{lemma}
\end{boxN}
    \paragraph{Historical background.} The result of Lemma~\ref{lem:heavy_tailed_batch} was initially proven by \cite{Bahr_1965} for $d = 1$. It was later extended to Banach spaces \citep[][Proposition 2.4]{Pisier_1975}, and optimal constants were derived \cite{Cox_1982}. An alternative proof for the one-dimensional i.i.d. case was rediscovered by \cite{Cherapanamjeri_2022} and later extended to higher dimensions by \cite{Kornilov_2023}. Later, \cite{Hubler_2025} corrected inaccuracies in the extension to higher dimensions.
    
    Thus, applying Lemma~\ref{lem:heavy_tailed_batch}, we obtain:
    \begin{align}
        \expect{\norms{\nabla \func{x, \bxi} - \nabla \func{x}}^p} &\overset{\circledOne}{=} \expect{ \norms{\frac{1}{B} \sum_{i=1}^B \nabla \func{x, \xi_i} - \nabla \func{x}}^p} \nonumber\\
        &\overset{\circledTwo}{=} \expect{ \norms{\frac{1}{B} \sum_{i=1}^B \left(\nabla \func{x, \xi_i} - \nabla \func{x} \right)}^p}\nonumber\\
        &= \frac{1}{B^p} \expect{ \norms{ \sum_{i=1}^B \left(\nabla \func{x, \xi_i} - \nabla \func{x} \right)}^p}\nonumber\\
        &\overset{\eqref{eq:heavy_tailed_batch}}{\leq} \frac{2}{B^p} \sum_{i=1}^B \expect{\norms{\nabla \func{x, \xi_i} - \nabla \func{x}}^p}\nonumber\\
        &\overset{\text{Ass.~\ref{prop:SGC_under_heavy tailed noise}}}{\leq} \frac{2}{B^p} \sum_{i=1}^B \left( (\rho - 1) \norms{\nabla \func{x}}^p + \sigma^p \right)\nonumber\\
        &=\frac{2}{B^{p-1}}  \left( (\rho - 1) \norms{\nabla \func{x}}^p + \sigma^p \right)\nonumber\\
        &=\frac{2 (\rho-1)}{B^{p-1}} \norms{\nabla \func{x}}^p + \frac{2 \sigma^p}{B^{p-1}},
        \label{eq:variance_heavy_tailed_noise_sgc}
    \end{align}
    where in $\circledOne$ we used $\nabla \func{x, \bxi} = \frac{1}{B} \sum_{i =1}^B \nabla \func{x, \xi_i}$, in $\circledTwo$ we used $\nabla \func{x} = \frac{1}{B} \sum_{i =1}^B \nabla \func{x}$.

    Then, in the proofs for NSGD, we use the following:
    \begin{align*}
        \expect{\norms{\nabla \func{x, \bxi} - \nabla \func{x}}} &= \expect{\left(\norms{\nabla \func{x, \bxi} - \nabla \func{x}}^p \right)^{\frac{1}{p}}}\\
        &\overset{\circledOne}{\leq}\left(\expect{\norms{\nabla \func{x, \bxi} - \nabla \func{x}}^p }\right)^{\frac{1}{p}}\\
        &\overset{\eqref{eq:variance_heavy_tailed_noise_sgc}}{\leq} \left( \frac{2 (\rho - 1)}{B^{p-1}} \norms{\nabla \func{x}}^p + \frac{2 \sigma^p}{B^{p-1}} \right)^{\frac{1}{p}}\\
        &\overset{\circledTwo}{\leq } \left( \frac{1}{8^p} \norms{\nabla \func{x}}^p + \frac{2 \sigma^p}{B^{p-1}} \right)^{\frac{1}{p}}\\
        &\overset{\circledThree}{\leq} \frac{1}{8} \norms{\nabla \func{x}} + \frac{2^{\frac{1}{p}} \sigma}{B^{\frac{p-1}{p}}},
    \end{align*}
    where in $\circledOne$ we applied Jensen's inequality, in $\circledTwo$ we chose $B \geq 2^{\frac{3p +1}{p-1}} (\rho-1)^{\frac{1}{p-1}}$, in $\circledThree$ we applied inequality $|a+b|^{1/p} \leq |a|^{1/p} + |b|^{1/p}$, and by performing the remaining steps of the proof identically, we obtain the following final convergence rates for \textit{non-convex setup under $(L_0,L_1)$-smooth and generalized strong growth conditions} with $\eta \leq \frac{\lambda}{L_0 + L_1  \lambda}$:
    \begin{equation*}
        \min_{k\in [0, N-1]}\expect{\norms{\nabla \func{x^k}}} \lesssim \frac{F_0}{\eta N} + \frac{\sigma}{B^{\frac{p-1}{p}}} + \lambda,
    \end{equation*}
    and for \textit{convex setup under $(\cL_0,\cL_1)$-smooth and strong growth conditions} with $\eta \leq \frac{\lambda}{\cL_0 + \sqrt{\rho} \cL_1 \lambda}$:
    \begin{equation*}
        \expect{\func{x^{N}}} - f^* \lesssim \left( 1 - \frac{\eta}{R_0} \right)^N F_0 + \lambda R_0,
    \end{equation*}
    with batch size $B \gtrsim  (\rho-1)^{\frac{1}{p-1}}$ and hyperparameter $\lambda>0$.

\section{Extending the Analysis to \texorpdfstring{$(\cH_0,\cH_1)$}{TEXT}-Smoothness} 
\label{app:Extending the Analysis to H_0 H_1 Smoothness}

In this section, we extend the analysis of Section~\ref{subsec:Convex functions} to functions satisfying $(\cH_0,\cH_1)$-smoothness \cite{Vaswani_2025,Liu_2025,Alimisis_2025}. We show that, under stochastic $(\cH_0,\cH_1)$-smoothness and the interpolation condition, one can recover deterministic $(H_0,H_1)$-smoothness with the same constants, i.e., $H_0=\cH_0$ and $H_1=\cH_1$. We also complement this extension with an analysis of $(H_0,H_1)$-smoothness in the deterministic setting, focusing on gradient descent under different convexity assumptions.

\subsection{Assumptions on the objective function}

\paragraph{Smoothness.} We begin by introducing the stochastic $(\cH_0,\cH_1)$-smoothness assumption. This condition relates the local smoothness of each stochastic realization to its suboptimality gap.

\begin{boxF}
\begin{assumption}[Optional]\label{ass:H_0H_1_Smooth_for_xi}
    Let $f(\cdot,\xi)$ be $(\cH_0,\cH_1)$-smooth for almost every $\xi \sim \mathcal{D}$. That is, for almost every $\xi \sim \mathcal{D}$ and all $x,y \in \mathbb{R}^d$ such that $\norms{y-x} \leq \frac{1}{\sqrt{\cH_1}}$, we have
    \begin{equation*}
        \norms{\nabla \func{y,\xi} - \nabla \func{x,\xi}}
        \leq
        \left(
            \cH_0
            +
            \cH_1
            \left[
                \func{x,\xi} - f_\xi^*
            \right]
        \right)
        \norms{y-x},
    \end{equation*}
    where $f_\xi^* := \inf_{x \in \mathbb{R}^d} \func{x,\xi}$.
\end{assumption}
\end{boxF}

We compare Assumption~\ref{ass:H_0H_1_Smooth_for_xi} with its deterministic counterpart.

\begin{boxF}
\begin{assumption}[Optional]\label{ass:H_0H_1_Smooth}
    Let $f$ be $(H_0,H_1)$-smooth. That is, for all $x,y \in \mathbb{R}^d$ such that $\norms{y-x} \leq \frac{1}{\sqrt{H_1}}$, we have
    \begin{equation*}
        \norms{\nabla \func{y} - \nabla \func{x}}
        \leq
        \left(
            H_0
            +
            H_1
            \left[
                \func{x} - f^*
            \right]
        \right)
        \norms{y-x}.
    \end{equation*}
\end{assumption}
\end{boxF}

The following lemma shows that, under the interpolation condition, the stochastic $(\cH_0,\cH_1)$-smoothness assumption implies its deterministic counterpart with the same constants.

\begin{boxN}
\begin{lemma}
\label{lem:stochastic_H0H1_implies_deterministic_H0H1}
Let Assumption~\ref{ass:unbiased} hold and suppose that Assumption~\ref{ass:H_0H_1_Smooth_for_xi} is satisfied. Assume additionally that the interpolation condition from Definition~\ref{def:interpolation} holds. Then $f$ satisfies Assumption~\ref{ass:H_0H_1_Smooth} with the same constants
\begin{equation*}
    H_0 = \cH_0,
    \qquad
    H_1 = \cH_1.
\end{equation*}
That is, for all $x,y \in \mathbb{R}^d$ such that $\norms{y-x} \leq \frac{1}{\sqrt{H_1}}$, we have
\begin{equation*}
    \norms{\nabla \func{y} - \nabla \func{x}}
    \leq
    \left(
        H_0
        +
        H_1
        \left[
            \func{x} - f^*
        \right]
    \right)
    \norms{y-x}.
\end{equation*}
\end{lemma}
\end{boxN}

\begin{proof}
By Assumption~\ref{ass:unbiased}, for any $x \in \mathbb{R}^d$,
\begin{equation*}
    \expect{\nabla \func{x,\xi}} = \nabla \func{x}.
\end{equation*}
Therefore, for any $x,y \in \mathbb{R}^d$, we have
\begin{align}
    \norms{\nabla \func{y} - \nabla \func{x}}
    &=
    \norms{
        \expect{
            \nabla \func{y,\xi}
            -
            \nabla \func{x,\xi}
        }
    }
    \nonumber \\
    &\leq
    \expect{
        \norms{
            \nabla \func{y,\xi}
            -
            \nabla \func{x,\xi}
        }
    },
    \label{eq:H0H1_jensen}
\end{align}
where we used Jensen's inequality.

Now take arbitrary $x,y \in \mathbb{R}^d$ such that $\norms{y-x} \leq \frac{1}{\sqrt{\cH_1}}$. Applying Assumption~\ref{ass:H_0H_1_Smooth_for_xi} inside the expectation in~\eqref{eq:H0H1_jensen}, we obtain
\begin{align}
    \norms{\nabla \func{y} - \nabla \func{x}}
    &\leq
    \expect{
        \left(
            \cH_0
            +
            \cH_1
            \left[
                \func{x,\xi} - f_\xi^*
            \right]
        \right)
        \norms{y-x}
    }
    \nonumber \\
    &=
    \left(
        \cH_0
        +
        \cH_1
        \expect{
            \func{x,\xi} - f_\xi^*
        }
    \right)
    \norms{y-x}
    \nonumber \\
    &=
    \left(
        \cH_0
        +
        \cH_1
        \left[
            \func{x}
            -
            \expect{f_\xi^*}
        \right]
    \right)
    \norms{y-x}.
    \label{eq:H0H1_before_interpolation}
\end{align}

It remains to show that $\expect{f_\xi^*}=f^*$. Let $x^* \in \argmin f$. By the interpolation condition from Definition~\ref{def:interpolation}, we have
\begin{equation*}
    x^* \in \argmin f(\cdot,\xi)
    \quad
    \text{a.s.}
\end{equation*}
Hence,
\begin{equation*}
    f_\xi^*
    =
    \func{x^*,\xi}
    \quad
    \text{a.s.}
\end{equation*}
Taking expectation with respect to $\xi$, we get
\begin{equation*}
    \expect{f_\xi^*}
    =
    \expect{\func{x^*,\xi}}
    =
    \func{x^*}
    =
    f^*.
\end{equation*}
Substituting this identity into~\eqref{eq:H0H1_before_interpolation}, we obtain
\begin{equation*}
    \norms{\nabla \func{y} - \nabla \func{x}}
    \leq
    \left(
        \cH_0
        +
        \cH_1
        \left[
            \func{x} - f^*
        \right]
    \right)
    \norms{y-x}.
\end{equation*}
Finally, setting $H_0=\cH_0$ and $H_1=\cH_1$, we obtain
\begin{equation*}
    \norms{\nabla \func{y} - \nabla \func{x}}
    \leq
    \left(
        H_0
        +
        H_1
        \left[
            \func{x} - f^*
        \right]
    \right)
    \norms{y-x},
\end{equation*}
for all $x,y \in \mathbb{R}^d$ such that $\norms{y-x} \leq \frac{1}{\sqrt{H_1}}$.
This completes the proof.
\end{proof}

Assumption~\ref{ass:H_0H_1_Smooth} is strictly weaker than Assumption~\ref{ass:L_0_L_1Smooth}, and consequently also strictly weaker than classical $L$-smoothness. To see this, Proposition~B.1 of \cite{Alimisis_2025} establishes that every $(L_0,L_1)$-smooth function is also $(H_0,H_1)$-smooth with $H_0 = L_0 + L_0 L_1$ and $H_1 = \frac{4 L_1^2 + L_1}{2}$. Moreover, Example~1 of \cite{Liu_2025} provides a simple function, motivated by deep neural network training, that satisfies $(H_0,H_1)$-smoothness but does not satisfy $(L_0,L_1)$-smoothness. Thus, $(H_0,H_1)$-smoothness captures a broader class of objectives than the $(L_0,L_1)$-smoothness condition.

Similarly to $(L_0,L_1)$-smoothness, Assumption~\ref{ass:H_0H_1_Smooth} yields the following descent-type inequality: for all $x,y \in \mathbb{R}^d$ such that $\norms{y-x}\leq \frac{1}{\sqrt{H_1}}$,
\begin{equation}
    \label{eq:H0H1_descent_inequality}
    f(y) - f(x)
    \leq
    \dotprod{\nabla f(x)}{y-x}
    +
    \frac{H_0 + H_1 \left( \func{x} - f^* \right)}{2}
    \norms{y-x}^2.
\end{equation}

\paragraph{Convexity.} We also impose convexity at the level of stochastic realizations. The next lemma shows that this assumption immediately implies convexity of the expected objective.

\begin{boxF}
\begin{assumption}[Optional]\label{ass:Convexity_for_xi}
    Let $f(\cdot,\xi)$ be convex for almost every $\xi \sim \mathcal{D}$. That is, for almost every $\xi \sim \mathcal{D}$ and all $x,y \in \mathbb{R}^d$, we have
    \begin{equation*}
        \func{y,\xi}
        \geq
        \func{x,\xi}
        +
        \dotprod{\nabla \func{x,\xi}}{y-x}.
    \end{equation*}
\end{assumption}
\end{boxF}

\begin{boxN}
\begin{lemma}
\label{lem:stochastic_convexity_implies_deterministic_convexity}
Let Assumptions~\ref{ass:Convexity_for_xi} and~\ref{ass:unbiased} hold. Then $f$ is convex, i.e., for all $x,y \in \mathbb{R}^d$, we have
\begin{equation*}
    \func{y}
    \geq
    \func{x}
    +
    \dotprod{\nabla \func{x}}{y-x}.
\end{equation*}
\end{lemma}
\end{boxN}

\begin{proof}
By Assumption~\ref{ass:Convexity_for_xi}, for almost every $\xi \sim \mathcal{D}$ and all $x,y \in \mathbb{R}^d$, we have
\begin{equation*}
    \func{y,\xi}
    \geq
    \func{x,\xi}
    +
    \dotprod{\nabla \func{x,\xi}}{y-x}.
\end{equation*}
Taking expectation with respect to $\xi$ on both sides, we obtain
\begin{align}
    \expect{\func{y,\xi}}
    &\geq
    \expect{\func{x,\xi}}
    +
    \expect{
        \dotprod{\nabla \func{x,\xi}}{y-x}
    }
    \nonumber \\
    &=
    \expect{\func{x,\xi}}
    +
    \dotprod{
        \expect{\nabla \func{x,\xi}}
    }{y-x}.
\end{align}
Using the definitions $\func{x}=\expect{\func{x,\xi}}$ and $\func{y}=\expect{\func{y,\xi}}$, together with Assumption~\ref{ass:unbiased}, we get
\begin{equation*}
    \func{y}
    \geq
    \func{x}
    +
    \dotprod{\nabla \func{x}}{y-x}.
\end{equation*}
Therefore, $f$ is convex.
\end{proof}

\subsection{Implication of convexity and monotone behavior of the distance to optimum}
\label{app:Implication of convexity and monotonicity_H0_H1}

In this subsection, we derive a useful consequence of convexity and show that, under the corresponding step-size conditions, the distance to the optimum is non-increasing for both ClipSGD and NSGD in the $(\cH_0,\cH_1)$-smooth setting.

\paragraph{Implication of convexity.}
If Assumptions~\ref{ass:Convexity_for_xi} and~\ref{ass:unbiased} hold, then by Lemma~\ref{lem:stochastic_convexity_implies_deterministic_convexity}, the expected objective $f$ is convex. Therefore, for all $x,y \in \mathbb{R}^d$, we have
\begin{equation}
    \label{eq:Convexity_H0H1}
    \func{x} - \func{y}
    \leq
    \dotprod{\nabla \func{x}}{x-y}.
\end{equation}
Taking $y=x^*$, where $x^* \in \argmin f$, we obtain the following relation between the function suboptimality gap and the gradient norm:
\begin{align}
    \func{x^k} - f^*
    &\overset{\eqref{eq:Convexity_H0H1}}{\leq}
    \dotprod{\nabla \func{x^k}}{x^k - x^*}
    \overset{\eqref{eq:scalar_product_bound}}{\leq}
    \norms{\nabla \func{x^k}}
    \norms{x^k - x^*}
    \overset{\circledOne}{\leq}
    \norms{\nabla \func{x^k}}
    \underbrace{\norms{x^0 - x^*}}_{=:R_0}.
    \label{eq:relationship_between_function_gap_and_gradient_norm_H0H1}
\end{align}
Here, in $\circledOne$ we used the monotonicity of the distance to the optimum established below for ClipSGD and NSGD.

For the proofs below, we use the following standard self-bounding consequence of convexity and $(\cH_0,\cH_1)$-smoothness. For any mini-batch $\bxi^k$, define
\begin{equation*}
    F_k(\bxi^k)
    :=
    \func{x^k,\bxi^k}
    -
    \func{x^*,\bxi^k}.
\end{equation*}
Under interpolation, $x^*$ minimizes each realization almost surely, and hence also minimizes the mini-batch objective $\func{\cdot,\bxi^k}$. Therefore, $F_k(\bxi^k)\geq 0$. Moreover,
\begin{equation}
    \label{eq:self_bounding_H0H1_batch}
    \norms{\nabla \func{x^k,\bxi^k}}^2
    \leq
    2
    \left(
        \cH_0
        +
        \cH_1 F_k(\bxi^k)
    \right)
    F_k(\bxi^k).
\end{equation}
We also use convexity of $\func{\cdot,\bxi^k}$ and interpolation, which imply
\begin{equation}
    \label{eq:convexity_batch_H0H1}
    F_k(\bxi^k)
    \leq
    \dotprod{
        \nabla \func{x^k,\bxi^k}
    }{
        x^k-x^*
    }.
\end{equation}

% \paragraph{Monotone behavior of the distance to optimum for ClipSGD.}

\begin{boxN}
\begin{lemma}
\label{lem:monotone_clipsgd_H0_H1}
Let Assumptions~\ref{ass:H_0H_1_Smooth_for_xi}, \ref{ass:Convexity_for_xi}, and the interpolation condition from Definition~\ref{def:interpolation} hold. Let $H_0=\cH_0$ and $H_1=\cH_1$. Then ClipSGD from Algorithm~\ref{algo:clip_algo} with step size
\begin{equation*}
    \eta
    \leq
    \left[
        9
        \left(
            H_0
            +
            \max\left\{\sqrt{H_1},H_1R_0\right\}c
        \right)
    \right]^{-1},
\end{equation*}
where $R_0:=\norms{x^0-x^*}$, guarantees
\begin{equation*}
    \norms{x^{k+1}-x^*}^2
    \leq
    \norms{x^k-x^*}^2.
\end{equation*}
Consequently,
\begin{equation*}
    \norms{x^k-x^*}
    \leq
    \norms{x^0-x^*}
    =
    R_0
    \quad
    \text{for all } k\geq 0.
\end{equation*}
\end{lemma}
\end{boxN}

\begin{proof}
We prove the statement by induction. The claim is trivial for $k=0$. Assume that
\begin{equation*}
    \norms{x^k-x^*}
    \leq
    R_0.
\end{equation*}
We show that
\begin{equation*}
    \norms{x^{k+1}-x^*}
    \leq
    \norms{x^k-x^*}.
\end{equation*}

For a mini-batch $\bxi^k$, define
\begin{equation*}
    F_k(\bxi^k)
    :=
    \func{x^k,\bxi^k}
    -
    \func{x^*,\bxi^k}.
\end{equation*}
By the interpolation condition, $x^*$ minimizes each stochastic realization almost surely, and hence it also minimizes the mini-batch objective. Therefore,
\begin{equation*}
    F_k(\bxi^k)\geq 0.
\end{equation*}
Moreover, by convexity,
\begin{equation}
    \label{eq:convexity_batch_H0H1_Clip}
    F_k(\bxi^k)
    \leq
    \dotprod{\nabla \func{x^k,\bxi^k}}{x^k-x^*}.
\end{equation}
Using the induction hypothesis and Cauchy's inequality, we also get
\begin{align}
    F_k(\bxi^k)
    &\leq
    \dotprod{\nabla \func{x^k,\bxi^k}}{x^k-x^*}\leq
    \norms{\nabla \func{x^k,\bxi^k}}
    \norms{x^k-x^*}\leq
    \norms{\nabla \func{x^k,\bxi^k}}R_0.
    \label{eq:F_bound_by_gradient_R0_H0H1_Clip}
\end{align}

The ClipSGD update has the form
\begin{equation*}
    x^{k+1}
    =
    x^k
    -
    \eta\alpha_{\bxi^k}\nabla \func{x^k,\bxi^k},
\end{equation*}
where
\begin{equation*}
    \alpha_{\bxi^k}
    =
    \min
    \left\{
        1,
        \frac{c}{\norms{\nabla \func{x^k,\bxi^k}}}
    \right\}.
\end{equation*}
Then
\begin{align}
    \norms{x^{k+1}-x^*}^2
    &=
    \norms{
        x^k-x^*
        -
        \eta\alpha_{\bxi^k}\nabla \func{x^k,\bxi^k}
    }^2
    \nonumber \\
    &=
    \norms{x^k-x^*}^2
    -
    2\eta\alpha_{\bxi^k}
    \dotprod{\nabla \func{x^k,\bxi^k}}{x^k-x^*}
    +
    \eta^2\alpha_{\bxi^k}^2
    \norms{\nabla \func{x^k,\bxi^k}}^2
    \nonumber \\
    &\overset{\eqref{eq:convexity_batch_H0H1_Clip}}{\leq}
    \norms{x^k-x^*}^2
    -
    2\eta\alpha_{\bxi^k}
    F_k(\bxi^k)
    +
    \eta^2\alpha_{\bxi^k}^2
    \norms{\nabla \func{x^k,\bxi^k}}^2.
    \label{eq:ClipSGD_H0H1_monotone_start}
\end{align}

By the self-bounding property of convex $(H_0,H_1)$-smooth functions applied to the mini-batch objective, we have
\begin{equation}
    \label{eq:self_bounding_H0H1_Clip}
    \norms{\nabla \func{x^k,\bxi^k}}^2
    \leq
    2
    \left(
        H_0
        +
        H_1F_k(\bxi^k)
    \right)
    F_k(\bxi^k).
\end{equation}
Substituting~\eqref{eq:self_bounding_H0H1_Clip} into~\eqref{eq:ClipSGD_H0H1_monotone_start}, we obtain
\begin{align}
    \norms{x^{k+1}-x^*}^2
    &\leq
    \norms{x^k-x^*}^2
    -
    2\eta\alpha_{\bxi^k}
    F_k(\bxi^k)
    \nonumber \\
    &\quad\quad\quad
    +
    2\eta^2\alpha_{\bxi^k}^2
    \left(
        H_0
        +
        H_1F_k(\bxi^k)
    \right)
    F_k(\bxi^k).
    \label{eq:ClipSGD_H0H1_monotone_main}
\end{align}

We now consider two cases.

\textbf{Case 1: $\norms{\nabla \func{x^k,\bxi^k}}\leq c$.}
In this case, $\alpha_{\bxi^k}=1$. From~\eqref{eq:F_bound_by_gradient_R0_H0H1_Clip}, we have
\begin{equation*}
    F_k(\bxi^k)
    \leq
    cR_0.
\end{equation*}
Therefore,
\begin{equation*}
    H_0+H_1F_k(\bxi^k)
    \leq
    H_0+H_1cR_0
    \leq
    H_0+\max\left\{\sqrt{H_1},H_1R_0\right\}c.
\end{equation*}
Using this bound in~\eqref{eq:ClipSGD_H0H1_monotone_main}, we obtain
\begin{align*}
    \norms{x^{k+1}-x^*}^2
    &\leq
    \norms{x^k-x^*}^2
    -
    2\eta F_k(\bxi^k)
    \\
    &\quad\quad\quad
    +
    2\eta^2
    \left(
        H_0+\max\left\{\sqrt{H_1},H_1R_0\right\}c
    \right)
    F_k(\bxi^k)
    \\
    &\leq
    \norms{x^k-x^*}^2
    -
    2\eta F_k(\bxi^k)
    +
    \frac{2}{9}\eta F_k(\bxi^k)
    \\
    &=
    \norms{x^k-x^*}^2
    -
    \frac{16}{9}\eta F_k(\bxi^k)
    \\
    &\leq
    \norms{x^k-x^*}^2.
\end{align*}

\textbf{Case 2: $\norms{\nabla \func{x^k,\bxi^k}}>c$.}
In this case,
\begin{equation*}
    \alpha_{\bxi^k}
    =
    \frac{c}{\norms{\nabla \func{x^k,\bxi^k}}}.
\end{equation*}
From~\eqref{eq:F_bound_by_gradient_R0_H0H1_Clip}, we get
\begin{equation*}
    \alpha_{\bxi^k}F_k(\bxi^k)
    \leq
    cR_0.
\end{equation*}
Since $\alpha_{\bxi^k}\leq 1$, we have
\begin{align*}
    \alpha_{\bxi^k}
    \left(
        H_0+H_1F_k(\bxi^k)
    \right)
    &=
    \alpha_{\bxi^k}H_0
    +
    H_1\alpha_{\bxi^k}F_k(\bxi^k)
    \\
    &\leq
    H_0+H_1cR_0
    \\
    &\leq
    H_0+\max\left\{\sqrt{H_1},H_1R_0\right\}c.
\end{align*}
Using this bound in~\eqref{eq:ClipSGD_H0H1_monotone_main}, we obtain
\begin{align*}
    \norms{x^{k+1}-x^*}^2
    &\leq
    \norms{x^k-x^*}^2
    -
    2\eta\alpha_{\bxi^k}F_k(\bxi^k)
    \\
    &\quad\quad\quad
    +
    2\eta^2\alpha_{\bxi^k}
    \left[
        \alpha_{\bxi^k}
        \left(
            H_0+H_1F_k(\bxi^k)
        \right)
    \right]
    F_k(\bxi^k)
    \\
    &\leq
    \norms{x^k-x^*}^2
    -
    2\eta\alpha_{\bxi^k}F_k(\bxi^k)
    \\
    &\quad\quad\quad
    +
    2\eta^2\alpha_{\bxi^k}
    \left(
        H_0+\max\left\{\sqrt{H_1},H_1R_0\right\}c
    \right)
    F_k(\bxi^k)
    \\
    &\leq
    \norms{x^k-x^*}^2
    -
    2\eta\alpha_{\bxi^k}F_k(\bxi^k)
    +
    \frac{2}{9}\eta\alpha_{\bxi^k}F_k(\bxi^k)
    \\
    &=
    \norms{x^k-x^*}^2
    -
    \frac{16}{9}\eta\alpha_{\bxi^k}F_k(\bxi^k)
    \\
    &\leq
    \norms{x^k-x^*}^2.
\end{align*}

Combining both cases, we conclude that
\begin{equation*}
    \norms{x^{k+1}-x^*}^2
    \leq
    \norms{x^k-x^*}^2.
\end{equation*}
Thus, by induction,
\begin{equation*}
    \norms{x^k-x^*}
    \leq
    \norms{x^0-x^*}
    =
    R_0
    \quad
    \text{for all } k\geq 0.
\end{equation*}

Finally, we note that the step-size condition also implies
\begin{equation*}
    \eta c
    \leq
    \frac{c}{
        9
        \left(
            H_0+\sqrt{H_1}c
        \right)
    }
    \leq
    \frac{1}{\sqrt{H_1}},
\end{equation*}
which ensures that all ClipSGD updates remain within the locality radius required by the $(H_0,H_1)$-smoothness condition.
This completes the proof.
\end{proof}

% \paragraph{Monotone behavior of the distance to optimum for NSGD.}

\begin{boxN}
\begin{lemma}
\label{lem:monotone_nsgd_H0_H1}
Let Assumptions~\ref{ass:H_0H_1_Smooth_for_xi}, \ref{ass:Convexity_for_xi}, and the interpolation condition from Definition~\ref{def:interpolation} hold. Let $H_0=\cH_0$ and $H_1=\cH_1$. Then NSGD from Algorithm~\ref{algo:clip_algo} with any $\lambda>0$ and step size
\begin{equation*}
    \eta
    \leq
    \frac{\lambda}{
        H_0
        +
        \max\left\{\sqrt{H_1},H_1R_0\right\}\lambda
    },
\end{equation*}
where $R_0:=\norms{x^0-x^*}$, guarantees
\begin{equation*}
    \norms{x^{k+1}-x^*}^2
    \leq
    \norms{x^k-x^*}^2.
\end{equation*}
Consequently,
\begin{equation*}
    \norms{x^k-x^*}
    \leq
    \norms{x^0-x^*}
    =
    R_0
    \quad
    \text{for all } k\geq 0.
\end{equation*}
\end{lemma}
\end{boxN}

\begin{proof}
We prove the statement by induction. The claim is trivial for $k=0$. Assume that
\begin{equation*}
    \norms{x^k-x^*}
    \leq
    R_0.
\end{equation*}
We show that
\begin{equation*}
    \norms{x^{k+1}-x^*}
    \leq
    \norms{x^k-x^*}.
\end{equation*}

For a mini-batch $\bxi^k$, define
\begin{equation*}
    F_k(\bxi^k)
    :=
    \func{x^k,\bxi^k}
    -
    \func{x^*,\bxi^k}.
\end{equation*}
By the interpolation condition, $x^*$ minimizes each stochastic realization almost surely, and hence it also minimizes the mini-batch objective. Therefore,
\begin{equation*}
    F_k(\bxi^k)\geq 0.
\end{equation*}
Moreover, by convexity,
\begin{equation}
    \label{eq:convexity_batch_H0H1_NSGD}
    F_k(\bxi^k)
    \leq
    \dotprod{\nabla \func{x^k,\bxi^k}}{x^k-x^*}.
\end{equation}
Using the induction hypothesis and Cauchy's inequality, we also get
\begin{align}
    F_k(\bxi^k)
    &\leq
    \dotprod{\nabla \func{x^k,\bxi^k}}{x^k-x^*}\leq
    \norms{\nabla \func{x^k,\bxi^k}}
    \norms{x^k-x^*}\leq
    \norms{\nabla \func{x^k,\bxi^k}} R_0.
    \label{eq:F_bound_by_gradient_R0_H0H1_NSGD}
\end{align}

Introduce the NSGD direction
\begin{equation*}
    \Gf{x^k,\bxi^k}
    \coloneqq
    \frac{\nabla \func{x^k,\bxi^k}}
    {\norms{\nabla \func{x^k,\bxi^k}}+\lambda}.
\end{equation*}
Then the update rule of NSGD can be written as
\begin{equation*}
    x^{k+1}
    =
    x^k
    -
    \eta \Gf{x^k,\bxi^k}.
\end{equation*}
Hence,
\begin{align}
    \norms{x^{k+1}-x^*}^2
    &=
    \norms{x^k-x^*-\eta \Gf{x^k,\bxi^k}}^2
    \nonumber \\
    &=
    \norms{x^k-x^*}^2
    -
    2\eta
    \dotprod{\Gf{x^k,\bxi^k}}{x^k-x^*}
    +
    \eta^2
    \norms{\Gf{x^k,\bxi^k}}^2
    \nonumber \\
    &=
    \norms{x^k-x^*}^2
    -
    \frac{2\eta}{
        \norms{\nabla \func{x^k,\bxi^k}}+\lambda
    }
    \dotprod{\nabla \func{x^k,\bxi^k}}{x^k-x^*}
    \nonumber \\
    &\quad\quad\quad
    +
    \eta^2
    \left(
        \frac{
            \norms{\nabla \func{x^k,\bxi^k}}
        }{
            \norms{\nabla \func{x^k,\bxi^k}}+\lambda
        }
    \right)^2
    \nonumber \\
    &\overset{\eqref{eq:convexity_batch_H0H1_NSGD}}{\leq}
    \norms{x^k-x^*}^2
    -
    \frac{2\eta}{
        \norms{\nabla \func{x^k,\bxi^k}}+\lambda
    }
    F_k(\bxi^k)
    \nonumber \\
    &\quad\quad\quad
    +
    \eta^2
    \left(
        \frac{
            \norms{\nabla \func{x^k,\bxi^k}}
        }{
            \norms{\nabla \func{x^k,\bxi^k}}+\lambda
        }
    \right)^2.
    \label{eq:NSGD_H0H1_monotone_start}
\end{align}

By the self-bounding property of convex $(H_0,H_1)$-smooth functions applied to the mini-batch objective, we have
\begin{equation}
    \label{eq:self_bounding_H0H1_NSGD}
    \norms{\nabla \func{x^k,\bxi^k}}^2
    \leq
    2
    \left(
        H_0
        +
        H_1F_k(\bxi^k)
    \right)
    F_k(\bxi^k).
\end{equation}
Substituting~\eqref{eq:self_bounding_H0H1_NSGD} into~\eqref{eq:NSGD_H0H1_monotone_start}, we obtain
\begin{align}
    \norms{x^{k+1}-x^*}^2
    &\leq
    \norms{x^k-x^*}^2
    -
    \frac{2\eta}{
        \norms{\nabla \func{x^k,\bxi^k}}+\lambda
    }
    F_k(\bxi^k)
    \nonumber \\
    &\quad\quad\quad
    +
    \frac{
        2\eta^2
        \left(
            H_0
            +
            H_1F_k(\bxi^k)
        \right)
    }{
        \left(
            \norms{\nabla \func{x^k,\bxi^k}}+\lambda
        \right)^2
    }
    F_k(\bxi^k)
    \nonumber \\
    &=
    \norms{x^k-x^*}^2
    -
    \frac{
        2\eta F_k(\bxi^k)
    }{
        \norms{\nabla \func{x^k,\bxi^k}}+\lambda
    }
    \left(
        1
        -
        \eta
        \frac{
            H_0
            +
            H_1F_k(\bxi^k)
        }{
            \norms{\nabla \func{x^k,\bxi^k}}+\lambda
        }
    \right).
    \label{eq:NSGD_H0H1_monotone_main}
\end{align}

It remains to show that the term in parentheses is non-negative. From~\eqref{eq:F_bound_by_gradient_R0_H0H1_NSGD}, we have
\begin{equation*}
    F_k(\bxi^k)
    \leq
    R_0\norms{\nabla \func{x^k,\bxi^k}}.
\end{equation*}
Therefore,
\begin{align}
    \frac{
        H_0
        +
        H_1F_k(\bxi^k)
    }{
        \norms{\nabla \func{x^k,\bxi^k}}+\lambda
    }
    &\leq
    \frac{
        H_0
        +
        H_1R_0\norms{\nabla \func{x^k,\bxi^k}}
    }{
        \norms{\nabla \func{x^k,\bxi^k}}+\lambda
    }
    \nonumber \\
    &\leq
    \frac{
        H_0
        +
        H_1R_0\lambda
    }{
        \lambda
    }.
    \label{eq:NSGD_stepsize_bound_H0H1}
\end{align}
Indeed, the last inequality follows from
\begin{align*}
    &
    \frac{
        H_0
        +
        H_1R_0\lambda
    }{
        \lambda
    }
    -
    \frac{
        H_0
        +
        H_1R_0\norms{\nabla \func{x^k,\bxi^k}}
    }{
        \norms{\nabla \func{x^k,\bxi^k}}+\lambda
    }
    \\
    &\quad =
    \frac{
        H_0\norms{\nabla \func{x^k,\bxi^k}}
        +
        H_1R_0\lambda^2
    }{
        \lambda
        \left(
            \norms{\nabla \func{x^k,\bxi^k}}+\lambda
        \right)
    }
    \geq 0.
\end{align*}
Using the step-size condition
\begin{equation*}
    \eta
    \leq
    \frac{\lambda}{
        H_0
        +
        \max\left\{\sqrt{H_1},H_1R_0\right\}\lambda
    }
    \leq
    \frac{\lambda}{
        H_0
        +
        H_1R_0\lambda
    },
\end{equation*}
together with~\eqref{eq:NSGD_stepsize_bound_H0H1}, we obtain
\begin{equation*}
    1
    -
    \eta
    \frac{
        H_0
        +
        H_1F_k(\bxi^k)
    }{
        \norms{\nabla \func{x^k,\bxi^k}}+\lambda
    }
    \geq
    0.
\end{equation*}
Substituting this into~\eqref{eq:NSGD_H0H1_monotone_main}, and using $F_k(\bxi^k)\geq 0$, we conclude that
\begin{equation*}
    \norms{x^{k+1}-x^*}^2
    \leq
    \norms{x^k-x^*}^2.
\end{equation*}
Thus, by induction,
\begin{equation*}
    \norms{x^k-x^*}
    \leq
    \norms{x^0-x^*}
    =
    R_0
    \quad
    \text{for all } k\geq 0.
\end{equation*}

Finally, we note that the step-size condition also implies
\begin{equation*}
    \eta
    \leq
    \frac{\lambda}{
        H_0+\sqrt{H_1}\lambda
    }
    \leq
    \frac{1}{\sqrt{H_1}}.
\end{equation*}
Moreover,
\begin{equation*}
    \norms{x^{k+1}-x^k}
    =
    \eta
    \norms{
        \frac{\nabla \func{x^k,\bxi^k}}
        {\norms{\nabla \func{x^k,\bxi^k}}+\lambda}
    }
    \leq
    \eta
    \leq
    \frac{1}{\sqrt{H_1}}.
\end{equation*}
Thus, all NSGD updates remain within the locality radius required by the $(H_0,H_1)$-smoothness condition.
This completes the proof.
\end{proof}

We next consider two settings: a stochastic setting, where we analyze ClipSGD and NSGD (see Algorithm~\ref{algo:clip_algo}) for convex $(H_0,H_1)$-smooth objectives, and a deterministic setting, where we analyze GD (see Algorithm~\ref{algo:GD}) under various convexity assumptions.

\subsection{Stochastic setting}

Below, we present the convergence result for ClipSGD.

\begin{boxA}
\begin{theorem}[Convex, ClipSGD]\label{th:Convex_ClipSGD_H_0H_1}
Let $f(\cdot,\xi)$ be $(\cH_0,\cH_1)$-smooth
(Assumption~\ref{ass:H_0H_1_Smooth_for_xi}) and convex
(Assumption~\ref{ass:Convexity_for_xi}). Suppose that the gradient oracle is unbiased
(Assumption~\ref{ass:unbiased}) and satisfies SGC
(Assumption~\ref{ass:strong growth condition} with $\sigma=0$). Then ClipSGD (Algorithm~\ref{algo:clip_algo}) with step size
$
    \eta
    \leq
    \left[
        9
        \left(
            H_0
            +
            \max\left\{\sqrt{H_1},H_1R_0\right\}c
        \right)
    \right]^{-1},
$
where $H_0=\cH_0$, $H_1=\cH_1$, and with batch size $B\geq 36(\rho-1)$ guarantees
\begin{equation*}
    \expect{\func{x^N}} - f^*
    \leq
    \max
    \left\{
        \left(
            1-\frac{\eta c}{36 R_0}
        \right)^{N/2}F_0,
        \frac{18R_0^2}{N\eta}
    \right\},
\end{equation*}
where $R_0=\norms{x^0-x^*}$ and $F_0=\func{x^0}-f^*$.
\end{theorem}
\end{boxA}

If the step size is chosen as
\begin{equation*}
    \eta
    =
    \left[
        9
        \left(
            H_0
            +
            \max\left\{\sqrt{H_1},H_1R_0\right\}c
        \right)
    \right]^{-1},
\end{equation*}
then the iteration complexity $N$ takes the following form:
\begin{equation*}
    \max
    \left\{
        \OboundTilde{
            \frac{
                \left[
                    H_0
                    +
                    \max\left\{\sqrt{H_1},H_1R_0\right\}c
                \right]R_0
            }{c}
        },
        \Obound{
            \frac{
                \left[
                    H_0
                    +
                    \max\left\{\sqrt{H_1},H_1R_0\right\}c
                \right]R_0^2
            }{\varepsilon}
        }
    \right\}.
\end{equation*}

\paragraph{Comparison to the prior work.}
The prior best-known convergence guarantee for ClipSGD under convex $(H_0,H_1)$-smoothness~\cite{Alimisis_2025} is stated in terms of a non-standard metric. In contrast, Theorem~\ref{th:Convex_ClipSGD_H_0H_1} provides a bound in the standard expected optimality gap $\expect{\func{x^N}}-f^*$.

\begin{proof}
Throughout the proof, we use the shorthand
\begin{equation*}
    F_k := \func{x^k}-f^*.
\end{equation*}
By Lemma~\ref{lem:stochastic_H0H1_implies_deterministic_H0H1}, Assumption~\ref{ass:H_0H_1_Smooth_for_xi} together with the interpolation condition implies deterministic $(H_0,H_1)$-smoothness of $f$ with $H_0=\cH_0$ and $H_1=\cH_1$.

We start by applying the descent inequality for $(H_0,H_1)$-smooth functions:
\begin{align}
    \func{x^{k+1}}-\func{x^k}
    &\overset{\circledOne}{\leq}
    \dotprod{\nabla \func{x^k}}{x^{k+1}-x^k}
    +
    \frac{
        H_0+H_1\left[\func{x^k}-f^*\right]
    }{2}
    \norms{x^{k+1}-x^k}^2
    \nonumber \\
    &\overset{\circledTwo}{=}
    -
    \eta
    \dotprod{
        \nabla \func{x^k}
    }{
        \clip{\nabla \func{x^k,\bxi^k}}
    }
    \nonumber \\
    &\quad\quad\quad
    +
    \frac{
        \eta^2
        \left(
            H_0+H_1F_k
        \right)
    }{2}
    \norms{
        \clip{\nabla \func{x^k,\bxi^k}}
    }^2
    \nonumber \\
    &\overset{\eqref{eq:relationship_between_function_gap_and_gradient_norm_H0H1}}{\leq}
    -
    \eta
    \dotprod{
        \nabla \func{x^k}
    }{
        \clip{\nabla \func{x^k,\bxi^k}}
    }
    \nonumber \\
    &\quad\quad\quad
    +
    \frac{
        \eta^2
        \left(
            H_0+H_1R_0\norms{\nabla \func{x^k}}
        \right)
    }{2}
    \norms{
        \clip{\nabla \func{x^k,\bxi^k}}
    }^2.
    \label{eq:Convex_clipSGD_smooth_H0_H1}
\end{align}
Here, in $\circledOne$ we used~\eqref{eq:H0H1_descent_inequality}; its locality condition
$\norms{x^{k+1}-x^k}\leq 1/\sqrt{H_1}$ is verified in Remark~\ref{rem:Step_size_convex_ClipSGD_H0_H1}. In $\circledTwo$ we used the ClipSGD update
$
    x^{k+1}
    =
    x^k
    -
    \eta \clip{\nabla \func{x^k,\bxi^k}}.$

We now consider two cases depending on the deterministic gradient norm:
\begin{equation*}
    \norms{\nabla \func{x^k}}\geq c
    \qquad
    \text{and}
    \qquad
    \norms{\nabla \func{x^k}}<c.
\end{equation*}

\fbox{1) The case $\norms{\nabla \func{x^k}}\geq c$.}

Define
\begin{equation*}
    \delta
    :=
    \mathbbm{1}_{\left\{
        \norms{
            \nabla \func{x^k,\bxi^k}
            -
            \nabla \func{x^k}
        }
        >
        3
        \sqrt{\frac{\rho-1}{B}}
        \norms{\nabla \func{x^k}}
    \right\}}.
\end{equation*}
Since
\begin{equation*}
    \clip{\nabla \func{x^k,\bxi^k}}
    =
    \alpha_{\bxi^k}
    \nabla \func{x^k,\bxi^k},
    \qquad
    \alpha_{\bxi^k}
    :=
    \min
    \left\{
        1,
        \frac{c}{\norms{\nabla \func{x^k,\bxi^k}}}
    \right\},
\end{equation*}
we have
\begin{align}
    &\expectcond{
        -
        \eta
        \dotprod{
            \nabla \func{x^k}
        }{
            \clip{\nabla \func{x^k,\bxi^k}}
        }
    }{x^k}\nonumber\\
    &=
    \mathbb{P}\left(\delta=0\mid x^k\right)
    \underbrace{
    \expectcond{
        -
        \eta
        \alpha_{\bxi^k}
        \dotprod{
            \nabla \func{x^k}
        }{
            \nabla \func{x^k,\bxi^k}
        }
    }{x^k,\delta=0}}_{=:T_1}
    \nonumber \\
    &\quad\quad\quad
    +
    \mathbb{P}\left(\delta=1\mid x^k\right)
    \underbrace{
    \expectcond{
        -
        \eta
        \alpha_{\bxi^k}
        \dotprod{
            \nabla \func{x^k}
        }{
            \nabla \func{x^k,\bxi^k}
        }
    }{x^k,\delta=1}}_{=:T_2}.
    \label{eq:Convex_clipSGD_T1_and_T2_H0_H1}
\end{align}
By Markov's inequality and~\eqref{eq:variance_with_batch_size},
\begin{align}
    \mathbb{P}\left(\delta=1\mid x^k\right)
    &=
    \mathbb{P}
    \left(
        \norms{
            \nabla \func{x^k,\bxi^k}
            -
            \nabla \func{x^k}
        }^2
        >
        9
        \frac{\rho-1}{B}
        \norms{\nabla \func{x^k}}^2
        \mid x^k
    \right)
    \nonumber \\
    &\leq
    \frac{
        \expectcond{
            \norms{
                \nabla \func{x^k,\bxi^k}
                -
                \nabla \func{x^k}
            }^2
        }{x^k}
    }{
        9
        \frac{\rho-1}{B}
        \norms{\nabla \func{x^k}}^2
    }\leq
    \frac{1}{9}.
    \label{eq:Convex_clipSGD_prob1_H0_H1}
\end{align}
Hence,
\begin{equation*}
    \mathbb{P}\left(\delta=0\mid x^k\right)
    \geq
    \frac{8}{9}.
\end{equation*}

We now estimate $T_1$. On the event $\delta=0$, we have
\begin{equation*}
    \norms{
        \nabla \func{x^k,\bxi^k}
        -
        \nabla \func{x^k}
    }
    \leq
    3
    \sqrt{\frac{\rho-1}{B}}
    \norms{\nabla \func{x^k}}.
\end{equation*}
Therefore,
\begin{align}
    T_1
    &=
    \expectcond{
        -
        \eta
        \alpha_{\bxi^k}
        \dotprod{
            \nabla \func{x^k}
        }{
            \nabla \func{x^k,\bxi^k}
        }
    }{x^k,\delta=0}
    \nonumber \\
    &=
    \expectcond{
        -
        \eta
        \alpha_{\bxi^k}
        \norms{\nabla \func{x^k}}^2
        -
        \eta
        \alpha_{\bxi^k}
        \dotprod{
            \nabla \func{x^k}
        }{
            \nabla \func{x^k,\bxi^k}
            -
            \nabla \func{x^k}
        }
    }{x^k,\delta=0}
    \nonumber \\
    &\leq
    \expectcond{
        -
        \eta
        \alpha_{\bxi^k}
        \norms{\nabla \func{x^k}}^2
        +
        \eta
        \alpha_{\bxi^k}
        \norms{\nabla \func{x^k}}
        \norms{
            \nabla \func{x^k,\bxi^k}
            -
            \nabla \func{x^k}
        }
    }{x^k,\delta=0}
    \nonumber \\
    &\leq
    \expectcond{
        -
        \eta
        \alpha_{\bxi^k}
        \norms{\nabla \func{x^k}}^2
        +
        3
        \eta
        \alpha_{\bxi^k}
        \sqrt{\frac{\rho-1}{B}}
        \norms{\nabla \func{x^k}}^2
    }{x^k,\delta=0}
    \nonumber \\
    &\overset{\circledOne}{\leq}
    \expectcond{
        -
        \frac{\eta\alpha_{\bxi^k}}{2}
        \norms{\nabla \func{x^k}}^2
    }{x^k,\delta=0}
    \nonumber \\
    &\overset{\circledTwo}{\leq}
    -
    \frac{\eta c}{4}
    \norms{\nabla \func{x^k}}.
    \label{eq:Convex_clipSGD_T1_H0_H1}
\end{align}
Here, in $\circledOne$ we used $B\geq 36(\rho-1)$. In $\circledTwo$, we used that on the event $\delta=0$,
\begin{align*}
    \norms{\nabla \func{x^k,\bxi^k}}
    &\leq
    \norms{
        \nabla \func{x^k,\bxi^k}
        -
        \nabla \func{x^k}
    }
    +
    \norms{\nabla \func{x^k}}
    \\
    &\leq
    \left(
        1
        +
        3
        \sqrt{\frac{\rho-1}{B}}
    \right)
    \norms{\nabla \func{x^k}}
    \\
    &\leq
    2\norms{\nabla \func{x^k}},
\end{align*}
and, since $\norms{\nabla \func{x^k}}\geq c$,
\begin{equation*}
    \alpha_{\bxi^k}
    \geq
    \min
    \left\{
        1,
        \frac{c}{2\norms{\nabla \func{x^k}}}
    \right\}
    =
    \frac{c}{2\norms{\nabla \func{x^k}}}.
\end{equation*}

Next, we estimate $T_2$:
\begin{align}
    T_2
    &=
    \expectcond{
        -
        \eta
        \alpha_{\bxi^k}
        \dotprod{
            \nabla \func{x^k}
        }{
            \nabla \func{x^k,\bxi^k}
        }
    }{x^k,\delta=1}
    \nonumber \\
    &\leq
    \eta
    \norms{\nabla \func{x^k}}
    \expectcond{
        \alpha_{\bxi^k}
        \norms{\nabla \func{x^k,\bxi^k}}
    }{x^k,\delta=1}
    \nonumber \\
    &\leq
    \eta c
    \norms{\nabla \func{x^k}}.
    \label{eq:Convex_clipSGD_T2_H0_H1}
\end{align}
Combining~\eqref{eq:Convex_clipSGD_T1_and_T2_H0_H1}, \eqref{eq:Convex_clipSGD_T1_H0_H1}, \eqref{eq:Convex_clipSGD_T2_H0_H1}, and~\eqref{eq:Convex_clipSGD_prob1_H0_H1}, we obtain
\begin{align}
    \expectcond{
        -
        \eta
        \dotprod{
            \nabla \func{x^k}
        }{
            \clip{\nabla \func{x^k,\bxi^k}}
        }
    }{x^k}
    &\leq
    -
    \frac{\eta c}{4}
    \norms{\nabla \func{x^k}}
    \mathbb{P}\left(\delta=0\mid x^k\right)
    \nonumber \\
    &\quad\quad\quad
    +
    \eta c
    \norms{\nabla \func{x^k}}
    \mathbb{P}\left(\delta=1\mid x^k\right)
    \nonumber \\
    &\leq
    -
    \eta c
    \left(
        \frac{8}{9}\cdot \frac{1}{4}
        -
        \frac{1}{9}
    \right)
    \norms{\nabla \func{x^k}}
    \nonumber \\
    &=
    -
    \frac{\eta c}{9}
    \norms{\nabla \func{x^k}}.
    \label{eq:Convex_clipSGD_first_term_smooth_H0_H1}
\end{align}

Taking conditional expectation in~\eqref{eq:Convex_clipSGD_smooth_H0_H1}, using~\eqref{eq:Convex_clipSGD_first_term_smooth_H0_H1}, and using $\norms{\clip{\nabla \func{x^k,\bxi^k}}}\leq c$, we get
\begin{align}
    \expectcond{F_{k+1}}{x^k}-F_k
    &\leq
    -
    \frac{\eta c}{9}
    \norms{\nabla \func{x^k}}
    +
    \frac{
        \eta^2
        \left(
            H_0+H_1R_0\norms{\nabla \func{x^k}}
        \right)
    }{2}
    c^2
    \nonumber \\
    &\overset{\circledOne}{\leq}
    -
    \frac{\eta c}{9}
    \norms{\nabla \func{x^k}}
    +
    \frac{
        \eta^2
        \left(
            H_0+H_1R_0c
        \right)
    }{2}
    c
    \norms{\nabla \func{x^k}}
    \nonumber \\
    &\overset{\circledTwo}{\leq}
    -
    \frac{\eta c}{9}
    \norms{\nabla \func{x^k}}
    +
    \frac{
        \eta^2
        \left(
            H_0+\max\left\{\sqrt{H_1},H_1R_0\right\}c
        \right)
    }{2}
    c
    \norms{\nabla \func{x^k}}
    \nonumber \\
    &\overset{\circledThree}{\leq}
    -
    \frac{\eta c}{18}
    \norms{\nabla \func{x^k}}
    \nonumber \\
    &\overset{\eqref{eq:relationship_between_function_gap_and_gradient_norm_H0H1}}{\leq}
    -
    \frac{\eta c}{18R_0}
    F_k.
    \label{eq:Convex_clipSGD_case1_descent_H0_H1}
\end{align}
Here, in $\circledOne$ we used $\norms{\nabla \func{x^k}}\geq c$, in $\circledTwo$ we used
$H_1R_0\leq \max\{\sqrt{H_1},H_1R_0\}$, and in $\circledThree$ we used the step-size condition.

Thus, in the first case we obtain
\begin{equation}
    \expectcond{F_{k+1}}{x^k}
    \leq
    \left(
        1-\frac{\eta c}{18R_0}
    \right)
    F_k.
    \label{eq:Convex_clipSGD_case1_ITOG_H0_H1}
\end{equation}

\fbox{2) The case $\norms{\nabla \func{x^k}}<c$.}

In this case, $\nabla \func{x^k}=\clip{\nabla \func{x^k}}$. From~\eqref{eq:Convex_clipSGD_smooth_H0_H1}, we have
\begin{align}
    \expectcond{F_{k+1}}{x^k}-F_k
    &\leq
    -
    \eta
    \expectcond{
        \dotprod{
            \nabla \func{x^k}
        }{
            \clip{\nabla \func{x^k,\bxi^k}}
        }
    }{x^k}
    \nonumber \\
    &\quad\quad\quad
    +
    \frac{
        \eta^2
        \left(
            H_0+H_1R_0\norms{\nabla \func{x^k}}
        \right)
    }{2}
    \expectcond{
        \norms{
            \clip{\nabla \func{x^k,\bxi^k}}
        }^2
    }{x^k}
    \nonumber \\
    &\overset{\circledOne}{\leq}
    -
    \eta
    \expectcond{
        \dotprod{
            \nabla \func{x^k}
        }{
            \clip{\nabla \func{x^k,\bxi^k}}
        }
    }{x^k}
    \nonumber \\
    &\quad\quad\quad
    +
    \frac{
        \eta^2
        \left(
            H_0+\max\left\{\sqrt{H_1},H_1R_0\right\}c
        \right)
    }{2}
    \expectcond{
        \norms{
            \clip{\nabla \func{x^k,\bxi^k}}
        }^2
    }{x^k}
    \nonumber \\
    &\overset{\circledTwo}{=}
    -
    \frac{\eta}{2}
    \norms{\nabla \func{x^k}}^2
    -
    \frac{\eta}{2}
    \expectcond{
        \norms{
            \clip{\nabla \func{x^k,\bxi^k}}
        }^2
    }{x^k}
    \nonumber \\
    &\quad\quad\quad
    +
    \frac{\eta}{2}
    \expectcond{
        \norms{
            \clip{\nabla \func{x^k,\bxi^k}}
            -
            \nabla \func{x^k}
        }^2
    }{x^k}
    \nonumber \\
    &\quad\quad\quad
    +
    \frac{
        \eta^2
        \left(
            H_0+\max\left\{\sqrt{H_1},H_1R_0\right\}c
        \right)
    }{2}
    \expectcond{
        \norms{
            \clip{\nabla \func{x^k,\bxi^k}}
        }^2
    }{x^k}
    \nonumber \\
    &=
    -
    \frac{\eta}{2}
    \norms{\nabla \func{x^k}}^2
    +
    \frac{\eta}{2}
    \expectcond{
        \norms{
            \clip{\nabla \func{x^k,\bxi^k}}
            -
            \clip{\nabla \func{x^k}}
        }^2
    }{x^k}
    \nonumber \\
    &\quad\quad\quad
    -
    \frac{\eta}{2}
    \left(
        1
        -
        \eta
        \left(
            H_0+\max\left\{\sqrt{H_1},H_1R_0\right\}c
        \right)
    \right)
    \expectcond{
        \norms{
            \clip{\nabla \func{x^k,\bxi^k}}
        }^2
    }{x^k}
    \nonumber \\
    &\overset{\circledThree}{\leq}
    -
    \frac{\eta}{2}
    \norms{\nabla \func{x^k}}^2
    +
    \frac{\eta}{2}
    \expectcond{
        \norms{
            \clip{\nabla \func{x^k,\bxi^k}}
            -
            \clip{\nabla \func{x^k}}
        }^2
    }{x^k}.
    \label{eq:Convex_clipSGD_case2_prevariance_H0_H1}
\end{align}
Here, in $\circledOne$ we used $\norms{\nabla \func{x^k}}<c$ and
$H_1R_0\leq \max\{\sqrt{H_1},H_1R_0\}$; in $\circledTwo$ we used the identity
\begin{equation*}
    -\eta\dotprod{a}{b}
    =
    -
    \frac{\eta}{2}\norms{a}^2
    -
    \frac{\eta}{2}\norms{b}^2
    +
    \frac{\eta}{2}\norms{a-b}^2
\end{equation*}
with $a=\nabla \func{x^k}$ and $b=\clip{\nabla \func{x^k,\bxi^k}}$; in $\circledThree$ we used
\begin{equation*}
    \eta
    \left(
        H_0+\max\left\{\sqrt{H_1},H_1R_0\right\}c
    \right)
    \leq
    \frac{1}{9}
    \leq
    1.
\end{equation*}

Since clipping is the Euclidean projection onto the ball of radius $c$, it is non-expansive. Hence,
\begin{align}
    \expectcond{F_{k+1}}{x^k}-F_k
    &\leq
    -
    \frac{\eta}{2}
    \norms{\nabla \func{x^k}}^2
    +
    \frac{\eta}{2}
    \expectcond{
        \norms{
            \nabla \func{x^k,\bxi^k}
            -
            \nabla \func{x^k}
        }^2
    }{x^k}
    \nonumber \\
    &\overset{\eqref{eq:variance_with_batch_size}}{\leq}
    -
    \frac{\eta}{2}
    \norms{\nabla \func{x^k}}^2
    +
    \frac{\eta(\rho-1)}{2B}
    \norms{\nabla \func{x^k}}^2
    \nonumber \\
    &\overset{\circledOne}{\leq}
    -
    \frac{\eta}{2}
    \norms{\nabla \func{x^k}}^2
    +
    \frac{\eta}{72}
    \norms{\nabla \func{x^k}}^2
    \nonumber \\
    &\leq
    -
    \frac{\eta}{9}
    \norms{\nabla \func{x^k}}^2
    \nonumber \\
    &\overset{\eqref{eq:relationship_between_function_gap_and_gradient_norm_H0H1}}{\leq}
    -
    \frac{\eta}{9R_0^2}
    F_k^2.
    \label{eq:Convex_clipSGD_case2_descent_H0_H1}
\end{align}
where in $\circledOne$ we used $B\geq 36(\rho-1)$.

Thus, in the second case we obtain
\begin{equation}
    \expectcond{F_{k+1}}{x^k}
    \leq
    F_k
    -
    \frac{\eta}{9R_0^2}
    F_k^2.
    \label{eq:Convex_clipSGD_case2_ITOG_H0_H1}
\end{equation}

\textbf{In summary.}
Let $\mathcal{F}_k$ be the sigma-algebra generated by the history up to
$x^k$, and define
\begin{equation*}
    G_k:=\norms{\nabla \func{x^k}},
    \qquad
    I_k:=\mathbbm{1}_{\{G_k\geq c\}}.
\end{equation*}
The random variable $I_k$ is $\mathcal{F}_k$-measurable. Set
\begin{equation*}
    a:=\frac{\eta c}{18R_0},
    \qquad
    b:=\frac{\eta}{9R_0^2}.
\end{equation*}
The two conditional estimates~\eqref{eq:Convex_clipSGD_case1_ITOG_H0_H1}
and~\eqref{eq:Convex_clipSGD_case2_ITOG_H0_H1} can be combined as
\begin{equation}
    \label{eq:combined_conditional_recursion_H0_H1}
    \expectcond{F_{k+1}}{\mathcal{F}_k}
    \leq
    F_k
    -
    aF_k I_k
    -
    bF_k^2(1-I_k).
\end{equation}
Indeed, on the event $I_k=1$ this is exactly the large-gradient estimate,
whereas on the event $I_k=0$ it is exactly the small-gradient estimate.

Taking expectations in~\eqref{eq:combined_conditional_recursion_H0_H1}, we obtain
\begin{equation}
    \label{eq:expected_combined_recursion_H0_H1}
    A_{k+1}
    \leq
    A_k
    -
    a\expect{F_k I_k}
    -
    b\expect{F_k^2(1-I_k)},
    \qquad
    A_k:=\expect{F_k}.
\end{equation}

We use the following elementary inequality: for any nonnegative random
variable $X$ and any indicator random variable $I$,
\begin{equation}
    \label{eq:elementary_indicator_inequality_H0_H1}
    a\expect{XI}
    +
    b\expect{X^2(1-I)}
    \geq
    \min\left\{
        \frac{a}{2}\expect{X},
        b\left(\expect{X}\right)^2
    \right\}.
\end{equation}
Indeed, let $A=\expect{X}$ and $r=\expect{X(1-I)}$. Since
\begin{equation*}
    \expect{X^2(1-I)}
    =
    \expect{\left(X(1-I)\right)^2}
    \geq
    \left(\expect{X(1-I)}\right)^2
    =
    r^2,
\end{equation*}
we have
\begin{equation*}
    a\expect{XI}
    +
    b\expect{X^2(1-I)}
    \geq
    a(A-r)+br^2.
\end{equation*}
Minimizing the right-hand side over $r\in[0,A]$ gives
\begin{equation*}
    a(A-r)+br^2
    \geq
    \min\left\{
        \frac{a}{2}A,
        bA^2
    \right\},
\end{equation*}
which proves~\eqref{eq:elementary_indicator_inequality_H0_H1}.

Applying~\eqref{eq:elementary_indicator_inequality_H0_H1} to $X=F_k$ and
$I=I_k$, inequality~\eqref{eq:expected_combined_recursion_H0_H1} yields
\begin{equation}
    \label{eq:deterministic_gap_recursion_H0_H1}
    A_{k+1}
    \leq
    A_k
    -
    \min\left\{
        \frac{a}{2}A_k,
        bA_k^2
    \right\}.
\end{equation}
In particular, $(A_k)_{k\geq0}$ is non-increasing.

Define
\begin{equation*}
    \alpha:=\frac{a}{2}=\frac{\eta c}{36R_0},
    \qquad
    \beta:=b=\frac{\eta}{9R_0^2}.
\end{equation*}
Then~\eqref{eq:deterministic_gap_recursion_H0_H1} becomes
\begin{equation}
    \label{eq:alpha_beta_gap_recursion_H0_H1}
    A_{k+1}
    \leq
    A_k-\min\{\alpha A_k,\beta A_k^2\}.
\end{equation}

Let
\begin{equation*}
    \tau
    :=
    \min\left\{
        k\in\{0,\ldots,N-1\}
        :
        A_k<\frac{\alpha}{\beta}
    \right\},
\end{equation*}
with the convention $\tau=N$ if the set is empty. Since $(A_k)$ is
non-increasing, the indices $k<\tau$ correspond to the linear part of
\eqref{eq:alpha_beta_gap_recursion_H0_H1}, and the indices $k\geq\tau$
correspond to the quadratic part.

If $\tau>N/2$, then for at least $N/2$ iterations,
\begin{equation*}
    A_{k+1}\leq (1-\alpha)A_k.
\end{equation*}
Using monotonicity for the remaining iterations, we get
\begin{equation}
    \label{eq:ClipSGD_final_case1_H0_H1}
    A_N
    \leq
    (1-\alpha)^{N/2}A_0
    =
    \left(
        1-\frac{\eta c}{36R_0}
    \right)^{N/2}F_0.
\end{equation}

If $\tau\leq N/2$, then for $k=\tau,\ldots,N-1$,
\begin{equation*}
    A_{k+1}\leq A_k-\beta A_k^2.
\end{equation*}
If $A_N=0$, the claim is trivial. Otherwise,
\begin{equation*}
    \frac{1}{A_{k+1}}-\frac{1}{A_k}
    =
    \frac{A_k-A_{k+1}}{A_kA_{k+1}}
    \geq
    \beta\frac{A_k}{A_{k+1}}
    \geq
    \beta.
\end{equation*}
Summing from $k=\tau$ to $N-1$ and using $N-\tau\geq N/2$, we obtain
\begin{equation}
    \label{eq:ClipSGD_final_case2_H0_H1}
    A_N
    \leq
    \frac{2}{\beta N}
    =
    \frac{18R_0^2}{N\eta}.
\end{equation}

Combining~\eqref{eq:ClipSGD_final_case1_H0_H1} and
\eqref{eq:ClipSGD_final_case2_H0_H1}, we conclude that
\begin{equation*}
    \expect{\func{x^N}}-f^*
    =
    A_N
    \leq
    \max\left\{
        \left(
            1-\frac{\eta c}{36R_0}
        \right)^{N/2}F_0,
        \frac{18R_0^2}{N\eta}
    \right\}.
\end{equation*}

\begin{remark}
\label{rem:Step_size_convex_ClipSGD_H0_H1}
In the first step of~\eqref{eq:Convex_clipSGD_smooth_H0_H1}, we use the descent lemma for $(H_0,H_1)$-smooth functions~\eqref{eq:H0H1_descent_inequality}. Its locality condition
$\norms{x^{k+1}-x^k}\leq 1/\sqrt{H_1}$ is guaranteed by the chosen step size. Indeed,
\begin{align*}
    \norms{x^{k+1}-x^k}
    &=
    \eta
    \norms{
        \clip{\nabla \func{x^k,\bxi^k}}
    }\leq
    \eta c\leq
    \frac{
        c
    }{
        9
        \left(
            H_0+\max\left\{\sqrt{H_1},H_1R_0\right\}c
        \right)
    }
    \\
    &\leq
    \frac{
        c
    }{
        9\sqrt{H_1}c
    }\leq
    \frac{1}{\sqrt{H_1}}.
\end{align*}
\end{remark}
\end{proof}

        We now present the convergence result for NSGD.

\begin{boxA}
\begin{theorem}[Convex, NSGD]\label{th:Convex_NSGD_H0_H1}
Let $f(\cdot,\xi)$ be $(\cH_0,\cH_1)$-smooth
(Assumption~\ref{ass:H_0H_1_Smooth_for_xi}) and convex
(Assumption~\ref{ass:Convexity_for_xi}). Suppose that the gradient oracle is unbiased
(Assumption~\ref{ass:unbiased}) and satisfies SGC
(Assumption~\ref{ass:strong growth condition} with $\sigma=0$). Then NSGD (Algorithm~\ref{algo:clip_algo}) with any $\lambda>0$, step size
$
    \eta
    \leq
    \frac{\lambda}{
        H_0
        +
        \max\left\{\sqrt{H_1},H_1R_0\right\}\lambda
    },
$
where $H_0=\cH_0$, $H_1=\cH_1$, and batch size $B\geq 64(\rho-1)$ guarantees
\begin{equation*}
    \expect{\func{x^N}} - f^*
    \leq
    \left(
        1-\frac{\eta}{4R_0}
    \right)^N F_0
    +
    6\lambda R_0,
\end{equation*}
where $R_0=\norms{x^0-x^*}$ and $F_0=\func{x^0}-f^*$.
\end{theorem}
\end{boxA}

Theorem~\ref{th:Convex_NSGD_H0_H1} establishes the result for the case $\sigma=0$. The case $\sigma>0$ is handled directly in the proof below. The theorem shows that NSGD converges to an error floor of $6\lambda R_0$. To ensure convergence to a desired accuracy $\varepsilon$, i.e.,
$\expect{\func{x^N}}-f^*\leq \varepsilon$, it is sufficient to choose
\begin{equation*}
    \lambda
    \leq
    \frac{\varepsilon}{12R_0}.
\end{equation*}
Setting $\lambda=\frac{\varepsilon}{12R_0}$ and taking the maximal admissible step size
\begin{equation*}
    \eta
    =
    \frac{\lambda}{
        H_0
        +
        \max\left\{\sqrt{H_1},H_1R_0\right\}\lambda
    }
    =
    \left[
        \frac{12H_0R_0}{\varepsilon}
        +
        \max\left\{\sqrt{H_1},H_1R_0\right\}
    \right]^{-1},
\end{equation*}
we obtain the following iteration complexity:
\begin{equation*}
    N
    =
    \Obound{
        \left(
            \frac{H_0R_0^2}{\varepsilon}
            +
            \max\left\{\sqrt{H_1},H_1R_0\right\}R_0
        \right)
        \log\frac{F_0}{\varepsilon}
    }.
\end{equation*}

\begin{proof}
For simplicity of presentation, let us introduce the notation
$
    \Gf{x^k,\bxi^k}
    \coloneqq
    \frac{
        \nabla \func{x^k,\bxi^k}
    }{
        \norms{\nabla \func{x^k,\bxi^k}}+\lambda
    }.
$
Then the NSGD update has the form
$
    x^{k+1}
    =
    x^k
    -
    \eta \Gf{x^k,\bxi^k}.
$
Throughout the proof, we use the shorthand
$
    F_k := \func{x^k}-f^*.
$ By Lemma~\ref{lem:stochastic_H0H1_implies_deterministic_H0H1}, Assumption~\ref{ass:H_0H_1_Smooth_for_xi} together with the interpolation condition implies deterministic $(H_0,H_1)$-smoothness of $f$ with $H_0=\cH_0$ and $H_1=\cH_1$.

Using the descent inequality for $(H_0,H_1)$-smooth functions, we obtain
\begin{align}
    \func{x^{k+1}}-\func{x^k}
    &\overset{\circledOne}{\leq}
    \dotprod{\nabla \func{x^k}}{x^{k+1}-x^k}
    +
    \frac{
        H_0+H_1\left[\func{x^k}-f^*\right]
    }{2}
    \norms{x^{k+1}-x^k}^2
    \nonumber \\
    &\overset{\circledTwo}{\leq}
    \dotprod{\nabla \func{x^k}}{x^{k+1}-x^k}
    \nonumber \\
    &\quad\quad\quad+
    \frac{
        H_0
        +
        \max\left\{\sqrt{H_1},H_1R_0\right\}
        \norms{\nabla \func{x^k}}
    }{2}
    \norms{x^{k+1}-x^k}^2
    \nonumber \\
    &\overset{\circledThree}{=}
    -
    \eta
    \dotprod{
        \nabla \func{x^k}
    }{
        \Gf{x^k,\bxi^k}
    }
    \nonumber \\
    &\quad\quad\quad+
    \frac{
        \eta^2
        \left(
            H_0
            +
            \max\left\{\sqrt{H_1},H_1R_0\right\}
            \norms{\nabla \func{x^k}}
        \right)
    }{2}
    \norms{\Gf{x^k,\bxi^k}}^2
    \nonumber \\
    &\overset{\circledFour}{\leq}
    -
    \eta
    \frac{
        \dotprod{
            \nabla \func{x^k}
        }{
            \nabla \func{x^k,\bxi^k}
        }
    }{
        \norms{\nabla \func{x^k,\bxi^k}}+\lambda
    }
    \nonumber \\
    &\quad\quad\quad+
    \frac{
        \eta^2
        \left(
            H_0
            +
            \max\left\{\sqrt{H_1},H_1R_0\right\}
            \norms{\nabla \func{x^k}}
        \right)
    }{2}
    \nonumber \\
    &=
    -
    \eta
    \frac{
        \dotprod{
            \nabla \func{x^k,\bxi^k}
            +
            \nabla \func{x^k}
            -
            \nabla \func{x^k,\bxi^k}
        }{
            \nabla \func{x^k,\bxi^k}
        }
    }{
        \norms{\nabla \func{x^k,\bxi^k}}+\lambda
    }
    \nonumber \\
    &\quad\quad\quad
    +
    \frac{
        \eta^2
        \left(
            H_0
            +
            \max\left\{\sqrt{H_1},H_1R_0\right\}
            \norms{\nabla \func{x^k}}
        \right)
    }{2}
    \nonumber \\
    &=
    -
    \eta
    \frac{
        \norms{\nabla \func{x^k,\bxi^k}}^2
    }{
        \norms{\nabla \func{x^k,\bxi^k}}+\lambda
    }
    -
    \eta
    \frac{
        \dotprod{
            \nabla \func{x^k}
            -
            \nabla \func{x^k,\bxi^k}
        }{
            \nabla \func{x^k,\bxi^k}
        }
    }{
        \norms{\nabla \func{x^k,\bxi^k}}+\lambda
    }
    \nonumber \\
    &\quad\quad\quad
    +
    \frac{
        \eta^2
        \left(
            H_0
            +
            \max\left\{\sqrt{H_1},H_1R_0\right\}
            \norms{\nabla \func{x^k}}
        \right)
    }{2}
    \nonumber \\
    &\overset{\circledFive}{\leq}
    -
    \eta
    \frac{
        \norms{\nabla \func{x^k,\bxi^k}}^2
    }{
        \norms{\nabla \func{x^k,\bxi^k}}+\lambda
    }
    +
    \eta
    \frac{
        \norms{
            \nabla \func{x^k}
            -
            \nabla \func{x^k,\bxi^k}
        }
        \norms{\nabla \func{x^k,\bxi^k}}
    }{
        \norms{\nabla \func{x^k,\bxi^k}}+\lambda
    }
    \nonumber \\
    &\quad\quad\quad
    +
    \frac{
        \eta^2
        \left(
            H_0
            +
            \max\left\{\sqrt{H_1},H_1R_0\right\}
            \norms{\nabla \func{x^k}}
        \right)
    }{2}
    \nonumber \\
    &\leq
    -
    \eta
    \frac{
        \norms{\nabla \func{x^k,\bxi^k}}^2
    }{
        \norms{\nabla \func{x^k,\bxi^k}}+\lambda
    }
    +
    \eta
    \norms{
        \nabla \func{x^k,\bxi^k}
        -
        \nabla \func{x^k}
    }
    \nonumber \\
    &\quad\quad\quad
    +
    \frac{
        \eta^2
        \left(
            H_0
            +
            \max\left\{\sqrt{H_1},H_1R_0\right\}
            \norms{\nabla \func{x^k}}
        \right)
    }{2}
    \nonumber \\
    &=
    -
    \eta
    \norms{\nabla \func{x^k,\bxi^k}}
    +
    \eta
    \frac{
        \lambda
        \norms{\nabla \func{x^k,\bxi^k}}
    }{
        \norms{\nabla \func{x^k,\bxi^k}}+\lambda
    }
    \nonumber \\
    &\quad\quad\quad
    +
    \eta
    \norms{
        \nabla \func{x^k,\bxi^k}
        -
        \nabla \func{x^k}
    }
    \nonumber \\
    &\quad\quad\quad+
    \frac{
        \eta^2
        \left(
            H_0
            +
            \max\left\{\sqrt{H_1},H_1R_0\right\}
            \norms{\nabla \func{x^k}}
        \right)
    }{2}
    \nonumber \\
    &\leq
    -
    \eta
    \norms{\nabla \func{x^k,\bxi^k}}
    +
    \eta\lambda
    +
    \eta
    \norms{
        \nabla \func{x^k,\bxi^k}
        -
        \nabla \func{x^k}
    }
    \nonumber \\
    &\quad\quad\quad
    +
    \frac{
        \eta^2
        \left(
            H_0
            +
            \max\left\{\sqrt{H_1},H_1R_0\right\}
            \norms{\nabla \func{x^k}}
        \right)
    }{2}
    \nonumber \\
    &\overset{\circledSix}{\leq}
    -
    \eta
    \norms{\nabla \func{x^k}}
    +
    2\eta
    \norms{
        \nabla \func{x^k,\bxi^k}
        -
        \nabla \func{x^k}
    }
    \nonumber \\
    &\quad\quad\quad
    +
    \frac{
        \eta^2
        \left(
            H_0
            +
            \max\left\{\sqrt{H_1},H_1R_0\right\}
            \norms{\nabla \func{x^k}}
        \right)
    }{2}
    +
    \eta\lambda
    \nonumber \\
    &\overset{\circledSeven}{\leq}
    -
    \frac{\eta}{2}
    \norms{\nabla \func{x^k}}
    +
    2\eta
    \norms{
        \nabla \func{x^k,\bxi^k}
        -
        \nabla \func{x^k}
    }
    +
    \frac{3}{2}\eta\lambda.
    \label{eq:Convex_NSGD_smooth_H0_H1}
\end{align}
Here, in $\circledOne$ we used~\eqref{eq:H0H1_descent_inequality}; its locality condition
$\norms{x^{k+1}-x^k}\leq 1/\sqrt{H_1}$ is verified in Remark~\ref{rem:Step_size_Convex_NSGD_H0_H1}. In $\circledTwo$ we used~\eqref{eq:relationship_between_function_gap_and_gradient_norm_H0H1} and
$H_1R_0\leq \max\left\{\sqrt{H_1},H_1R_0\right\}$. In $\circledThree$ we used the NSGD update. In $\circledFour$ we used
$\norms{\Gf{x^k,\bxi^k}}\leq 1$. In $\circledFive$ we used the Cauchy--Schwarz inequality. In $\circledSix$ we used the triangle inequality. Finally, in $\circledSeven$ we used the step-size condition and the fact that
\begin{equation*}
    \eta
    \leq
    \frac{\lambda}{
        H_0+
        \max\left\{\sqrt{H_1},H_1R_0\right\}\lambda
    }
    \leq
    \frac{
        \norms{\nabla \func{x^k}}+\lambda
    }{
        H_0+
        \max\left\{\sqrt{H_1},H_1R_0\right\}
        \norms{\nabla \func{x^k}}
    }.
\end{equation*}
Indeed,
\begin{align*}
    &
    \frac{
        \norms{\nabla \func{x^k}}+\lambda
    }{
        H_0+
        \max\left\{\sqrt{H_1},H_1R_0\right\}
        \norms{\nabla \func{x^k}}
    }
    -
    \frac{\lambda}{
        H_0+
        \max\left\{\sqrt{H_1},H_1R_0\right\}\lambda
    }
    \\
    &\quad =
    \frac{
        H_0\norms{\nabla \func{x^k}}
        +
        \max\left\{\sqrt{H_1},H_1R_0\right\}\lambda^2
    }{
        \left(
            H_0+
            \max\left\{\sqrt{H_1},H_1R_0\right\}
            \norms{\nabla \func{x^k}}
        \right)
        \left(
            H_0+
            \max\left\{\sqrt{H_1},H_1R_0\right\}\lambda
        \right)
    }
    \geq 0.
\end{align*}

Taking conditional expectation in~\eqref{eq:Convex_NSGD_smooth_H0_H1}, we get
\begin{align}
    \expectcond{F_{k+1}}{x^k}-F_k
    &\leq
    -
    \frac{\eta}{2}
    \norms{\nabla \func{x^k}}
    +
    2\eta
    \expectcond{
        \norms{
            \nabla \func{x^k,\bxi^k}
            -
            \nabla \func{x^k}
        }
    }{x^k}
    +
    \frac{3}{2}\eta\lambda
    \nonumber \\
    &\leq
    -
    \frac{\eta}{2}
    \norms{\nabla \func{x^k}}
    +
    2\eta
    \left(
        \expectcond{
            \norms{
                \nabla \func{x^k,\bxi^k}
                -
                \nabla \func{x^k}
            }^2
        }{x^k}
    \right)^{1/2}
    +
    \frac{3}{2}\eta\lambda
    \nonumber \\
    &\overset{\eqref{eq:variance_with_batch_size}}{\leq}
    -
    \frac{\eta}{2}
    \norms{\nabla \func{x^k}}
    +
    2\eta
    \left(
        \frac{(\rho-1)\norms{\nabla \func{x^k}}^2+\sigma^2}{B}
    \right)^{1/2}
    +
    \frac{3}{2}\eta\lambda
    \nonumber \\
    &\leq
    -
    \frac{\eta}{2}
    \norms{\nabla \func{x^k}}
    +
    2\eta
    \sqrt{\frac{\rho-1}{B}}
    \norms{\nabla \func{x^k}}
    +
    \frac{2\eta\sigma}{\sqrt{B}}
    +
    \frac{3}{2}\eta\lambda
    \nonumber \\
    &\overset{\circledOne}{\leq}
    -
    \frac{\eta}{4}
    \norms{\nabla \func{x^k}}
    +
    \frac{2\eta\sigma}{\sqrt{B}}
    +
    \frac{3}{2}\eta\lambda,
    \label{eq:Convex_NSGD_SeMI_Itog_H0_H1}
\end{align}
where in $\circledOne$ we used $B\geq 64(\rho-1)$.

Thus,
\begin{align}
    \expectcond{F_{k+1}}{x^k}
    &\leq
    F_k
    -
    \frac{\eta}{4}
    \norms{\nabla \func{x^k}}
    +
    \frac{2\eta\sigma}{\sqrt{B}}
    +
    \frac{3}{2}\eta\lambda
    \nonumber \\
    &\overset{\eqref{eq:relationship_between_function_gap_and_gradient_norm_H0H1}}{\leq}
    \left(
        1-\frac{\eta}{4R_0}
    \right)
    F_k
    +
    \frac{2\eta\sigma}{\sqrt{B}}
    +
    \frac{3}{2}\eta\lambda.
    \label{eq:Convex_NSGD_ITOG_H0_H1}
\end{align}

Taking full expectation and unrolling the recursion, we obtain
\begin{align*}
    \expect{F_N}
    &\leq
    \left(
        1-\frac{\eta}{4R_0}
    \right)^N
    F_0
    +
    \left(
        \frac{2\eta\sigma}{\sqrt{B}}
        +
        \frac{3}{2}\eta\lambda
    \right)
    \sum_{i=0}^{N-1}
    \left(
        1-\frac{\eta}{4R_0}
    \right)^i
    \\
    &\leq
    \left(
        1-\frac{\eta}{4R_0}
    \right)^N
    F_0
    +
    \left(
        \frac{2\eta\sigma}{\sqrt{B}}
        +
        \frac{3}{2}\eta\lambda
    \right)
    \frac{4R_0}{\eta}
    \\
    &=
    \left(
        1-\frac{\eta}{4R_0}
    \right)^N
    F_0
    +
    \frac{8\sigma R_0}{\sqrt{B}}
    +
    6\lambda R_0.
\end{align*}
For $\sigma=0$, this gives the claimed result:
\begin{equation*}
    \expect{\func{x^N}}-f^*
    \leq
    \left(
        1-\frac{\eta}{4R_0}
    \right)^N
    F_0
    +
    6\lambda R_0.
\end{equation*}

\begin{remark}
\label{rem:Step_size_Convex_NSGD_H0_H1}
In the first step of~\eqref{eq:Convex_NSGD_smooth_H0_H1}, we use the descent lemma for $(H_0,H_1)$-smooth functions~\eqref{eq:H0H1_descent_inequality}. Its locality condition
$\norms{x^{k+1}-x^k}\leq 1/\sqrt{H_1}$ is guaranteed by the chosen step size. Indeed,
\begin{align*}
    \norms{x^{k+1}-x^k}
    &=
    \eta
    \norms{\Gf{x^k,\bxi^k}}\leq
    \eta\leq
    \frac{\lambda}{
        H_0
        +
        \max\left\{\sqrt{H_1},H_1R_0\right\}\lambda
    }\leq
    \frac{\lambda}{
        \sqrt{H_1}\lambda
    }
    =
    \frac{1}{\sqrt{H_1}}.
\end{align*}
\end{remark}
\end{proof}
        \subsubsection{Extension to heavy-tailed noise}
\label{app:Convex_NSGD_H0_H1_heavy_tailed}

We now extend Theorem~\ref{th:Convex_NSGD_H0_H1} to the heavy-tailed setting. The proof follows the same argument as in Theorem~\ref{th:Convex_NSGD_H0_H1}; the only modification is the way we control the stochastic error term
\begin{equation*}
    \expectcond{
        \norms{
            \nabla \func{x^k,\bxi^k}
            -
            \nabla \func{x^k}
        }
    }{x^k}.
\end{equation*}
Instead of the second-moment bound~\eqref{eq:variance_with_batch_size}, we use the heavy-tailed mini-batch estimate from Appendix~\ref{app:NSGD_heavy_tailed_noise}. Namely, under Assumption~\ref{prop:SGC_under_heavy tailed noise}, for $p\in(1,2)$ and mini-batch size
\begin{equation*}
    B
    \geq
    2^{\frac{3p+1}{p-1}}
    (\rho-1)^{\frac{1}{p-1}},
\end{equation*}
we have
\begin{equation}
    \label{eq:HT_noise_bound_H0_H1_NSGD}
    \expectcond{
        \norms{
            \nabla \func{x^k,\bxi^k}
            -
            \nabla \func{x^k}
        }
    }{x^k}
    \leq
    \frac{1}{8}\norms{\nabla \func{x^k}}
    +
    \frac{2^{1/p}\sigma}{B^{\frac{p-1}{p}}}.
\end{equation}

\begin{boxD}
\begin{proposition}[Heavy-tailed, convex NSGD]
\label{cor:Convex_NSGD_HT_H0_H1}
Let $f(\cdot,\xi)$ be $(\cH_0,\cH_1)$-smooth
(Assumption~\ref{ass:H_0H_1_Smooth_for_xi}) and convex
(Assumption~\ref{ass:Convexity_for_xi}). Suppose that the gradient oracle is unbiased
(Assumption~\ref{ass:unbiased}) and satisfies the generalized $p$-th central moment condition
(Assumption~\ref{prop:SGC_under_heavy tailed noise}). Then NSGD (Algorithm~\ref{algo:clip_algo}) with any $\lambda>0$, step size
$
    \eta
    \leq
    \frac{\lambda}{
        H_0
        +
        \max\left\{\sqrt{H_1},H_1R_0\right\}\lambda
    },
$
where $H_0=\cH_0$, $H_1=\cH_1$, and batch size
$
    B
    \geq
    2^{\frac{3p+1}{p-1}}
    (\rho-1)^{\frac{1}{p-1}},
$
guarantees
\begin{equation*}
    \expect{\func{x^N}} - f^*
    \leq
    \left(
        1-\frac{\eta}{4R_0}
    \right)^N F_0
    +
    \frac{2^{3+\frac{1}{p}}\sigma R_0}{B^{\frac{p-1}{p}}}
    +
    6\lambda R_0,
\end{equation*}
where $R_0=\norms{x^0-x^*}$ and $F_0=\func{x^0}-f^*$.
In particular, when $\sigma=0$, this reduces to
\begin{equation*}
    \expect{\func{x^N}} - f^*
    \leq
    \left(
        1-\frac{\eta}{4R_0}
    \right)^N F_0
    +
    6\lambda R_0.
\end{equation*}
\end{proposition}
\end{boxD}

\begin{proof}
The proof is identical to the proof of Theorem~\ref{th:Convex_NSGD_H0_H1} up to the estimate of the stochastic error term. Starting from~\eqref{eq:Convex_NSGD_smooth_H0_H1}, taking conditional expectation, and using~\eqref{eq:HT_noise_bound_H0_H1_NSGD}, we obtain
\begin{align}
    \expectcond{F_{k+1}}{x^k}-F_k
    &\leq
    -
    \frac{\eta}{2}
    \norms{\nabla \func{x^k}}
    +
    2\eta
    \expectcond{
        \norms{
            \nabla \func{x^k,\bxi^k}
            -
            \nabla \func{x^k}
        }
    }{x^k}
    +
    \frac{3}{2}\eta\lambda
    \nonumber \\
    &\leq
    -
    \frac{\eta}{2}
    \norms{\nabla \func{x^k}}
    +
    \frac{\eta}{4}
    \norms{\nabla \func{x^k}}
    +
    \frac{2^{1+\frac{1}{p}}\eta\sigma}{B^{\frac{p-1}{p}}}
    +
    \frac{3}{2}\eta\lambda
    \nonumber \\
    &=
    -
    \frac{\eta}{4}
    \norms{\nabla \func{x^k}}
    +
    \frac{2^{1+\frac{1}{p}}\eta\sigma}{B^{\frac{p-1}{p}}}
    +
    \frac{3}{2}\eta\lambda.
    \label{eq:Convex_NSGD_HT_SeMI_Itog_H0_H1}
\end{align}
Using the convexity implication~\eqref{eq:relationship_between_function_gap_and_gradient_norm_H0H1}, we further get
\begin{align}
    \expectcond{F_{k+1}}{x^k}
    &\leq
    \left(
        1-\frac{\eta}{4R_0}
    \right)
    F_k
    +
    \frac{2^{1+\frac{1}{p}}\eta\sigma}{B^{\frac{p-1}{p}}}
    +
    \frac{3}{2}\eta\lambda.
    \label{eq:Convex_NSGD_HT_recursion_H0_H1}
\end{align}
Taking full expectation and unrolling the recursion yields
\begin{align*}
    \expect{F_N}
    &\leq
    \left(
        1-\frac{\eta}{4R_0}
    \right)^N
    F_0
    +
    \left(
        \frac{2^{1+\frac{1}{p}}\eta\sigma}{B^{\frac{p-1}{p}}}
        +
        \frac{3}{2}\eta\lambda
    \right)
    \sum_{i=0}^{N-1}
    \left(
        1-\frac{\eta}{4R_0}
    \right)^i
    \\
    &\leq
    \left(
        1-\frac{\eta}{4R_0}
    \right)^N
    F_0
    +
    \left(
        \frac{2^{1+\frac{1}{p}}\eta\sigma}{B^{\frac{p-1}{p}}}
        +
        \frac{3}{2}\eta\lambda
    \right)
    \frac{4R_0}{\eta}
    \\
    &=
    \left(
        1-\frac{\eta}{4R_0}
    \right)^N
    F_0
    +
    \frac{2^{3+\frac{1}{p}}\sigma R_0}{B^{\frac{p-1}{p}}}
    +
    6\lambda R_0.
\end{align*}
This completes the proof.
\end{proof}

\subsection{Deterministic setting}

Using the $(H_0,H_1)$-smoothness assumption allows us to improve existing
convergence bounds of gradient descent with the warm-up stepsize strategy.
Throughout this subsection, we assume $H_0,H_1>0$; the case $H_1=0$ reduces
to the standard $H_0$-smooth setting.

\begin{algorithm}[H]
       \caption{Gradient Descent Method (GD)}
       \label{algo:GD}
    \begin{algorithmic}
       \STATE {\bfseries Input:} $x_0 \in \mathbb{R}^d$, iterations number $N$, $H_0,H_1\geq 0$, universal constant $\theta\in(0,3/4]$
       \FOR{$k=0$ {\bfseries to} $N-1$}
       \STATE $x^{k+1} \gets x^k-\theta \min\left\{\frac{1}{H_0},\frac{1}{3H_1 (\func{x^k} - f^*)}\right\}\nabla f(x^k)$
       \ENDFOR
       \STATE {\bfseries Return:} $x^{N}$
    \end{algorithmic}
\end{algorithm}
\vspace{-1em}

Here and below, if $F_k:= \func{x^k} - f^*) =0$, we use the convention
$\frac{1}{3H_1F_k}=+\infty$, so that $\eta_k=\theta/H_0$. The step size of
Algorithm~\ref{algo:GD} corresponds to a warm-up strategy: as the gap $F_k$
decreases, the admissible step size increases until it reaches the constant
regime $\theta/H_0$.

In what follows, we consider non-convex, convex
(Assumption~\ref{ass:Convexity}), and strongly convex objectives.

\begin{boxF}
\begin{assumption}[Strong-convexity]\label{ass:Strong_Convexity}
    There exists a constant $\mu > 0$ such that function $f$ satisfies for all
    $x,y \in \mathbb{R}^d$:
    \begin{equation}\label{eq:Strong_Convexity}
        \func{x} - \func{y} + \frac{\mu}{2}\norms{x-y}^2
        \leq
        \dotprod{\nabla \func{x}}{x-y}.
    \end{equation}
\end{assumption}
\end{boxF}

We first record two auxiliary facts that will be used in all deterministic
proofs.

\begin{boxN}
\begin{lemma}[Locality of the GD step]
\label{lem:GD_locality_H0_H1}
Let $f$ be $(H_0,H_1)$-smooth
(Assumption~\ref{ass:H_0H_1_Smooth}) and let
$F_k:=\func{x^k}-f^*\geq 0$. Then GD from Algorithm~\ref{algo:GD} satisfies
\begin{equation*}
    \norms{x^{k+1}-x^k}
    =
    \eta_k\norms{\nabla \func{x^k}}
    \leq
    \frac{1}{\sqrt{H_1}}.
\end{equation*}
Consequently, the local $(H_0,H_1)$-smoothness inequality can be applied to
the pair $(x^k,x^{k+1})$.
\end{lemma}
\end{boxN}

\begin{proof}
Let $G_k:=\norms{\nabla \func{x^k}}$. If $G_k=0$, the claim is immediate. Otherwise, set
\begin{equation*}
    y^k
    :=
    x^k
    -
    \frac{1}{\sqrt{H_1}}
    \frac{\nabla \func{x^k}}{G_k}.
\end{equation*}
Since $\norms{y^k-x^k}=1/\sqrt{H_1}$, Assumption~\ref{ass:H_0H_1_Smooth}
gives
\begin{align*}
    \func{y^k}
    &\leq
    \func{x^k}
    +
    \dotprod{\nabla \func{x^k}}{y^k-x^k}
    +
    \frac{H_0+H_1F_k}{2}\norms{y^k-x^k}^2 \\
    &=
    \func{x^k}
    -
    \frac{G_k}{\sqrt{H_1}}
    +
    \frac{H_0+H_1F_k}{2H_1}.
\end{align*}
Using $\func{y^k}\geq f^*$, we obtain
\begin{equation*}
    G_k
    \leq
    \frac{H_0}{2\sqrt{H_1}}
    +
    \frac{3\sqrt{H_1}}{2}F_k.
\end{equation*}
Therefore,
\begin{align*}
    \eta_kG_k
    &\leq
    \theta
    \min\left\{
        \frac{1}{H_0},
        \frac{1}{3H_1F_k}
    \right\}
    \left(
        \frac{H_0}{2\sqrt{H_1}}
        +
        \frac{3\sqrt{H_1}}{2}F_k
    \right) \\
    &\leq
    \theta
    \left(
        \frac{1}{2\sqrt{H_1}}
        +
        \frac{1}{2\sqrt{H_1}}
    \right)
    =
    \frac{\theta}{\sqrt{H_1}}
    \leq
    \frac{1}{\sqrt{H_1}}.
\end{align*}
\end{proof}

\begin{boxN}
\begin{lemma}[Basic descent]
\label{lem:GD_basic_descent_H0_H1}
Let $f$ be $(H_0,H_1)$-smooth
(Assumption~\ref{ass:H_0H_1_Smooth}). Then GD from
Algorithm~\ref{algo:GD} satisfies
\begin{equation}
    \label{eq:GD_H0H1_basic_descent}
    F_{k+1}
    \leq
    F_k
    -
    \frac{\eta_k}{2}
    \norms{\nabla \func{x^k}}^2.
\end{equation}
In particular, $(F_k)_{k\geq0}$ is non-increasing.
\end{lemma}
\end{boxN}

\begin{proof}
By Lemma~\ref{lem:GD_locality_H0_H1}, the update remains in the local
smoothness region. Hence
\begin{align*}
    F_{k+1}-F_k
    &=
    \func{x^{k+1}}-\func{x^k} \\
    &\leq
    -\eta_k\norms{\nabla \func{x^k}}^2
    +
    \frac{\eta_k^2}{2}
    \left(H_0+H_1F_k\right)
    \norms{\nabla \func{x^k}}^2.
\end{align*}
Moreover,
\begin{equation*}
    \eta_k(H_0+H_1F_k)
    \leq
    \theta
    \left(1+\frac{1}{3}\right)
    =
    \frac{4\theta}{3}
    \leq
    1.
\end{equation*}
Thus,
\[
    F_{k+1}-F_k
    \leq
    -\frac{\eta_k}{2}
    \norms{\nabla \func{x^k}}^2.
\]
\end{proof}

\subsubsection{Non-convex setting}

\begin{boxA}
\begin{theorem}\label{th:GD_NC}
Let $f$ be $(H_0,H_1)$-smooth (Assumption~\ref{ass:H_0H_1_Smooth}). Then GD
(Algorithm~\ref{algo:GD}) guarantees
\begin{equation*}
    \min_{0\leq k\leq N-1}\norms{\nabla \func{x^k}}
    \leq
    \begin{cases}
        \sqrt{\frac{2H_0F_0}{\theta N}},
        & M=0,\\[0.5em]
        \min\left\{
        \sqrt{\frac{6H_1F_0\sum_{k=0}^{M-1}F_k}{\theta M^2}},
        \sqrt{\frac{2H_0F_0}{\theta(N-M)}}
        \right\},
        & 1\leq M\leq N-1,\\[0.5em]
        \sqrt{\frac{6H_1F_0\sum_{k=0}^{N-1}F_k}{\theta N^2}},
        & M=N,
    \end{cases}
\end{equation*}
where $F_k=\func{x^k}-f^*$, $\theta=\frac{1}{4\sqrt{2}+4}$, and $M$ is the
smallest index $k\in\{0,\ldots,N-1\}$ such that
$F_k<\frac{H_0}{3H_1}$, with the convention $M=N$ if this set is empty.
\end{theorem}
\end{boxA}

\begin{proof}
From Lemma~\ref{lem:GD_basic_descent_H0_H1}, we have
\begin{equation}
    \label{eq:GD_NC_basic_descent}
    F_k-F_{k+1}
    \geq
    \frac{\eta_k}{2}
    \norms{\nabla \func{x^k}}^2.
\end{equation}

If $F_k\geq H_0/(3H_1)$, then
\[
    \eta_k=\frac{\theta}{3H_1F_k}.
\]
Therefore,
\begin{equation}
    \label{eq:GD_NC_large_gap}
    \frac{\theta}{6H_1F_k}
    \norms{\nabla \func{x^k}}^2
    \leq
    F_k-F_{k+1}.
\end{equation}

If $F_k< H_0/(3H_1)$, then
\[
    \eta_k=\frac{\theta}{H_0}.
\]
Therefore,
\begin{equation}
    \label{eq:GD_NC_small_gap}
    \norms{\nabla \func{x^k}}^2
    \leq
    \frac{2H_0}{\theta}
    \left(F_k-F_{k+1}\right).
\end{equation}

By Lemma~\ref{lem:GD_basic_descent_H0_H1}, the sequence $(F_k)$ is
non-increasing. Hence, by the definition of $M$, the indices
$k=0,\ldots,M-1$ correspond to the large-gap phase, and the indices
$k=M,\ldots,N-1$ correspond to the small-gap phase.

Assume first that $1\leq M\leq N-1$. Summing~\eqref{eq:GD_NC_large_gap} for
$k=0,\ldots,M-1$, we get
\begin{align*}
    \sum_{k=0}^{M-1}
    \frac{\theta}{6H_1F_k}
    \norms{\nabla \func{x^k}}^2
    &\leq
    \sum_{k=0}^{M-1}
    \left(F_k-F_{k+1}\right)
    \leq
    F_0.
\end{align*}
Using
\[
    \sum_{k=0}^{M-1}\frac{\theta}{6H_1F_k}
    \geq
    \frac{\theta M^2}{6H_1\sum_{k=0}^{M-1}F_k},
\]
we obtain
\begin{equation}
    \label{eq:GD_NC_large_gap_rate}
    \min_{0\leq k\leq M-1}
    \norms{\nabla \func{x^k}}
    \leq
    \sqrt{
        \frac{6H_1F_0\sum_{k=0}^{M-1}F_k}{\theta M^2}
    }.
\end{equation}

Summing~\eqref{eq:GD_NC_small_gap} for $k=M,\ldots,N-1$, we get
\begin{align*}
    \sum_{k=M}^{N-1}
    \norms{\nabla \func{x^k}}^2
    &\leq
    \frac{2H_0}{\theta}
    \sum_{k=M}^{N-1}
    \left(F_k-F_{k+1}\right)
    \leq
    \frac{2H_0F_0}{\theta}.
\end{align*}
Thus,
\begin{equation}
    \label{eq:GD_NC_small_gap_rate}
    \min_{M\leq k\leq N-1}
    \norms{\nabla \func{x^k}}
    \leq
    \sqrt{
        \frac{2H_0F_0}{\theta(N-M)}
    }.
\end{equation}
Combining~\eqref{eq:GD_NC_large_gap_rate} and
\eqref{eq:GD_NC_small_gap_rate} gives the interior case.

If $M=0$, then all indices are in the small-gap phase, and summing
\eqref{eq:GD_NC_small_gap} for $k=0,\ldots,N-1$ gives
\[
    \min_{0\leq k\leq N-1}
    \norms{\nabla \func{x^k}}
    \leq
    \sqrt{\frac{2H_0F_0}{\theta N}}.
\]
If $M=N$, then all indices are in the large-gap phase, and summing
\eqref{eq:GD_NC_large_gap} for $k=0,\ldots,N-1$ gives
\[
    \min_{0\leq k\leq N-1}
    \norms{\nabla \func{x^k}}
    \leq
    \sqrt{
        \frac{6H_1F_0\sum_{k=0}^{N-1}F_k}{\theta N^2}
    }.
\]
\end{proof}

\subsubsection{Convex setting}

\begin{boxN}
\begin{lemma}[Monotonicity of the distance to the optimum]
\label{lem:argument_monotonicity}
Let $f$ be $(H_0,H_1)$-smooth (Assumption~\ref{ass:H_0H_1_Smooth}) and
convex (Assumption~\ref{ass:Convexity}). Then GD from Algorithm~\ref{algo:GD}
satisfies
\begin{equation}
    \label{eq:argument_monotonicity}
    \norms{x^k-x^*}
    \leq
    \norms{x^0-x^*}
    =
    R
    \quad
    \text{for all } k\geq0.
\end{equation}
\end{lemma}
\end{boxN}

\begin{proof}
By Lemma~\ref{lem:GD_basic_descent_H0_H1}, we have
\[
    F_{k+1}
    \leq
    F_k
    -
    \frac{\eta_k}{2}
    \norms{\nabla \func{x^k}}^2.
\]
Since $F_{k+1}\geq0$, this implies
\begin{equation}
    \label{eq:GD_gradient_self_bound_from_descent}
    \eta_k\norms{\nabla \func{x^k}}^2
    \leq
    2F_k.
\end{equation}
Using convexity, we obtain
\begin{align*}
    \norms{x^{k+1}-x^*}^2
    &=
    \norms{x^k-x^*-\eta_k\nabla \func{x^k}}^2 \\
    &=
    \norms{x^k-x^*}^2
    -
    2\eta_k
    \dotprod{\nabla \func{x^k}}{x^k-x^*}
    +
    \eta_k^2\norms{\nabla \func{x^k}}^2 \\
    &\leq
    \norms{x^k-x^*}^2
    -
    2\eta_k F_k
    +
    \eta_k^2\norms{\nabla \func{x^k}}^2 \\
    &\overset{\eqref{eq:GD_gradient_self_bound_from_descent}}{\leq}
    \norms{x^k-x^*}^2.
\end{align*}
The result follows by induction.
\end{proof}

\begin{boxA}
\begin{theorem}\label{th:GD_Convex}
Let $f$ be $(H_0,H_1)$-smooth (Assumption~\ref{ass:H_0H_1_Smooth}) and convex
(Assumption~\ref{ass:Convexity}). Then GD (Algorithm~\ref{algo:GD}) guarantees
the following bound. Let
\[
    T
    :=
    \min\left\{
        k\in\{0,\ldots,N-1\}
        :
        F_k<\frac{H_0}{3H_1}
    \right\},
\]
with the convention $T=N$ if this set is empty. If $T<N$, then
\begin{equation*}
    \func{x^N}-f^*
    \leq
    \min\left\{
        \left(1-\frac{\theta}{6H_1R^2}\right)^T F_0,
        \frac{2H_0R^2}{\theta(N-T)}
    \right\},
\end{equation*}
where $R=\norms{x^0-x^*}$ and $\theta=\frac{1}{4\sqrt{2}+4}$. If $T=N$, then
\begin{equation*}
    \func{x^N}-f^*
    \leq
    \left(1-\frac{\theta}{6H_1R^2}\right)^N F_0.
\end{equation*}
\end{theorem}
\end{boxA}

\begin{proof}
By convexity and Lemma~\ref{lem:argument_monotonicity},
\begin{equation}
    \label{eq:GD_convex_gap_gradient_relation}
    F_k
    \leq
    \dotprod{\nabla \func{x^k}}{x^k-x^*}
    \leq
    R\norms{\nabla \func{x^k}}.
\end{equation}
Combining~\eqref{eq:GD_convex_gap_gradient_relation} with
Lemma~\ref{lem:GD_basic_descent_H0_H1}, we get
\begin{align}
    F_{k+1}
    &\leq
    F_k
    -
    \frac{\eta_k}{2}
    \norms{\nabla \func{x^k}}^2
    \nonumber\\
    &\leq
    F_k
    -
    \frac{\theta}{2R^2}
    \frac{F_k^2}{\max\{H_0,3H_1F_k\}}.
    \label{eq:GD_convex_master_recursion}
\end{align}

If $F_k\geq H_0/(3H_1)$, then
\[
    \max\{H_0,3H_1F_k\}=3H_1F_k,
\]
and~\eqref{eq:GD_convex_master_recursion} yields
\begin{equation}
    \label{eq:GD_convex_linear_phase}
    F_{k+1}
    \leq
    \left(
        1-\frac{\theta}{6H_1R^2}
    \right)F_k.
\end{equation}
If $F_k< H_0/(3H_1)$, then
\[
    \max\{H_0,3H_1F_k\}=H_0,
\]
and~\eqref{eq:GD_convex_master_recursion} yields
\begin{equation}
    \label{eq:GD_convex_sublinear_phase}
    F_{k+1}
    \leq
    F_k
    -
    \frac{\theta}{2H_0R^2}F_k^2.
\end{equation}

By Lemma~\ref{lem:GD_basic_descent_H0_H1}, $(F_k)$ is non-increasing.
Therefore, by the definition of $T$, the recursion
\eqref{eq:GD_convex_linear_phase} applies for $k=0,\ldots,T-1$, while
\eqref{eq:GD_convex_sublinear_phase} applies for $k=T,\ldots,N-1$ if $T<N$.

Unrolling~\eqref{eq:GD_convex_linear_phase}, we get
\begin{equation}
    \label{eq:GD_convex_rate_1}
    F_T
    \leq
    \left(
        1-\frac{\theta}{6H_1R^2}
    \right)^T
    F_0.
\end{equation}
Since $(F_k)$ is non-increasing, $F_N\leq F_T$, and hence
\begin{equation}
    \label{eq:GD_convex_rate_1_final}
    F_N
    \leq
    \left(
        1-\frac{\theta}{6H_1R^2}
    \right)^T
    F_0.
\end{equation}

If $T<N$, then from~\eqref{eq:GD_convex_sublinear_phase},
\[
    F_{k+1}
    \leq
    F_k
    -
    \frac{\theta}{2H_0R^2}F_k^2,
    \qquad
    k=T,\ldots,N-1.
\]
If $F_N=0$, the desired bound is immediate. Otherwise,
\begin{equation*}
    \frac{1}{F_{k+1}}-\frac{1}{F_k}
    =
    \frac{F_k-F_{k+1}}{F_kF_{k+1}}
    \geq
    \frac{\theta}{2H_0R^2}
    \frac{F_k}{F_{k+1}}
    \geq
    \frac{\theta}{2H_0R^2}.
\end{equation*}
Summing from $k=T$ to $N-1$, we obtain
\begin{equation}
    \label{eq:GD_convex_rate_2}
    F_N
    \leq
    \frac{2H_0R^2}{\theta(N-T)}.
\end{equation}
Combining~\eqref{eq:GD_convex_rate_1_final} and
\eqref{eq:GD_convex_rate_2} gives the claim for $T<N$. If $T=N$, only the
linear phase occurs, and~\eqref{eq:GD_convex_linear_phase} gives
\[
    F_N
    \leq
    \left(
        1-\frac{\theta}{6H_1R^2}
    \right)^N
    F_0.
\]
\end{proof}

\subsubsection{Strongly convex setting}

\begin{boxA}
\begin{theorem}\label{th:GD_strongly_convex}
Let $f$ be $(H_0,H_1)$-smooth (Assumption~\ref{ass:H_0H_1_Smooth}) and
$\mu$-strongly convex (Assumption~\ref{ass:Strong_Convexity}). Then GD
(Algorithm~\ref{algo:GD}) needs at most
\begin{equation*}
    N
    =
    \Obound{
        \frac{H_0}{\theta\mu}
        \log\left(\frac{F_0}{\varepsilon}\right)
        +
        \frac{H_1F_0}{\theta\mu}
    }
\end{equation*}
iterations to reach accuracy $\func{x^N}-f^*\leq\varepsilon$, where
$\theta=\frac{1}{4\sqrt{2}+4}$.
\end{theorem}
\end{boxA}

\begin{proof}
Since a strongly convex function is also convex, we apply
Theorem~\ref{th:GD_Convex} with restarts. By strong convexity,
\begin{equation*}
    R^2
    =
    \norms{x^0-x^*}^2
    \overset{\eqref{eq:Strong_Convexity}}{\leq}
    \frac{2}{\mu}
    \left(\func{x^0}-f^*\right)
    =
    \frac{2F_0}{\mu}.
\end{equation*}
Therefore, the convex bound implies
\begin{align*}
    \func{x^N}-f^*
    &\leq
    \min\left\{
        \left(1-\frac{\theta}{6H_1R^2}\right)^T F_0,
        \frac{2H_0R^2}{\theta(N-T)}
    \right\} \\
    &\leq
    \min\left\{
        \exp\left(-\frac{\theta T}{6H_1R^2}\right)F_0,
        \frac{4H_0F_0}{\theta\mu(N-T)}
    \right\} \\
    &\leq
    \min\left\{
        \exp\left(-\frac{\theta\mu T}{12H_1F_0}\right)F_0,
        \frac{4H_0F_0}{\theta\mu(N-T)}
    \right\}.
\end{align*}
Thus, to guarantee $\func{x^N}-f^*\leq F_0/2$, it is sufficient to take
\[
    N
    \geq
    \frac{8H_0}{\theta\mu}
    +
    \frac{12H_1F_0}{\theta\mu}\log 2.
\]
Indeed, if
$T\geq \frac{12H_1F_0}{\theta\mu}\log 2$, the exponential term is at most
$F_0/2$; otherwise, $N-T\geq \frac{8H_0}{\theta\mu}$ and the sublinear term is
at most $F_0/2$.

Repeating this halving argument gives
\begin{equation*}
    N
    =
    \Obound{
        \frac{H_0}{\theta\mu}
        \log\left(\frac{F_0}{\varepsilon}\right)
        +
        \frac{H_1F_0}{\theta\mu}
    }.
\end{equation*}
\end{proof}

\begin{boxD}
\begin{remark}[Non-convex + PL]\label{rem:GD_PL}
If the Polyak-\L{}ojasiewicz (PL) inequality
\[
    \norms{\nabla \func{x}}^2
    \geq
    2\mu(\func{x}-f^*)
\]
holds for all $x\in\mathbb{R}^d$ with $\mu>0$, then the same recursion as in
Lemma~\ref{lem:GD_basic_descent_H0_H1} gives
\[
    F_{k+1}
    \leq
    F_k
    -
    \eta_k\mu F_k.
\]
Consequently, the same warm-up argument yields the complexity
\[
    \Obound{
        \frac{H_0}{\theta\mu}
        \log\left(\frac{F_0}{\varepsilon}\right)
        +
        \frac{H_1F_0}{\theta\mu}
    }
\]
iterations to reach $\func{x^N}-f^*\leq\varepsilon$.
\end{remark}
\end{boxD}